\documentclass[oneside]{amsart}
\usepackage{amssymb}
\usepackage{amsmath}
\usepackage{amsfonts}
\usepackage{geometry}
\usepackage{soul}
\usepackage{bbm}
\usepackage{comment}
\usepackage[colorlinks=true]{hyperref}
\usepackage{tikz}
\usepackage{mathrsfs}
\usepackage{multirow}
\usetikzlibrary{matrix,arrows,decorations.pathmorphing}
\usepackage{verbatim }
\usepackage{mathtools}
\usepackage{enumitem}
\usepackage[author={}]{pdfcomment}
\usepackage{todonotes}
\usepackage[normalem]{ulem}
\usepackage[T1]{fontenc}

\newcounter{cprop}[section]

\newtheorem{theorem}[cprop]{Theorem}
\theoremstyle{plain}

\newtheorem{lemma}[cprop]{Lemma}

\newtheorem{assumption}[cprop]{Assumption}

\numberwithin{equation}{section}
\usepackage{amsaddr}
\theoremstyle{definition}
\newtheorem{definition}[cprop]{Definition}
\newtheorem{example}[cprop]{Example}

\newtheorem{remark}[cprop]{Remark}

\def\cB{\mathcal{B}}

\def\cL{\mathcal{L}}

\def\R{\mathbb{R}}

\def\txtd{{\textnormal{d}}}
\makeatletter
\def\namedlabel#1#2{\begingroup
  \def\@currentlabel{#2}%
  \label{#1}\endgroup
}
\makeatother
\makeatletter
\def\@tocline#1#2#3#4#5#6#7{\relax
  \ifnum #1>\c@tocdepth
  \else
  \par \addpenalty\@secpenalty\addvspace{#2}%
  \begingroup \hyphenpenalty\@M
  \@ifempty{#4}{%
    \@tempdima\csname r@tocindent\number#1\endcsname\relax
  }{%
    \@tempdima#4\relax
  }%
  \parindent\z@ \leftskip#3\relax \advance\leftskip\@tempdima\relax
  \rightskip\@pnumwidth plus4em \parfillskip-\@pnumwidth
  #5\leavevmode\hskip-\@tempdima
  \ifcase #1
  \or\or \hskip 1em \or \hskip 2em \else \hskip 3em \fi%
  #6\nobreak\relax
  \hfill\hbox to\@pnumwidth{\@tocpagenum{#7}}\par
  \nobreak
  \endgroup
  \fi}
\makeatother

\def\txtd{{\textnormal{d}}}

\newcommand{\vertiii}[1]{{\left\vert\kern-0.25ex\left\vert\kern-0.25ex\left\vert #1
    \right\vert\kern-0.25ex\right\vert\kern-0.25ex\right\vert}}
\allowdisplaybreaks
\title{A semigroup approach for rough partial differential equations with jumps}

\author[A. Blessing Neam\c{t}u, H. Gao, and H. Li]{%
  Alexandra Blessing Neam\c{t}u$^{a}$,
  Hongjun Gao$^{b}$, and Haoyuan Li$^{b,*}$%
}

\address{%
  $^{a}$ Department of Mathematics and Statistics,
  University of Konstanz,\\
  Universit\"atsstra\ss{}e~10, 78464 Konstanz, Germany
}

\address{%
  $^{b}$ School of Mathematics, Southeast University,\\
  Nanjing 211189, PR China
}

\email{alexandra.blessing@uni-konstanz.de}
\email{hjgao@seu.edu.cn}
\email{230228445@seu.edu.cn}

\thanks{$^{*}$Corresponding author: Haoyuan Li.}
\date{}

\begin{document}

\maketitle

\begin{abstract}
 We establish an infinite-dimensional sewing lemma for discontinuous paths on a scale of Banach spaces. This allows us to construct Young and rough convolutions with jumps. Based on this we analyze the local and global well-posedness of semilinear parabolic rough partial differential equations (PDEs) with jumps in mixed space-time regularity spaces. 
\end{abstract}


\section{Introduction}

The main goal of this work is to investigate the well-posedness of non-autonomous parabolic rough partial differential equations of the form
	\begin{align} \label{main}
		\begin{cases}\txtd u_t = A(t)u_t~\txtd t + f(u_t)~\txtd \mathbf{X}_t, \\
			u_0\in \cB_\alpha,
		\end{cases}
	\end{align}
	on a family of Banach spaces $(\cB_\alpha)_{\alpha\in\R}$, where $\mathbf{X}$ is a rough path with finite $p$-variation for $p\in[2,3)$. The Young case, i.e. $p\in[1,2)$ is also included.~We consider the mild formulation
    \begin{align}\label{mild}
        u_t = S_{t,0} u_0 + \int_0^t S_{t,r} f(u_r) \txtd \mathbf{X}_r,
    \end{align}
where $S$ is the evolution family generated by $A(\cdot)$; see Section \ref{sec:preliminaries} for the functional analytic setting. For simplicity, we omit the drift term in \eqref{main}. This can easily be incorporated under a standard local Lipschitz continuity assumption.~For viscosity solutions for rough PDEs with jumps we refer to the recent work~\cite{BF},  which relies on a rough forward backward approach. \\

The theory of rough paths with jumps provides a pathwise framework for differential equations driven by discontinuous signals~\cite{FrizShekhar2017,FZ18}. Further developments include signature-based approximation of functionals of  c\`adl\`ag rough paths and L\'evy-type signature models~\cite{C1}, as well as functional It\^o formulas and Taylor expansions for non-anticipative functionals of  c\`adl\`ag rough paths \cite{C2}.~These developments concern signatures and functional calculus for discontinuous drivers.\\ 

On the other hand, rough PDEs with jumps~\cite{BF} arise naturally in systems that are perturbed by temporally 
irregular and discontinuous forcing terms. Several examples include 
fluid equations subject to intermittent L\'evy forcing, such as the 
two-dimensional stochastic Navier-Stokes equations studied in 
\cite{BrzezniakHausenblasZhu2013,DongXie2011}, as well as transport and fluid 
equations driven by signals that are irregular in time and are naturally 
treated using unbounded rough drivers 
\cite{BailleulGubinelli2017,HocquetHofmanova2018}. Reaction-diffusion and 
population models provide another important class of examples:~heavy-tailed 
jumps may induce transitions 
between metastable states that are different from transitions 
caused by Gaussian noise 
\cite{DebusscheHoegeleImkeller2011,DebusscheHoegeleImkeller2013}. Finally, 
rough equations with jumps arise naturally in multiscale systems with 
intermittent or heavy-tailed fast dynamics which have been explored in
\cite{Chevyrev2018,ChevyrevFriz2019,FrizShekhar2017,FZ18,LayritzPavlyukevichRammigKuehn2025}.

While mild solutions for jump-driven SPDEs \cite{BrzezniakHausenblasZhu2013,Hausenblas2005,MarinelliPrevotRockner2010,PeszatZabczyk2007} and RPDEs \cite{DeyaGubinelliTindel2012, GHN21, GerasimovicsHairer2019, GubinelliTindel2010, HBN19,HocquetNeamtu2024} driven by H\"older-continuous paths are both well developed, their unification requires the development of new techniques, in particular of an infinite-dimensional sewing lemma for $p$-variation paths which allows one to define the stochastic convolution \eqref{mild}. In particular, for continuous paths ${\bf X}$, their \(p\)-variation control becomes uniformly small on sufficiently short intervals. This local smallness is one of the main ingredients in the fixed-point argument used to prove local well-posedness of \eqref{main}. However, in the presence of jumps, this mechanism breaks down. A control associated with the \(p\)-variation of \({\bf X} \) may have atoms and therefore does not have to become small on intervals containing a jump. Consequently, estimates that are closed by choosing a sufficiently short time interval for H\"older-continuous paths, as investigated in~\cite{GHN21}, cannot be applied uniformly across jump times.~In particular, one has to distinguish intervals with small continuous rough variation from intervals containing a large jump. Since a finite \(p\)-variation path has at most finitely many jumps larger than any prescribed threshold, the large jumps can be isolated, whereas the accumulation of small jumps can be handled with rough-path estimates.~In this way, the theory of rough PDEs with jumps lies at the intersection of rough evolution equations driven by continuous paths and probabilistic SPDEs with L\'evy noise. \\

In this work, we consider the mild formulation which exploits the smoothing properties of the parabolic evolution family, while importing from c\`adl\`ag rough-path theory a pathwise  treatment of discontinuous rough drivers.~In the random-field setting \cite{Balan14,Balan15} developed solution theories for
SPDEs driven by $\alpha$-stable and spatially colored L\'evy noise, whereas in the infinite-dimensional evolution-equation
setting~\cite{KosmalaRiedle22} established existence and uniqueness of
mild solutions driven by cylindrical stable noise, under coefficient
assumptions involving fractional domains of the generator, similar to our setting.~Moreover,
\cite{Hausenblas2005} establishes existence, uniqueness and regularity
for parabolic SPDEs driven by Poisson random measures, while
\cite{BrzezniakHausenblas09} developed
maximal regularity estimates for L\'evy-driven stochastic
convolutions. Further maximal inequalities and exponential estimates
were obtained in~\cite{ZhuBrzezniakLiu19}, with applications to the existence and
uniqueness of mild solutions of stochastic quasi-geostrophic
equations. Complementing these semigroup-based approaches,~ \cite{BechtelGermVeraar26} developed an
extended critical variational framework for SPDEs with L\'evy noise.~In particular, our results are applicable to finite-dimensional L\'evy processes. As semimartingales,
they admit It\^o rough-path enhancements~\cite{FZ18} with almost surely finite
$p$-variation on compact time intervals for $p>2$. Choosing $p\in(2,3)$ therefore
allows an application of our deterministic rough PDEs results,
provided that the coefficients satisfy suitable spatial
regularity assumptions.~Identifying the resulting solutions with stochastic mild
solutions requires a corresponding identification of the rough
convolution integral with the It\^o stochastic convolution, which we postpone to a future work.~A related approach to discontinuous forcing is provided by Marcus
canonical equations and their rough-path formulation
\cite{ChevyrevFriz2019}.~In our setting, the rough convolutions 
 are defined through forward compensated Riemann sums.
An identification with a Marcus formulation would require an
additional treatment of the dynamics across jumps and is not a topic of this work. 
\\

Our results build on the discontinuous rough-path framework \cite{FZ18} and the theory of non-autonomous rough PDEs developed by \cite{GHN21}. Due to the reasons mentioned above, combining these frameworks requires additional tools. The finite-dimensional theory \cite{FZ18} does not address the interaction between discontinuous rough integration and parabolic evolution operator acting on a monotone scale of interpolation spaces.~Conversely, the continuous rough-PDE theory~\cite{GHN21} exploits temporal regularity to control convolution errors and obtain spatial smoothing.~The smoothing effect of the parabolic evolution family cannot compensate a jump time, while the associated variation controls of the rough input are in general not small on short intervals containing the jump. We therefore develop a multiplicative sewing argument with nonregular controls in an infinite-dimensional setting and establish convolution estimates compatible with both the spatial scale and the jump structure.~These estimates provide the basis for the subsequent well-posedness analysis.\\



    One of the main contributions of this work is to introduce an infinite-dimensional sewing lemma handling discontinuous paths. At an informal level, this reads as follows, see Theorem \ref{thm:mul01} for more details.
    \begin{theorem}{\em(Informal statement of the sewing lemma)}
        Let $(\mathcal{M}, |\cdot|, d)$ be a submultiplicative monoid and let $\mu$ be a two-time parameter $\mathcal{M}$-valued function. Assume that there exist $N\geq 1$, controls $\omega_{1,j}, \omega_{2,j}$ and exponents $\alpha_{1,j},\alpha_{2,j}>0$, $\alpha_{1,j}+\alpha_{2,j}>1$, for all $j=1,\dots,N$  such that 
  \[  d(\mu_{t,s}, \mu_{t,u}\circ \mu_{u,s})\lesssim  \sum_{j=1}^{N} \omega_{1,j}^{\alpha_{1,j}}(t,u) \omega_{2,j}^{\alpha_{2,j}}(u,s).\]
	Then there exists a unique multiplicative map $\varphi$ such that for every $0\leq s \leq t \leq T$, we have  
\begin{align*}
&\varphi_{t,s}:=\lim\limits_{|\mathcal{P}|\to 0}\mu_{t,s}^{\mathcal{P}},
\end{align*}
in an appropriate sense, where $\mu^{\mathcal{P}}_{t,s} = \prod\limits_{i=0}^{k-1} \mu_{t_{i+1},t_i}$ for a partition $\mathcal{P}$ of the corresponding interval. Moreover,
	\begin{align*}
d(\varphi_{t,s}, \mu_{t,s})\lesssim  \sum_{j=1}^{N} \omega_{1,j}^{\alpha_{1,j}}(t,s+) \omega_{2,j}^{\alpha_{2,j}}(t-,s).
\end{align*}

    \end{theorem}

    Based on this we can construct rough convolutions for discontinuous processes of $p$-variation for $p\in[2,3)$ and investigate the local and global well-posedness of \eqref{main}.~The treatment of jumps in sewing arguments has been developed in
both deterministic~\cite{FZ18} and stochastic settings~\cite{AllanPieper26}.~In particular~\cite{AllanPieper26} developed a stochastic sewing
lemma accommodating multiple discontinuous controls, which is
used to construct rough stochastic integrals and solve rough
stochastic differential equations with jumps.~The additional difficulty in the present work arises from the
parabolic evolution operators. The convolution sums involve an
evolution family generated by unbounded operators, and their
defects must be estimated across different levels of a spatial
Banach scale. Besides the usual controlled-path terms, these
defects contain propagator increments of the form
$(S_{u,s}-\mathrm{id})\xi_{t,u}$, for an approximation term $\xi$, whose estimates require a trade-off
between temporal and spatial regularity. At jump times,
one can not rely on the
smoothing estimates of the parabolic evolution family, as provided in~\eqref{smooth}. We address these issues through a deterministic
multiplicative sewing lemma with nonregular controls.\\

    {\em Structure of the paper.}  This manuscript is structured as follows. In Section \ref{sec:preliminaries} we describe the functional analytic background together with important properties of $p$-variation paths and controls.  Section \ref{sec:Sewing} establishes the multiplicative and affine sewing lemmas for discontinuous paths and constructs the Young integral. Section \ref{sec:wellpose} is devoted to the existence and uniqueness of local mild solutions to Young PDEs, as stated in Theorem \ref{local-2p}.~To this aim, we introduce a family of mixed space-variation regularity spaces encoding different space and time regularity required in the pathwise well-posedness argument. The local solution is extended to a global one in Theorem \ref{Glo-young}, under a spatial regularity assumption of the initial datum. Stability results are specified in Lemma \ref{cadlagGLOBAL}. Since we cannot rely on the smoothing effect of the parabolic evolution operators, we impose an additional spatial regularity assumption on the diffusion coefficient $f$. For rough PDEs with H\"older continuous drivers, an improvement of the spatial regularity of the diffusion coefficient was assumed in \cite{HBN19, MJ, NeamtuSeitz23}.   Section \ref{sec:sec:prough} treats the rough regime, i.e. paths of finite  $p$-variation  with $p\in[2,3)$. The local well-posedness of \eqref{main} is established in Theorem \ref{local-rough-jump}, whereas the existence of a global solution is postponed to a forthcoming work.  
    It is known that global well-posedness of rough PDEs based on the mild formulation is a challenging task \cite{HesseNeamtu2022}. 
    Due to the fact that we cannot rely on the smoothing property of parabolic evolution operators, we assume that the diffusion coefficient $f$ gains spatial regularity. 
       Finally, Section \ref{sec:example} contains two examples of diffusion coefficients which satisfy the assumption of Theorem \ref{Glo-young} and \ref{local-rough-jump}. We believe that our results can also cover SPDEs with boundary jump noise~\cite{HartmannPavlyukevich19} based on the techniques developed in~\cite{NeamtuSeitz23} in the continuous case.~Moreover, our global well-posedness result in the Young regime provides a
basis for constructing random dynamical systems provided that $A$ is time-independent, following the approaches in~\cite{BailleulRiedelScheutzow17, CG26,CG25,HesseNeamtu2022,KB23,YLZ23}.\\


\section{Preliminaries}\label{sec:preliminaries}
Throughout the paper, we fix a finite time horizon   $T>0$. We write $a\lesssim b$ if there exists a constant $C>0$, depending only
on fixed parameters, such that $a\leq Cb$.  The value of $C$ may change from line to line.
If $\mathcal E_1$
 and $\mathcal E_2$ are Banach spaces, we denote by $\cL(\mathcal{E}_1, \mathcal{E}_2)$ the space of continuous linear mapping from $\mathcal{E}_1$ to $\mathcal{E}_2$,  equipped with the operator norm $\|\cdot \|_{\cL(\mathcal{E}_1, 
 \mathcal{E}_2)}$. We write $\mathcal{L}(\mathcal{E}_1)=\mathcal{L}(\mathcal{E}_1, \mathcal{E}_1)$. \(\mathcal L_s(\mathcal{E}_1)\) denotes \(\mathcal L(\mathcal{E}_1)\) endowed
with the strong operator topology. 
\textcolor{black}{
 We write $C^k = C^k(\mathcal{E}_1, \mathcal{L}(\mathbb{R}^d, \mathcal{E}_2))$ for the space of $k$-times continuously Fr{\'e}chet differentiable  maps $f \colon \mathcal{E}_1 \to \mathcal{L}(\mathbb{R}^d, \mathcal{E}_2)$ such that $D^i f$ sends bounded subsets of $\mathcal E_1$ to bounded sets of $\mathcal E_2$, for all $i=0,\dots, k$. For $R>0$, 
 \begin{align*}
\|f\|_{C^k}:=
\sum_{i=0}^{k}
\sup_{|u|_{\mathcal E_1}\leq R}
\|D^if(u)\|_
{\mathcal L^i(\mathcal E_1;\mathcal L(\mathbb R^d,\mathcal E_2))}.
 \end{align*} 
Here $D^0 f=f$,  $\mathcal L^i(\mathcal E_1;F)$ denotes the space of bounded $i$-linear maps from $\mathcal E_1^i$ to $F$. 
We write $C_b^k(\mathcal E_1, \mathcal L(\mathbb R^d, \mathcal E_2))$ for the subspace of $C^k$, where $f$ and its derivatives up to order $k$ are uniformly bounded.   We equip this space with the norm 
	\begin{equation*}
		\|f\|_{C^k_b} := \sum_{i=0}^k\|D^if\|_\infty,
	\end{equation*}
	where $\|\hspace{0.5pt}\cdot\hspace{0.5pt}\|_{\infty}$ denotes the supremum norm on the corresponding spaces of operators. }
For $n\geq 2,$ define the simplex $\Delta_n=\{T\geq t_n\geq t_{n-1}\geq \dots \geq t_1\geq 0\}.$ Define $C_n ([0,T]; \mathcal{E}_1)$ as the space of continuous functions from $\Delta_n$ to $\mathcal{E}_1$. 
\subsection{Functional analytic tools}
We begin by introducing the notion of a monotone family of interpolation spaces, which provides  the spatial regularity associated with \eqref{main}.
\begin{definition}
A family of separable Banach spaces $(\cB_\alpha, |\cdot|_\alpha)_{\alpha\in \mathbb{R}}$ is called a monotone family of interpolation spaces if for every $\alpha<\beta$, $\cB_\beta$ is a continuously embedded, dense subspace of $\cB_\alpha$, and if the following interpolation inequality holds: for $\alpha\leq \beta\leq \gamma$ and $x\in \cB_\alpha \cap \cB_\gamma$:
$$
|x|_\beta^{\gamma-\alpha}\lesssim |x|_\alpha^{\gamma-\beta}|x|_\gamma^{\beta-\alpha}.
$$
\end{definition}
Next, we state the assumptions on the family $(A(t))_{t\in [0,T]}$. Fix a number $\vartheta>0$, $\lambda \in \R$ and define a sector $\Sigma_{\vartheta,\lambda}$ of the complex plane as follows:
	\begin{equation}\label{nota:sigma}
		\Sigma_{\vartheta,\lambda}:=\{\zeta \in \mathbb C,\, |\arg(\zeta -\lambda) | < \pi/2+\vartheta\}.
	\end{equation}
We use the convention $\arg 0=0$ throughout this paper, so that, in particular, $[\lambda, \infty) \subset \Sigma_{\vartheta,\lambda}$.
We further assume that  $(\mathcal E_1,|\cdot|_1)$  and  ($\mathcal E_0,|\cdot|_0)$ are Banach spaces satisfying
$\mathcal E_1\hookrightarrow \mathcal E_0$, where the embedding is continuous and dense.
\begin{assumption}\label{ass}
		Let $(A(t))_{t\in[0,T]}$ be a family of closed, densely defined linear operators on $\mathcal E_0$ with domains containing $\mathcal E_1$. Assume that there exist constants $\Lambda,M>0$ (depending only on $T$), control $\omega$, $\varrho \in (0,1]$ such that the following conditions hold.
		\begin{enumerate}[label=(A\arabic*)]
			\item\label{A1} For each $t\in[0,T]$, the resolvent set $\rho (A(t))$ contains $\Sigma_{\vartheta,\lambda}$, and for $i =0,1,$
			\begin{equation}\label{resolvent}
				\|(\zeta  - A(t))^{-1}\|_{\cL(\mathcal E_i)}\leq \Lambda\big(1+|\zeta |\big)^{-1},\quad \forall\, \zeta \in\Sigma_{\vartheta,\lambda}.
			\end{equation}
			
			\item\label{A2}
			For any $t\in[0,T],$ we have
			\[
			\|(\zeta -A(t))^{-1}\|_{\cL(\mathcal E_0,\mathcal E_1)}\leq M,\quad \forall\, \zeta \in\Sigma_{\vartheta,\lambda}.
			\]
			
			\item\label{A3} For all $(t,s) \in \Delta_2$:
			\begin{align*}
				\|A(t)  - A(s)\|_{\cL(\mathcal E_1,\mathcal E_0)} \leq \omega^\varrho(t,s).
			\end{align*}
			
			\item\label{A4} The above control satisfies the following integrability property: for all $(t,s)\in\Delta _2$
			\begin{align*}
				\int_s^t \frac{\omega^\varrho(r,s) dr}{r-s} < \infty.
			\end{align*}
		\end{enumerate}
        Note that, among the above assumptions, we require that  $\omega^\varrho$ is a continuous map. It satisfies  $\omega^\varrho(s,s)=0$, for all $s\in [0, T]$ and  $\omega^\varrho(t,r)+\omega^\varrho(r,s)\leq \omega^\varrho(t,s)$, for all $0\leq s\leq r\leq t\leq T$. 
	\end{assumption}
     	\begin{definition}
		\label{propagator}
		We say that $S\colon \Delta _2\to \cL(\mathcal E_0)$ is a propagator associated to the family $(A(t))_{t\in[0,T]}$ of unbounded, closed operators with dense domains containing $\mathcal E_1$ if the following properties hold.
		\begin{enumerate}[label=(P\arabic*)]
			\item\label{P1} $S\in C_2([0,T];\cL_s(\mathcal E_0))$ and there exist constants $\lambda ,\Lambda>0$ such that for every $(t,s) \in \Delta_2$
			\begin{equation}
				\label{growth_S}
				\|S_{t,s}\|_{\cL(\mathcal E_0)}, \|S_{t,s}\|_{\cL(\mathcal E_1)} \leq \Lambda e^{\lambda(t-s)}\;.
			\end{equation}
			
			\item\label{P2} $S_{t,t}=\textup{id}$ and $S_{t,s}=S_{t,u }S_{u ,s}$ for $(t,u ,s)\in\Delta _3$. 
			
			\item\label{P3} There exists $C_T>0$ such that for every $(t,s)\in \Delta_2,$ $s\neq t$:
			\[
			\|S_{t,s}-\textup{id}\|_{\cL(\mathcal E_1,\mathcal E_0)} \leq C_T |t-s|\,.
			\]
			\item\label{P4} For all $s,t \in [0,T]$ and $x \in \mathcal E_1$ we have:
			$$
			\frac{d}{dt}S_{t,s}x = A(t)S_{t,s}x \quad  \text{and}  \quad \frac{d}{ds}S_{t,s}x=-S_{t,s} A(s)x\,,$$
			where the differentiation is taking place in the Banach space $\mathcal E_0$.
			\item\label{P5} For every $(t,s)\in \Delta_2,$ $s\neq t$ we have that $S_{t,s} \mathcal E_0 \subset \mathcal E_1$ and for some constant $N_T>0$
			\begin{equation}
				\label{smoothing_S}
				\|A(t) S_{t,s}\|_{\cL(\mathcal E_0)} \lesssim \|S_{t,s}\|_{\cL(\mathcal E_0,\mathcal E_1)} \leq  N_{T}|t-s|^{-1}.
			\end{equation} 
		\end{enumerate}
	\end{definition}
	
	\begin{definition}\label{fullrange}
		Let $(\cB_\alpha )_{\alpha \in \mathbb R}$ be a monotone family of interpolation spaces, and fix an interval $I\subset\mathbb R.$
		We say that $S$ is a propagator on the full range $(\cB_{\alpha })_{\alpha \in I}$ if for every $\alpha \in I,$ $S$ restricted to $\cB_\alpha $ is itself a propagator, in the sense of Definition \ref{propagator} for the pair $(\mathcal E_0,\mathcal E_1)=(\cB_\alpha ,\cB_{\alpha +1})$.
	\end{definition}
\begin{remark}\label{dir}
In this paper, we fix a family of unbounded operators $(A(t))_{t\in [0,T]}$ that acts on the monotone family $(\cB_\alpha )_{\alpha \in \mathbb R}$, such that they generate the propagator $(S_{t,s})_{(t,s)\in \Delta_2}$  on the full range in the sense of  Definition \ref{fullrange}. We have the following two properties which will be used frequently, 

(1) Continuity. $S\in C_2([0,T]; \cL_s (\cB_\alpha))$. For $x\in \cB_\alpha,$ $\lim\limits_{u\uparrow s}S_{t,u}x=S_{t,s}x$, $\lim\limits_{u\downarrow t}S_{u,s}x=S_{t,s}x$. Especially, we denote $S_{t,t-}x:=\lim\limits_{s\uparrow t}S_{t,s}x=x, S_{t+,t}x:=\lim\limits_{s\downarrow t}S_{s,t}x=x$. 
 
(2) Smoothing inequalities. $S_{t,s}$ maps $\cB_\alpha$ to $\cB_\beta$, for all admissible
indices $\alpha,\beta$,  $\beta\geq \alpha$,  $\sigma \in [0,1]$, 
and $0\leq s<t\leq T$, one has
\begin{align}\label{smooth}
|(S_{t,s}-\textup{id})x|_\alpha\lesssim |t-s|^{\sigma}|x|_{\alpha+\sigma}, \quad |S_{t,s}x|_{\beta}\lesssim |t-s|^{-(\beta-\alpha)}|x|_\alpha.
\end{align}
\end{remark}
\subsection{\texorpdfstring{$P$}{P}-variation spaces and controls}
For $n\in \mathbb{N}$, we define a partition of $[0,T]$ by $
\mathcal{P}
:=
\{0=t_{0}<t_{1}<\cdots<t_{n}=T\}$,  and define the mesh size  by $|\mathcal{P}|:=\max\limits_{i}|t_{i+1}-t_{i}|$.   We write $(u,v)\in \mathcal{P}$ if $u$ and $v$ are consecutive points of
the partition $\mathcal P$.
A path is a function from $[0,T]$ (or any other interval clear from the context) into a normed space  $(\mathcal{E},~|\cdot|)$. 
Such a path is called  regulated if both left and right limits exist for all points and we write $X \in V^\infty( [0,T];\mathcal{E})$ accordingly,  endowed with the semi  norm 
$\|X\|_{V^{\infty}([0,T];\mathcal{E})}:=\sup\limits_{t\in [0,T]}|X_t-X_0|$. 
Denote
increment operator $\delta$ for $f: [0,T]\rightarrow \mathcal{E}$, $g: \Delta_2\rightarrow \mathcal{E}$ as 
\begin{align*}
 \delta f_{t,s}:=f_{t}-f_{s}, \quad \text{for} \quad (t,s)\in \Delta_{2}.
\end{align*}
\begin{align*}
\delta g_{t,u,s}:=g_{t,s}-g_{t,u}-g_{u,s}, \quad \text{for}\quad (t,u,s)\in \Delta_{3}.
\end{align*}
Beside, we denote the mild increment operator by $\delta^{S}$, $\delta^S f_{t,s}:= f_t -S_{t,s}f_s.$

For $p\in (0,\infty)$, $f: [0,T]\rightarrow \mathcal{E}$, we say $f\in V^{p}([0,T], \mathcal{E})$, which has finite $p$-variation, if
\begin{align*}
\|f\|_{V^{p}([0,T]; \mathcal{E})}:=(\sup\limits_{\mathcal{P}}\sum\limits_{[s,t]\in \mathcal{P}}|\delta f_{t,s}|^{p})^{\frac{1}{p}}<\infty.
\end{align*}
One readily checks that $V^1 \subset V^p \subset  V^q \subset V^\infty$ for $1\leq p\leq q<\infty$.
Set
$$
 Osc(X;[0,T]) := \sup\limits_{0 \leq s \leq t \leq T}  |X_{t}-X_{s}|.
$$
We further introduce the following notations: 
$\Delta_t^- X:=\delta X_{t,t-}:=\lim\limits_{s\uparrow t} \delta X_{t,s}$, $\Delta_t^+ X:=\delta X_{t+,t}:=\lim\limits_{s\downarrow t}\delta X_{s,t}$. 
\textcolor{black}{
We also define $p$-variation for a two-time parameter function $f: \Delta_2 \to \mathcal E$
\begin{align*}
\|f\|_{V^p([0,T];\mathcal E)}:=(\sup\limits_{\mathcal{P}}\sum\limits_{[s,t]\in \mathcal{P}}|\delta f_{t,s}|^{p})^{\frac{1}{p}}<\infty.
\end{align*}}
\begin{definition}(\textbf{Control})
A  control is a function $\omega = \omega(t, s)$ from $\Delta_{2}:=\{ (t,s):T\geq t \geq s \geq 0\}$ into $[0,\infty)$,
which is null on the diagonal and super-additive in the sense $\omega(t,u)+\omega(u,s) \leq \omega(t,s)$.
\end{definition}

\begin{remark}\label{reg-control}
A control $\omega$ is called regular if it is continuous at the diagonal, i.e. 
$ \lim\limits_{\substack{(t,s)\in\Delta_2\\ |t-s|\to0}}
    \omega(t,s)=0.$ A simple example of regular control is given by $(s,t)\to |t-s|^\theta$, for $\theta\geq 1$, see \cite[Section 1.2]{FV10} for more details.
\end{remark}
\begin{remark}
(1) Let $X\in V^p([0,T];\mathcal E)$. Then
$\omega_{X,p}(t,s)
    :=
    \|X\|_{V^p([s,t];\mathcal E)}^p,
 0\le s\le t\le T$ defines a control.
 
 
 (2) If $X$ is continuous, right-continuous or left-continuous, then control $\omega_{X,p}$ shares similar continuities.  For example, if $X$ is right-continuous at $s$, then $\omega_{X,p}(s+,s):=\lim\limits_{t\downarrow s}\omega_{X,p}(t,s)=0$. 
\begin{proof}
(1)    For the convenience of the reader we indicate a short proof. We observe that $\omega_{X,p}(s,s)=0$ at first. To prove super-additivity, let $0\le s\le u\le t\le T$ and let $\mathcal P_1$ and $\mathcal P_2$ be arbitrary finite partitions of
$[s,u]$ and $[u,t]$, respectively.  Then $\mathcal{P}:=\mathcal{P}_{1}\cup \mathcal{P}_{2}$ is a finite partition of $[s,t]$.
Hence, by the definition of
$\omega_{X,p}(t,s)$,
 $$\sum_{[a,b]\in \mathcal{P}_{1}}|\delta X_{b,a}|^{p}+\sum_{[a,b]\in \mathcal{P}_{2}}|\delta X_{b,a}|^{p}\leq {\omega}_{X,p}(t,s).$$
Taking the supremum over all $\mathcal{P}_{1}$ and $\mathcal{P}_{2}$ yields  the super-additivity of $\omega_{X,p}$, i.e.~$\omega_{X,p}(t,u)+\omega_{X,p}(u,s)\leq \omega_{X,p}(t,s)$.

(2) The proof is similar to that of \cite[Lemma 7.1]{FZ18} and is omitted.
\end{proof} \end{remark}
 
\begin{lemma}\label{product-control}
Let $\omega_1$ and $\omega_2$ be controls on $[0,T]$, and let
$a,b>0$ satisfy $a+b\geq1$. Then
\[
\omega(t,s)
:=
\omega_1^a(t,s)\omega_2^b(t,s),
\qquad (t,s)\in\Delta_2,
\]
is also a control.
\end{lemma}

\begin{proof}
Since $\omega_1$ and $\omega_2$ vanish on the diagonal, we have
\[
\omega(s,s)
=
\omega_1^a(s,s)\omega_2^b(s,s)
=0.
\]

It remains to prove the super-additivity. Fix
$0\leq s\leq u\leq t\leq T$. 
By the super-additivity of $\omega_1$ and $\omega_2$,
\[
\omega_1(t,s)\geq \omega_1(t,u)+\omega_1(u,s),
\qquad
\omega_2(t,s)\geq \omega_2(t,u)+\omega_2(u,s).
\]
Therefore,
\[
\omega(t,s)
\geq
( \omega_1(t,u)+\omega_1(u,s))^a(\omega_2(t,u)+\omega_2(u,s))^b.
\]
Let $r:=a+b\geq 1$.  By H\"older's inequality with conjugate
exponents $r/a$ and $r/b$, we have
\begin{align*}
&\omega_1 ^{a/r}(t,u)\omega_2 ^{b/r}(t,u)
+
\omega_1 ^{a/r}(u,s)\omega_2 ^{b/r}(u,s)
\\
 \leq &~~
\bigl(\omega_1(t,u)+\omega_1(u,s)\bigr)^{a/r}
\bigl(\omega_2(t,u)+\omega_2(u,s)\bigr)^{b/r}.
\end{align*}
Since $r\geq 1$, 
we obtain
\begin{align*}
&(\omega_1^{a/r}(t,u)\omega_2^{b/r}(t,u))^r+ (\omega_1^{a/r}(u,s)\omega_2^{b/r}(u,s))^r \\
\leq & (\omega_1^{a/r}(t,u)\omega_2^{b/r}(t,u)+\omega_1^{a/r}(u,s)\omega_2^{b/r}(u,s))^r\\
\leq & \bigl(\omega_1(t,u)+\omega_1(u,s)\bigr)^{a}
\bigl(\omega_2(t,u)+\omega_2(u,s)\bigr)^{b}.
\end{align*}
This implies
\begin{align*}
\omega(t,u)+\omega(u,s)
&=\omega_1 ^a (t,u) \omega_2  ^b(t,u)+\omega_1 ^a(u,s)\omega_2 ^b (u,s)\\
&\leq
\bigl(\omega_1(t,u)+\omega_1(u,s)\bigr)^{a}
\bigl(\omega_2(t,u)+\omega_2(u,s)\bigr)^{b}\\
&\leq
\omega_1^a(t,s)\omega_2^b(t,s)\\
&=
\omega(t,s).
\end{align*}
Thus $\omega$ is null on the diagonal and super-additive, and hence
it is a control.
\end{proof}
Note that  every control is non-decreasing in
its first argument and non-increasing in its second argument. In particular,
 for any $(t,s)\in \Delta_2$,  the limits 
\begin{equation*}
\omega(t-,s):=
\begin{cases}
\lim\limits_{r\uparrow t}\omega(r,s), &\text{if}\quad s<t,\\
0&\text{if} \quad s=t,
\end{cases}
\\ \quad 
\omega(t, s+):= 
\begin{cases}
\lim \limits_{r\downarrow s} \omega (t,r), &\text{if}\quad  s<t,\\
0& \text{if}\quad s=t,
\end{cases}
\end{equation*}
are well-defined controls. Indeed,   for $0<s<u<r<t$, the super-additivity of $\omega$ gives $\omega(r,u)+\omega(u,s)\leq \omega(r,s)$. Since $\omega(u-,s)<\omega(u,s)$, we have $\omega(r,u)+\omega(u-,s)\leq \omega(r,s)$. Letting $r\uparrow t$ yields  $\omega(t-,u)+\omega(u-,s)\leq \omega(t-,s)$. The proof of $\omega(t, s+)$ is similar to that of $\omega(t-,s).$
\begin{definition}(\textbf{Mild Control})
	We call a function $\sigma:\Delta_2\to \R_+$, where  $\Delta_{2}:=\{ (t,s):T\geq t \geq s \geq 0\}$, a mild control, if it is increasing in $t,$ decreasing in $s,$ null on the diagonal, and for any $\delta>0,$ there exist only finite $u\in[0,T]$ such that
	$$
	\sigma(u+,u-):= \lim_{t\downarrow u , s \uparrow u} \sigma(t,s) > \delta.
	$$
	We call $\sigma$ a mild control for some path $x$ if $d(x_t, x_s) \leq \sigma(t,s)$ for any $s,t\in [0,T].$ 
\end{definition}
\begin{remark}\label{lim-control}
 Every control is a mild control; see
\cite[Lemma 1.4]{FZ18}. 
 \end{remark}

\begin{lemma}\cite[Lemma 1.5]{FZ18}
Suppose $\sigma(t,s)$ is a mild control. Then for any $\varepsilon>0$, there exists a partition $\mathcal{P}$, such that for any $(s,t)\in \mathcal{P}$, one has 
\begin{align*}
\sigma(t-, s+)<\varepsilon.
\end{align*}
In particular, if $\sigma(t,s)$ is left-continuous in $t$, then $\sigma(t,s+)<\varepsilon$. 
\end{lemma}
\begin{lemma}\cite[Lemma 1.6]{FZ18}\label{small-OSC}
Let  $x:[0,T]\to (\mathcal{E},|\cdot|)$ be a regulated path. Then, for every
$\varepsilon>0$, there exists a finite partition $\mathcal{P}$, such that for any interval $(s,t)\in \mathcal{P}$, 
$Osc(x, (s,t))=\sup\limits_{u,v\in (s,t)} |x_u-x_v|<\varepsilon$. 
If $x$ is right-continuous, then for every $\varepsilon>0$, there exists $\delta>0$ such that $r\in [s,s+\delta]$, 
\[
Osc(x; [s,s+\delta])=\sup \limits_{u,v\in [s, s+\delta]}|x_u-x_v|<\varepsilon. 
\]
\end{lemma}

\begin{lemma}\label{partition}\cite[Lemma 2.1]{FZ18}
	Suppose $\omega_i(t,s),i=1,2$ are controls. Then for any $\varepsilon>0,$ there exists a partition $\mathcal{P}$ such that for any $(s,t)\in \mathcal{P},$
	$$
	\omega_1(t, s+)  \wedge 	\omega_2(t-, s) < \varepsilon.
	$$
 In particular, for any $s\leq \tau \leq t$ one has
	$$\omega_1(t, \tau)\omega_2(\tau,s)<(\omega_1(t, s)\vee \omega_2(t, s))\varepsilon < (\omega_1(T, 0)\vee \omega_2(T, 0))\varepsilon .$$
\end{lemma}
\section{Sewing lemma with jumps}\label{sec:Sewing}
This section extends the multiplicative and affine sewing lemmas to a discontinuous setting in infinite dimensional spaces. These results 
provide tools for defining integrals and constructing solutions to PDEs dirven by jump processes with $p$-variation. In contrast to the continuous case,
small mesh size alone does not guarantee that the local
multiplicative defect is small, since an arbitrarily short interval
may still cross a jump. Here, we  address this  difficulty by applying Lemma \ref{partition} to  choose suitable partitions for the nonregular controls.
\subsection{Multiplicative sewing lemma with jumps}

Denote a monoid by $(\mathcal{M}, \circ)$. For simplicity, we only use the notation $\mathcal{M}$. 
For a function $\mu \colon \Delta_2 \to \mathcal{M}$ and an arbitrary partition $\mathcal{P}= \{s =t_0 < t_1 < \dots < t_k = t\}$ of $[s,t]$ set: 
\begin{align}\label{promon}
	\mu^{\mathcal{P}}_{t,s} = \prod_{i=0}^{k-1} \mu_{t_{i+1},t_i}.
\end{align}
Here, we use the increment notation $\mu_{t,s}=\mu_{t}-\mu_{s}$.
\begin{definition}[\textbf{Submultiplicative monoid}]\label{monoid01} Let $(\mathcal{M}, \circ)$ be a monoid.
	We call a triple $(\mathcal{M},
    |\cdot|, d)$ a submultiplicative monoid if $d$ is a metric on $\mathcal{M}$ and a function $|\cdot| : \mathcal{M} \to \mathbb{R}^+$  satisfies  
	\begin{equation}
		\label{subddd}
		d(ac,bc) \leq d(a,b) |c|,\quad\text{and}\quad d(ca,cb) \leq  |c| d(a,b),
	\end{equation}
    for all $a,b,c \in \mathcal{M}$.
	Moreover, we assume that for all $C \in \mathbb{R}^+$ the sets $B_C := \{a \in \mathcal{M}:\, |a| \leq C \}$ are complete with respect to the metric $d$.
\end{definition}
\begin{remark}\label{mulconti}
		Let $\mathcal{M}$ be a submultiplicative monoid.
		As a consequence of \eqref{subddd}, the multiplication is continuous with respect to $d$ on the sets $B_C$.
		Indeed if $|b_n| \leq C$ and $d(a_n,a) \to 0$ and $d(b_n,b) \to 0$ as $n \to \infty$ then:
		\begin{align*}
			d(a_nb_n, ab) \leq d(a_n,a)|b_n| + |a| d(b_n,b) \leq C d(a_n,a) + |a| d(b_n,b) \to 0,
		\end{align*}
		as $n \to \infty$. The same is true if instead of $|b_n| \leq C$ we have $|a_n| \leq C$.
\end{remark}
Motivated by \cite[Definition~1.1]{FZ18}, we introduce two notions of
convergence for the partition products defined in \eqref{promon},
extending mesh and refinement convergence from Riemann--Stieltjes
sums to monoid-valued products. 
\begin{definition}(\textbf{Mesh and refinement convergence})\label{MSRS}
	For a partition $\mathcal{P}$ of $[0,T]$ and any $[s,t]\in \mathcal{P},$  we let $\mu_{t,s}  $ take values on a submultiplicative monoid $\mathcal{M}.$
	We say that the product $\mu^{\mathcal{P}}_{t,s}$ converges to $\mathcal{U}$ as follows
	\begin{itemize}
		\item Mesh sense (MS): if for any $\varepsilon>0,$ there exists $\delta>0,$ such that any $ |\mathcal{P}|<\delta, $ one has $$d(\mu_{t,s}^{\mathcal{P}}, ~ \mathcal{U})<\varepsilon.$$
		
		\item Refinement sense (RS): if for any $\varepsilon>0,$ there exists $\mathcal{P}^\varepsilon$ such that for any refinement $ \mathcal{P} \supset \mathcal{P}^\varepsilon, $ one has $$
	d(\mu^{\mathcal{P}}_{t,s}, ~ \mathcal{U})<\varepsilon.$$
	\end{itemize}	
\end{definition}
For example, suppose that  $\mu_{t,s}^{\mathcal{P}}=(S_{t,s}, \sum\limits_{(u,v)\in \mathcal{P}} S_{t,u}\xi_{v,u})$ and denote the limit by $\mathcal{U}=(S_{t,s}, \mathcal{K})$. For the second component, the two notions of convergence  above  can be otherwise known as Mesh Riemann-Stieltjes (MRS) convergence and Refinement Riemann-Stieltjes (RRS) convergence, respectively. Please see Remark \ref{MRSRRS} for more details. 
	\begin{definition}\label{def:monoid}
	Let $\mathcal{M}$ be a submultiplicative monoid.
	We say that a function $\mu :\Delta_2\to \mathcal{M}$ is 
	\begin{enumerate}[label=(\roman*)]
		\item \emph{multiplicative} if for every $(t,u ,s)\in\Delta_3$:
		\begin{equation*}
			\mu _{t,s} = \mu _{t,u }\mu _{u, s}.
		\end{equation*}
		\item  \emph{almost-multiplicative}  if there exists  controls $\omega_{1,j}(t,s), \omega_{2,j}(t,s)$  and positive constants $\alpha_{1,j}, \alpha_{2,j}$, satisfying $\alpha_{1,j}+\alpha_{2,j}>1$ for all $j=1,...,N,$ such that for each $(t,u ,s)\in\Delta _3:$
\begin{align}	\label{almostmul01}
	d(\mu_{t,s}, \mu_{t,u}\circ \mu_{u,s})\lesssim  \sum_{j=1}^{N} \omega_{1,j}^{\alpha_{1,j}}(t,u) \omega_{2,j}^{\alpha_{2,j}}(u,s).
\end{align}
\end{enumerate}
Let $\sigma: \Delta_{2}\rightarrow \mathbb{R}^{+}$ be a bounded function, inscreasing in $t$ and decreasing in $s$. We say that 
	\begin{enumerate}[label=(\roman*)]
		\setcounter{enumi}{2}
		\item  \label{submult:2}
        $\mu$ has growth bound with $\sigma$ if
		\begin{equation}\label{growth01}
			|\mu^{\mathcal{P}}_{t,s}| \lesssim  \sigma(t,s),
		\end{equation}
		for every $(t,s) \in \Delta_2,$ independently of the choice of partition $\mathcal{P}$ of $[s,t]$. 
	\end{enumerate}
	We denote by $\mathrm{BG}_2([0,T], \mathcal{M})$ the set of all functions $\mu: \Delta_2\to \mathcal{M}$ adimtting a growth bound $\sigma$ as above.
\end{definition}
\begin{remark}
In \cite[Theorem 3.4]{GHN21}, a multiplicative sewing lemma was proved under the assumption $d(\mu_{t,s}, \mu_{t,u}\circ \mu_{u,s})\leq \omega^{z}(t,s)$, $z>1$, where $\omega$ is a regular control. However, this cannot be applied to the jump case with nonregular controls. Near a jump time, this error
does not vanish as the interval length tends to zero, so the standard partition-refinement argument fails.
	\end{remark}
	\begin{theorem}[\textbf{Multiplicative sewing Lemma with Jumps}] \label{thm:mul01}
	Let $(\mathcal{M}, |\cdot|, d)$ be a submultiplicative monoid and let  $\mu\in BG_{2}([0,T], \mathcal{M})$. Assume that there exist controls $\omega_{1,j}, \omega_{2,j}$ and exponents $\alpha_{1,j},\alpha_{2,j}>0$, $\alpha_{1,j}+\alpha_{2,j}>1$, for all $j=1,...,N$  such that 
  \[  d(\mu_{t,s}, \mu_{t,u}\circ \mu_{u,s})\lesssim  \sum_{j=1}^{N} \omega_{1,j}^{\alpha_{1,j}}(t,u) \omega_{2,j}^{\alpha_{2,j}}(u,s).\]
	Then there exists a unique multiplicative map $\varphi \in \mathrm{BG}_2([0, T], \mathcal{M})$ such that for every $(t,s)\in \Delta_{2}$, 
\begin{align}
&\varphi_{t,s}:=RS-\lim\limits_{|\mathcal{P}|\to 0}\mu_{t,s}^{\mathcal{P}}, \label{jixian01}\\
& |\varphi_{t,s}|\lesssim \sigma(t,s). \nonumber
\end{align}
Moreover, 
	\begin{align}\label{est001}
d(\varphi_{t,s}, \mu_{t,s})\lesssim  \sum_{j=1}^{N} \omega_{1,j}^{\alpha_{1,j}}(t,s+) \omega_{2,j}^{\alpha_{2,j}}(t-,s),
\end{align}
 
Furthermore, if $\omega_{1,j}$ are right-continuous in $s$ or $\omega_{2,j}$ are left-continuous in $t$, then the convergence holds in the mesh sense.
\end{theorem}

\begin{proof}
Fix $(t,s)\in \Delta_2$ with $s<t$ and fix
 constants $\alpha_{1,j}'$, $\alpha_{2,j}'$,  $\eta$ such that
	$
	1<\eta<\min\limits_j\{\alpha'_{1,j}+\alpha'_{2,j}\}$, $ 0<\alpha_{1,j}'<\alpha_{1,j}$, $0<\alpha_{2,j}'<\alpha_{2,j}.
	$
	For any $\varepsilon>0,$ according to Lemma \ref{partition}, we can choose a partition $\mathcal{P}_{j}$ of $[s,t]$ such that for any $(u,v)\in \mathcal{P}_j$, $m\in (u,v)$, 
	$$
	\omega_{1,j}^{\alpha_{1,j}-\alpha_{1,j}'}(v,m) \omega_{2,j}^{\alpha_{2,j}-\alpha_{2,j}'}(m,u) < C \varepsilon,
	$$
	with $C$ depending on $\max\limits_{j=1,...,N}\left(\omega_{1,j}(t, s)\vee \omega_{2,j}(t, s)\right)$ and $\min\limits_{  j=1,...,N;i=1,2 }(\alpha_{i,j}-\alpha_{i,j}').$
	Then let 
    \begin{align}\label{Pnb}
    \mathcal{P}:= \vee_{j=1}^N \mathcal{P}_j,
    \end{align}
    and one has for any $(u,v) \in \mathcal{P},$
	\begin{equation} \label{o3o3}
    \begin{split}
    d(\mu_{v,u}, \mu_{v,m}\circ \mu_{m,u})&\leq  \sum_{j=1}^{N} \omega_{1,j}^{\alpha_{1,j}}(v,m) \omega_{2,j}^{\alpha_{2,j}}(m,u)  \\
		&\leq \varepsilon C \sum_{j=1}^{N} \omega_{1,j}^{\alpha_{1,j}'}(v,m) \omega_{2,j}^{\alpha_{2,j}'}(m,u)    \\
        &\leq  \varepsilon C \sum_{j=1}^{N} \big(\omega_{1,j}^{\frac{\alpha_{1,j}'}{\eta}}(v,u) \omega_{2,j}^{\frac{\alpha_{2,j}'}{\eta}}(v,u) \big)^\eta \\
		& \leq  \varepsilon C \omega_3^\eta(v,u),  
        \end{split}
	\end{equation}
	where  
    \begin{align}\label{o3}
    \omega_3 (v,u):=\sum_{j=1}^{N} \omega_{1,j}^{\frac{\alpha_{1,j}'}{\eta}}(v,u) \omega_{2,j}^{\frac{\alpha_{2,j}'}{\eta}}(v,u).
    \end{align}
    Since
$
    \frac{\alpha'_{1,j}+\alpha'_{2,j}}{\eta}>1,
$
    Lemma \ref{product-control} implies that $\omega_3$ is a control.
 For any other refinement $\mathcal{P}'$ of $\mathcal{P},$ and any $(u,v)\in \mathcal{P},$ without loss of generality, we assume that $\hat{\mathcal{P}}:=\mathcal{P}'|_{[u,v]}=\{u=t_{0}<t_{1}<\dots<t_{r}=v\}$ has $r$ intervals with $r\geq 2$.  
There exist $l\in \{1, \dots, r-1\}$,  such that 
\begin{align}\label{rr01}
\omega_3(t_{l+1}, t_{l-1})\leq \frac{2}{r-1}\omega_3(v, u).
\end{align}
We prove the inequality  \eqref{rr01} by contradiction.
Suppose 
\begin{align}\label{maodun}
\omega_3(t_{l+1}, t_{l-1})> \frac{2}{r-1}\omega_3(v, u), 
\end{align}
for every $l\in \{1, \dots, r-1\}.$
Summing over all $l=1,\dots, r-1$, we have 
$$
\sum_{l=1}^{r-1}\omega_{3}(t_{l+1},t_{l-1})>(r-1)\frac{2}{r-1}\omega_{3}(v,u)=2\omega_{3}(v,u).
$$
Consider the intervals $[t_{l-1}, t_{l+1}]$ with odd $l$ are pairwise disjoint up to the endpoints, and the same holds for even $l$. Since their union is contained in $[u,v]$, super-additivity gives
\begin{align}\label{odd}
\sum_{1\leq l \leq r-1}\omega_{3}(t_{l+1},t_{l-1})\leq \omega_{3}(v,u), \quad l~ \text{is odd}.
\end{align}
Similarly, for even $l$, we have 
\begin{align}\label{eve}
\sum_{1\leq l \leq r-1}\omega_{3}(t_{l+1},t_{l-1})\leq \omega_{3}(v,u).
\end{align}
Combining with \eqref{odd} and \eqref{eve}, we have 
\begin{align*}
    \sum_{l=1}^{r-1}\omega_{3}(t_{l+1},t_{l-1})\leq 2\omega_{3}(v,u),
\end{align*}
which contradicts \eqref{maodun}.
Since $\omega_3 (v,u)$ is also a control by Lemma \ref{product-control}, we may apply the argument used in the proof of \cite[Theorem 3.4]{GHN21} with $\omega_3$.
Denote by $\tilde{\mathcal{P}}$ the partition of $[u,v]$ obtained by removing one point $t_{l}$ from $\hat{\mathcal{P}}$.
One obtains that
\begin{align*}
	d(\mu^{\tilde{\mathcal{P}}}_{v,u}, ~\mu^{\hat{\mathcal{P}}}_{v,u}) &\leq   |\prod_{i = l+1}^{n-1} \mu_{t_{i+1},t_i}\;|\, d\big(\mu_{t_{l+1},t_{l-1}}, \mu_{t_{l+1},t_{l}}\circ\mu_{t_{l},t_{l-1}}\big)\, |\prod_{i = 0}^{l-2} \mu_{t_{i+1},t_i}| \nonumber\\
	& \leq \varepsilon C \sigma(v, t_{l+1}) \omega_{3}^{\eta}(t_{l+1},t_{l-1}) \sigma(t_{l-1},u) \\
	&\leq \varepsilon C\sigma^2(v,u) (\frac{2}{r-1})^{\eta} \omega_{3}^{\eta}(v,u).
	\end{align*}
Iterating this procedure on $\hat{\mathcal{P}}$, we have the following inequality
\begin{align}\label{maxin01}
d(\mu_{v,u}, \mu^{\hat{\mathcal{P}}}_{v,u})&\lesssim 
\varepsilon C \sigma^{2}(v,u) \sum_{r=2}^{\infty}(\frac{2}{r-1})^{\eta}\omega_{3}^{\eta}(v,u) \nonumber\\
&\lesssim  \varepsilon  2^{\eta} \zeta(\eta)C \sigma^{2}(v,u)\omega_{3}^{\eta}(v,u),
\end{align}
where $\zeta(\eta)=\sum\limits_{k=1}^{\infty}k^{-\eta}$ is the Riemann zeta function.
In the following, we show that the existence of the limit \eqref{jixian01}. Without loss of generality, we consider any other refinement $\mathcal{P}^{\prime}$ of $\mathcal{P}$. If $\mathcal{P}=\{s=t_{0}<t_{1}<\dots< t_{k}=t\}$,
we define
		$$\mathscr{M}^l_{t,s} = \prod_{i = l}^{k-1} \mu^{\mathcal{P}^{\prime}\cap [t_i, t_{i+1}]}_{t_{i+1},t_i} \prod_{i = 0}^{l-1} 
		\mu_{t_{i+1},t_i}.$$
		Note that $\mathscr{M}^{k}_{t,s} = \mu^{\mathcal{P}}_{t,s}$ and $\mathscr{M}^0_{t,s} = \mu^{\mathcal{P}^{\prime}}_{t,s}$.
Using the submultiplicativity and super-additivity of $\omega_{3}$ and  (\ref{maxin01}), we have 
\begin{equation}\label{cauchy01}
	\begin{split}
			d(\mu^{\mathcal{P}}_{t,s}, \mu^{\mathcal{P}^{\prime}}_{t,s}) &\leq \sum_{i = 0}^{k-1} d(\mathscr{M}^i_{t,s}, \mathscr{M}^{i+1}_{t,s})\\
            &\leq \sigma^2(t,s) \sum_{i = 0}^{k-1} d(\mu_{t_{i+1},t_i}, \mu^{\mathcal{P}^{\prime} \cap [t_i, t_{i+1}]}_{t_{i+1},t_i}) \\ 
			& \leq \varepsilon C \sigma^4(t,s) 2^{\eta}\zeta(\eta) \sum_{i = 0}^{k-1} \omega_{3}^{\eta}(t_{i+1},t_i)\\
             &\leq C'
       \varepsilon,
		\end{split}
        \end{equation}
    where $C'$ is a constant depending on $\sigma, s, t, \eta$. 
Since $\eta>1$, $\sigma(t,s)$ is bounded and $\omega_{3}$ is a control, we obtain the RS convergence. 


We consider two refinements $\mathcal{R}, \mathcal{R}'$ of $\mathcal{P}$, where $\mathcal{P}$ 
 is constructed as shown in \eqref{Pnb}. 
From \eqref{cauchy01} and triangle inequality, 
\begin{align}\label{kexi}
 d(\mu^\mathcal{R}_{t,s},\mu^{\mathcal{R}'}_{t,s}) \leq d(\mu^{\mathcal{R}}_{t,s},\mu^{\mathcal{P}}_{t,s}) + d(\mu^{\mathcal{P}}_{t,s},\mu^{\mathcal{R}'}_{t,s}) \leq 2C'\varepsilon,
\end{align}
which implies the sequence $(\mu^{\mathcal{P}}_{t,s})$ ordered by refinement is Cauchy. Since $|\mu_{t,s}^{\mathcal{P}}|\lesssim  \sigma(t,s)$ and the completeness of the sets $\{a\in \mathcal{M}:|a|\leq C\}$, we derive that $\mu_{t,s}^{\mathcal{P}}$ converges to a unique limit $$
\varphi_{t,s}:=RS-\lim\limits_{|\mathcal{P}|\to 0}\mu_{t,s}^{\mathcal{P}}, 
$$ satisfying $|\varphi_{t,s}|\lesssim \sigma(t,s)$. 
    For any fixed $(t,s)\in \Delta_2$,  \eqref{kexi} also allows us to derive that the limit $\varphi_{t,s}$ is  independent of  the choice of the refinement.

Next,  we show that $\varphi$ is multiplicative.  Let $\mathcal{P}_1$, $\mathcal{P}_2$ be the partitions of $[u,t], [s,u]$. Since
\[ \mu^{\mathcal{P}_1\vee \mathcal{P}_2}_{t,s}=\mu^{\mathcal{P}_1}_{t,u}\mu^{\mathcal{P}_2}_{u,s} \]
passing to the refinment limit and using the continuity of the multiplication on bounded sets (Remark 3.2) yields
\[ \varphi_{t,s}=\varphi_{t,u}\varphi_{u,s}.\]

Now, to prove \eqref{est001}, suppose $\mathcal{P}_{n}$ is a sequence of partitions on $[s,t]$ such that the limit above exists. Choose any constant $1<\theta<\min\limits_{j=1, \dots, N}\{\alpha_{1,j}+\alpha_{2,j}\}$. For any $(u,v)\in \mathcal{P}_{n}$, $m\in (u,v)$, one has
\begin{align*}
d(\mu_{v,u}, \mu_{v,m}\circ \mu_{m,u})\leq \big(\sum_{j=1}^{N} \omega_{1,j}^{\frac{\alpha_{1,j}}{\theta}}(v,m) \omega_{2,j}^{\frac{\alpha_{2,j}}{\theta}}(m,u) \big)^{\theta},
 \end{align*}
 Hence, we may use the computations in \eqref{maxin01} to 
\begin{align}\label{o4}
\omega_{4}(v,u):=\sum_{j=1}^{N}\omega_{1,j}^{\frac{\alpha_{1,j}}{\theta}}(v,u+)\omega_{2,j}^{\frac{\alpha_{2,j}}{\theta}}(v-,u).
\end{align}
Then for every $(t,s)\in \Delta_{2}$,  we have
\begin{equation}\label{limsew01}
\begin{split}
d(\mu^{\mathcal{P}_{n}}_{t,s}, \mu_{t,s})
&\lesssim \omega_{4}^{\theta}(t,s)\\
&\lesssim \big(\sum_{j=1}^{N}\omega_{1,j}^{\frac{\alpha_{1,j}}{\theta}}(t,s+)\omega_{2,j}^{\frac{\alpha_{2,j}}{\theta}}(t-,s)\big)^{\theta}\\
&\lesssim \sum_{j=1}^{N} \omega_{1,j}^{\alpha_{1,j}}(t,s+) \omega_{2,j}^{\alpha_{2,j}}(t-,s), 
\end{split}
\end{equation}
where  in the last inequality, we  apply
$ (\sum\limits_{i=1}^N a_i)^\theta \leq N^{\theta-1}  (\sum\limits_{i=1}^N a_i^\theta)
$, for $\theta>1$, $a_i\geq 0$. 
The inequality (\ref{est001}) follows by taking the limit.

It remains to show the uniqueness of the multiplicative map $\varphi: \Delta_2 \to \mathcal{M}$.  Suppose that 
$\tilde\varphi \in BG_2([0,T], \mathcal{M})$ is another multiplicative map,  satisfying $|\tilde \varphi_{t,s}|\lesssim  \tilde \sigma(t,s)$ and
\begin{align*}
d(\tilde\varphi_{t,s}, \mu_{t,s})\lesssim  \sum_{j=1}^{N} \omega_{1,j}^{\alpha_{1,j}}(t,s+) \omega_{2,j}^{\alpha_{2,j}}(t-,s).
\end{align*}
For any $\varepsilon>0$, the exists partition $\mathcal{P}:=\{s=t_0<t_1<\cdots<t_n=t\}$, such that for any $(t_i, t_{i+1})\in \mathcal{P}$, $i=0,\cdots, n-1$, we have 
\begin{align}\label{o3o3o3}
d(\tilde\varphi_{t_{i+1},t_{i}}, \mu_{t_{i+1}, t_i})
\lesssim \varepsilon  \omega_3^\eta (t_{i+1},t_i)
\end{align}
from the similar deduction in \eqref{o3o3}. 
Define $\mathscr{D}_{t,s}^l=\varphi_{t,t_l}\tilde\varphi_{t_l,s}$, for $1\leq l \leq n-1$. Denote $\mathscr{D}^0_{t,s}=\varphi_{t,s}$ and $\mathscr{D}_{t,s}^n=\tilde\varphi_{t,s}$ and denote $\Gamma(t,s)=1+\sigma(t,s)+\tilde \sigma(t,s)$. By the same argument as in the proof of \eqref{cauchy01}, together with \eqref{o3o3o3}, we obtain
\begin{align*}
d(\varphi_{t,s}, \tilde\varphi_{t,s})&\leq \sum\limits_{l=0}^{n-1}d(\mathscr{D}^l_{t,s}, \mathscr{D}^{l+1}_{t,s})\\
&\leq  \sum\limits_{l=0}^{n-1}|\varphi_{t, t_{l+1}}|d(\varphi_{t_{l+1}, t_l}, \tilde \varphi_{t_{l+1}, t_l})|\tilde \varphi_{t_{l}, s}|\\
&\lesssim  \varepsilon  \Gamma(t,s)^2 \sum_{l=0}^{n-1} \omega_3^\eta (t_{l+1}, t_l)\\
&\lesssim \varepsilon \Gamma(t,s)^2 \omega_3^\eta (t,s).
\end{align*}
Since $\varepsilon$ is arbitrary,  we have $\varphi= \tilde \varphi$.  The uniqueness of the map is proved.


In the following, we give the proof of the mesh sense convergence.~Assume that for every
$j=1,\ldots,N$, the map
$
    u\to \omega_{1,j}(r,u)
$
is right-continuous for every fixed $r\in[0,T]$.  Fix some $\kappa>0$ such that $\kappa<\min\limits_{j=1,\dots,N.}\{(\alpha_{2,j}
+\alpha_{1,j}-1)\wedge \alpha_{1,j}\}$. 
Define
\begin{align}\label{Omega-kappa}
    \Omega_\kappa(v,u)
    :=
    \sum_{j=1}^N
    \omega_{1,j}(v,u)^{\alpha_{1,j}-\kappa}
    \omega_{2,j}(v,u)^{\alpha_{2,j}}.
\end{align}
Since~$
    (\alpha_{1,j}-\kappa)+\alpha_{2,j}>1,$
Lemma~\ref{product-control} implies that
$\Omega_\kappa$ is a control.
Fix $(t,s)\in\Delta_2$ with $s<t$, and let $\varepsilon>0$. 
Choose $\rho>0$,  such that 
$$
\rho
    \left[
        1+
        C2^\eta \zeta(\eta)\sigma^4(t,s)\Omega_\kappa(t,s)
    \right]
    <\varepsilon.
$$
According to the definition of RS convergence,  there exists a partition 
$$
    \mathcal Q^\rho
    =
    \{s=\tau_0<\tau_1<\cdots<\tau_m=t\}
$$
such that for any refinement $\mathcal{R}$ of $\mathcal{Q}^{\rho}$, 
\begin{align}\label{rs-rho}
d(\mu^{\mathcal{R}}_{t,s}, \varphi_{t,s})<\rho.
\end{align}
We first recall that for
every $u<T$ and every $j$,
$$
    \lim_{v\downarrow u}\omega_{1,j}(v,u)=0.
$$
Hence,  for every $i=0,\ldots,m-1$ and every
$j=1,\ldots,N$, there exists a constant
$$
    \delta_{i,j}\in(0,\tau_{i+1}-\tau_i]
$$
such that $\tau_i<v<\tau_i+\delta_{i,j}$ and 
\begin{align}\label{small-right-control}
    \omega_{1,j}(v,\tau_i)
    <
    \rho^{1/\kappa}.
\end{align}
Set
$$
    \operatorname{gap}(\mathcal Q^\rho)
    :=
    \min_{0\leq i<m}
    (\tau_{i+1}-\tau_i)
$$
and
\begin{align}\label{mesh-threshold}
    \delta_\rho
    :=
        \operatorname{gap}(\mathcal Q^\rho)  \wedge
    \min_{\substack{0\le i<m\\1\le j\le N}} \delta_{i,j}.
\end{align}
Since $\mathcal Q^\rho$ is finite and strictly increasing, all the
numbers in the minimum are positive. Hence $\delta_\rho>0$.
Let
$
    \mathcal P
    =
    \{s=p_0<p_1<\cdots<p_L=t\}
$
be an arbitrary partition satisfying
$    |\mathcal P|<\delta_\rho,
$
and set
$
    \tilde{\mathcal P}
    :=
    \mathcal P\vee\mathcal Q^\rho.
$
We claim that
each interval $(p_l,p_{l+1})$ contains at most one point of
$\mathcal Q^\rho$. Indeed, if
$\tau_i<\tau_j$ both belonged to $(p_l,p_{l+1})$, then
\[
    \operatorname{gap}(\mathcal Q^\rho)
    \leq
    \tau_j-\tau_i
    <
    p_{l+1}-p_l
    \leq
    |\mathcal P|
    <
    \operatorname{gap}(\mathcal Q^\rho),
\]
which yields a contradiction. 
If $(p_l,p_{l+1})$ contains no point of
$\mathcal Q^\rho$, then
$
    \mu_{p_{l+1},p_l}^{
        \tilde{\mathcal P}\cap [p_l,p_{l+1}]}
    =
    \mu_{p_{l+1},p_l}.
$
If it contains exactly one point $\tau$, then
$
    p_{l+1}-\tau
    <
    |\mathcal P|
    <
    \delta_\rho
    \leq
    \delta_{i,j},
$
which allows us to use  \eqref{small-right-control} to obtain
\begin{align}\label{kaka}
    \omega_{1,j}(p_{l+1},\tau)
    <
    \rho^{1/\kappa}
\end{align}
for every $j$. Thus, from \eqref{Omega-kappa}, 
 \eqref{kaka} and the monotonicity of the control, we derive 
 \begin{equation}\label{ms-local}
     \begin{split}
d\bigl(
    \mu_{p_{l+1},p_l},
    \mu_{p_{l+1},\tau}
    \circ
    \mu_{\tau,p_l}
\bigr)
&\leq
\sum_{j=1}^N
\omega_{1,j}(p_{l+1},\tau)^{\alpha_{1,j}}
\omega_{2,j}(\tau,p_l)^{\alpha_{2,j}}
\\
&\leq
\rho
\sum_{j=1}^N
\omega_{1,j}(p_{l+1},p_l)^{
    \alpha_{1,j}-\kappa}
\omega_{2,j}(p_{l+1},p_l)^{\alpha_{2,j}}
\\
&=
\rho\,
\Omega_\kappa(p_{l+1},p_l).
\end{split}
\end{equation}
Similarly, define
\begin{align*}
\mathscr{U}_{t,s}^{m}=\prod_{l=m}^{L-1} \mu^{\tilde{\mathcal{P}}\cap [p_l, p_{l+1}]}_{p_{l+1},p_l} \prod_{l= 0}^{m-1} 
		\mu_{p_{l+1},p_l}.
\end{align*}
Note that $\mathscr{U}_{t,s}^{L}=\mu_{t,s}^{\mathcal{P}}$ and $\mathscr{U}_{t,s}^{0}=\mu_{t,s}^{\tilde{\mathcal{P}}}$.
From \eqref{ms-local}, we obtain
\begin{equation}\label{ms-glo}
\begin{split}
d\bigl(
    \mu_{t,s}^{\mathcal P},
    \mu_{t,s}^{\tilde{\mathcal P}}
\bigr)&\leq \sum_{l=0}^{L-1}d(\mathscr{U}_{t,s}^{l},\mathscr{U}_{t,s}^{l+1})\\
&\leq
\sigma^2(t,s)
\sum_{l=0}^{L-1}
d\left(
    \mu_{p_{l+1},p_l},
    \mu_{p_{l+1},p_l}^{
        \tilde{\mathcal P}\cap[p_l,p_{l+1}]}
\right)
\\
&\leq
\rho  C2^\eta \zeta(\eta)\sigma^2(t,s)
\sum_{l=0}^{L-1}
\Omega_\kappa(p_{l+1},p_l)
\\
&\leq
\rho C2^\eta \zeta(\eta)\sigma^4(t,s)
\Omega_\kappa(t,s),
\end{split}
\end{equation}
where the last inequality follows from the super-additivity of
$\Omega_\kappa$.
Since
$
    \tilde{\mathcal P}
    \supset
    \mathcal Q^\rho,
$
\eqref{rs-rho} gives
\[
    d\bigl(
        \mu_{t,s}^{\tilde{\mathcal P}},
        \varphi_{t,s}
    \bigr)
    <\rho.
\]
Combining this with \eqref{ms-glo}, we obtain
\[
\begin{aligned}
    d\bigl(
        \mu_{t,s}^{\mathcal P},
        \varphi_{t,s}
    \bigr)
    &\leq
    d\bigl(
        \mu_{t,s}^{\mathcal P},
        \mu_{t,s}^{\tilde{\mathcal P}}
    \bigr)
    +
    d\bigl(
        \mu_{t,s}^{\tilde{\mathcal P}},
        \varphi_{t,s}
    \bigr)\\
    &<
    \rho
    \left[
        1+
        C2^\eta \zeta(\eta)\sigma^4 (t,s)\Omega_\kappa(t,s)
    \right]\\
    &<
    \varepsilon.
\end{aligned}
\]
Therefore, for every $\varepsilon>0$, there exists $\delta_\rho>0$ such that $|\mathcal P|<\delta_\rho$, one has 
$$
    d\bigl(
        \mu_{t,s}^{\mathcal P},
        \varphi_{t,s}
    \bigr)
    <\varepsilon.
$$
This proves the MS convergence.  A similar argument can be used when $\omega_2$ is left-continuous in $t$.

\end{proof}
\subsection{Affine sewing lemma with jumps}\label{AFF}
In what follows, we study the integrals 
$$
    \int_0^t S_{t,r}Y_r\,\textup{d} X_r,
$$
where the driving path $X$ has finite $p$-variation. We begin with the Young case $p\in[1,2)$, while the rough case  $p\in[2,3)$
 will be presented in Theorem \ref{affrough01}. To construct the integral, we represent the associated Riemann sums as products in a suitable monoid, which allows us to apply the multiplicative sewing lemma directly.

Let $\alpha \in \mathbb{R}$ and define $\mathcal{M}(\alpha):=\mathcal{L}(\mathcal{B}_{\alpha})\ltimes \mathcal{B}_{\alpha}$.  For $\mu_{i}=(S_{i}, x_{i})\in \mathcal{M(\alpha)}$.  The multiplication of two elements $\mu_{1}, \mu_{2}$ can be defined as $\mu_{1}\circ \mu_{2}=(S_{1}S_{2}, x_{1}+S_{1}x_{2})$. Moreover, $\mathcal{M}(\alpha)$ is a near ring, which forms a monoid and the multiplication is right-distributive, i.e.~$(\mu_{1}+\mu_{2})\circ \mu_{3}=\mu_{1}\circ \mu_{3}+\mu_{2}\circ \mu_{3}$. Denote the distance between $\mu_{1}$ and $\mu_{2}$ by
	$
	d(\mu_{1}, \mu_{2}):=\|S_{1}-S_{2}\|_{\mathcal{L}(\mathcal{B}_{\alpha})}+|x_{1}-x_{2}|_{\alpha}.
	$
\begin{remark}[\textbf{MRS v.s. RRS}]\label{MRSRRS}
For a partition $\mathcal P$ of $[s,t]$, let $\mu_{t,s}:=(S_{t,s}, S_{t,s}\xi_{t,s})$ 
$\mu_{t,s}^{\mathcal{P}}:=(S_{t,s}, \sum\limits_{(u,v)\in \mathcal{P}} S_{t,u}\xi_{v,u})$, $\mathcal{U}:=(S_{t,s}, \mathcal{K})$. 
 Using  $d$  defined above, we say that the Riemann sum $\sum\limits_{(u,v)\in \mathcal{P}} S_{t,u}\xi_{v,u}$ converges to $\mathcal{K}$ in the following sense if $\mu_{t,s}^{\mathcal P}$ converges to $\mathcal U$ in the corresponding sense in  Definition \ref{MSRS}. 
\begin{itemize}
    \item Mesh Riemann-Stieltjes (MRS) sense:  if for any $\varepsilon>0$, there exists $\delta>0$, such that any $|\mathcal{P}|<\delta$, one has 
    $$
     |\sum_{(u,v)\in \mathcal{P}} S_{t,u}\xi_{v,u}-\mathcal{K}|_{\alpha}<\varepsilon.
    $$
    \item Refinement Rienmann-Stieltjes (RRS) sense: if for any $\varepsilon>0$, there exists $\mathcal{P}_\varepsilon$ such that for any refinement $\mathcal{P}\supset \mathcal{P}_\varepsilon$, 
    $$
     |\sum_{(u,v)\in \mathcal{P}} S_{t,u}\xi_{v,u}-\mathcal{K}|_{\alpha}<\varepsilon.
    $$
\end{itemize}
\end{remark}
\begin{theorem}[\textbf{Affine Sewing Lemma with Jumps and Young}]\label{affyoung01}
Consider $(S_{t,s})_{(t, s)\in \Delta_2}$ as a propagator  in the sense of
Definition~\ref{fullrange}. Fix $\gamma \in \mathbb{R}$, $\theta\in (0,1]$. Let $\xi:\Delta_{2}\to \mathcal{B}_{\gamma}$, $\delta \xi:\Delta_{3}\to \mathcal{B}_{\gamma-\theta}$ and assume that there exist controls $\omega_1, \omega_2, \omega_3$ and positive constants $\alpha_1, \alpha_2, \alpha_3$ with $\min\{\alpha_{1}+\theta, \alpha_{2}+\alpha_{3}\}>1$ such that for every $(t,u,s)\in\Delta_3$
\begin{align}
|\xi_{t,s}|_{\gamma}&\leq \omega_{1}^{\alpha_{1}}(t,s), \label{xi001}\\
|\delta \xi_{t,u,s}|_{\gamma-\theta}&\leq \omega_{2}^{\alpha_{2}}(t,u)\omega_{3}^{\alpha_{3}}(u,s).\label{xi002}
\end{align} 
 Then for every $t\in [0,T]$ the limit 
	\begin{equation} \label{app01}
		\mathcal{I}_t(\xi) :=RRS- \lim_{|\mathcal{P}|\to 0} \sum_{[u,v]\in \mathcal{P}}S_{t,u}\xi_{v,u}
	\end{equation}
	exists uniquely in $\mathcal{B}_{\gamma-\theta}$, over arbitrary partitions $\mathcal{P}$ of $[0,t]$.
    
   Moreover, for every $0\leq s<t\leq T$, 
	\begin{equation}\label{sew01}
		| \delta^{S}\mathcal{I}_{t,s}(\xi)-S_{t,s}\xi_{t,s}|_{\gamma-\theta}
		\leq C [\omega^{\alpha_{1}}_{1}(t, s+)|t-s|^{\theta}+\omega_{2}^{\alpha_{2}}(t,s+)\omega_{3}^{\alpha_{3}}(t-,s)],
	\end{equation}
    where $C$ depends on $S$ and on $\min\{\alpha_{1}+\theta, \alpha_{2}+\alpha_{3}\}$. 
    
If $\omega_{2}$ are right-continuous in their second variable or $\omega_{3}$ are left-continuous in its first variable, then the convergence in \eqref{app01} holds in the MRS sense. 
\end{theorem}
\begin{proof}
Denote the monoid $\mathcal{M}=\mathcal{M}(\gamma-\theta)$. 
	
	\textbf{Step 1. Submultiplicative monoid structure on $\mathcal{M}$.}
	
	Given $\mu_{i}=(S_{i}, x_{i})$, $i=1,2.$
	Denote the distance between $\mu_{1}$ and $\mu_{2}$ by
	$$
	d(\mu_{1}, \mu_{2}):=\|S_{1}-S_{2}\|_{\mathcal{L}(\mathcal{B}_{\gamma-\theta)}}+|x_{1}-x_{2}|_{\gamma-\theta}  .
	$$
	 Define $$
	 |\mu|:=1+d(\mu,(\textup{id},0)). 
	 $$
We have 
 $d(\mu_1\circ\nu, \mu_2\circ \nu) \leq d(\mu_1,\mu_2) 
|\nu|$ and the inequality $d(\nu\circ\mu_1, \nu\circ \mu_2) \leq  |\nu| d(\mu_1,\mu_2)$ from  the right-distributivity of $\circ$. Since $|\cdot|$ is continuous with respect to $d$,
the completeness assumption from Definition \ref{monoid01} is satisfied and one concludes that 
$(\mathcal{M},|\cdot |,d)$ is a submultiplicative monoid.

\textbf{ Step 2. Almost-multiplicativity of $\mu$.} 

Define $\mu :\Delta_2 \to \mathcal{M}$ as
$
\mu _{t,s}:=(S_{t,s}, S_{t,s}\xi _{t,s})\,.
$
For every $(t,u,s)\in\Delta_3,$
observe that 
\begin{equation*}
	\mu _{t,s}-\mu _{t,u }\circ\mu _{u ,s}
	\equiv
	(0,S_{t,s}\delta \xi _{t,u ,s}+ S_{t,u }(S_{u ,s}-\textup{id})\xi _{t,u}).
\end{equation*}
From \eqref{xi001} and  \eqref{xi002}, we have
\begin{align*}
d(\mu _{t,s},\mu _{t,u }\circ\mu _{u ,s})
	&\lesssim  |\delta \xi_{t,u,s} |_{\gamma -\theta} + |(S_{u ,s}-\textup{id})\xi _{t,u }|_{\gamma -\theta}\\
    &\lesssim \omega_{2}^{\alpha_{2}}(t,u)\omega_{3}^{\alpha_{3}}(u,s)+\omega_{1}^{\alpha_{1}}(t,u)|u-s|^{\theta}\\
&:=\sum_{j=1}^{2}\omega_{1,j}^{\alpha_{1,j}}(t,u) \omega_{2,j}^{\alpha_{2,j}}(u,s),
\end{align*}
    which implies  $\mu $ is almost-multiplicative \eqref{almostmul01}.

\textbf{Step 3. Growth of $\mu$.}  

Let $\mathcal{P}:=\{s=t_0<t_1<\dots<t_n=t\}$ be a finite partition of $[s,t]$. Set 
\begin{align*}
\Xi^{\mathcal{P}}_{t,s}:=\sum\limits_{i=0}^{n-1}S_{t,t_i}\xi_{t_{i+1}, t_i}. 
\end{align*} 
Denote $\mu_{t,s}^{\mathcal{P}}:=(S_{t,s}, \Xi_{t,s}^{\mathcal{P}})$. 
Suppose that $a<b<c$ are three consecutive points of $\mathcal{P}$ and let $\mathcal{P}'$ be obtained by removing $b$. Then
\begin{align*}
\Xi_{t,s}^{\mathcal{P}}-\Xi_{t,s}^{\mathcal{P}'}=-S_{t,b}[S_{b,a}\delta \xi_{c,b,a}+(S_{b,a}-\textup{id})\xi_{c,b}],
\end{align*}
which yields
\begin{align}\label{Xilocal}
|\Xi_{t,s}^{\mathcal{P}}-\Xi_{t,s}^{\mathcal{P}'}|_{\gamma-\theta}\lesssim \omega_2(c,b)^{\alpha_2}\omega_3(b,a)^{\alpha_3}+\omega_1(c,b)^{\alpha_1}|b-a|^{\theta}.
\end{align}
Now, we choose $1<\beta<\min \{\alpha_2+\alpha_3, \alpha_1+\theta\}$ and define 
\begin{align*}
\Omega(v,u):=\omega_2(v,u+)^{\alpha_2/\beta}
\omega_3(v-,u)^{\alpha_3/\beta}
+\omega_1(v,u+)^{\alpha_1/\beta}|v-u|^{\theta/\beta}.
\end{align*}
From Lemma \ref{product-control} and Remark \ref{lim-control}, $\Omega$ is a control. 
Plugging this into \eqref{Xilocal}, we get 
\begin{align*}
|\Xi_{t,s}^{\mathcal{P}}-\Xi_{t,s}^{\mathcal{P}'}|_{\gamma-\theta}\lesssim \Omega(c,a)^\beta.
\end{align*}
There exists two intervals $[a,b], [b,c]$, such that 
\begin{align*}
\Omega(c,a)\leq \frac{2}{n-1}\Omega(t,s). 
\end{align*}
Applying the similar iteration argument as shown in  \eqref{maxin01}, we have 
\begin{align*}
|\Xi_{t,s}^{\mathcal{P}}-S_{t,s}\xi_{t,s}|_{\gamma-\theta}&\lesssim  \sum_{n=2}^{\infty}(\frac{2}{n-1})^\beta \Omega(t,s)^\beta\\
&\lesssim 2^\beta \zeta(\beta) \Omega(t,s)^\beta\\
&\lesssim \zeta(\beta) [\omega_2(t,s+)^{\alpha_2}\omega_3(t-,s)^{\alpha_3}
+\omega_1(t,s+)^{\alpha_1}|t-s|^\theta].
\end{align*}
According to the definition of  $|\cdot|$, we have
\begin{align*}
|\mu^{\mathcal{P}}_{t,s}&|=1+d(\mu^{\mathcal{P}}_{t,s}, (\textup{id},0))\\
&\leq 1+d(\mu^{\mathcal{P}}_{t,s}, \mu_{t,s})+d(\mu_{t,s},(\textup{id},0))
\end{align*}
From  (\ref{xi001}) and the continuous embedding $\cB_\gamma \hookrightarrow \cB_{\gamma-\theta}$,  we get
\begin{align*}
d(\mu_{t,s},(\textup{id},0))&=\|S_{t,s}-\textup{id}\|_{\mathcal{L}(\mathcal{B}_{\gamma-\theta})}+|S_{t,s}\xi_{t,s}|_{\gamma-\theta}\\
&\lesssim 1+\omega_{1}^{\alpha_{1}}(t,s).
\end{align*}
Thus we obtain
\begin{align*}
|\mu_{t,s}^{\mathcal{P}}|
&\lesssim 1+\omega_{1}^{\alpha_{1}}(t,s)+\omega_2(t,s+)^{\alpha_2}\omega_3(t-,s)^{\alpha_3}\\
&~~+\omega_1(t,s+)^{\alpha_1}|t-s|^\theta\\
&:=\sigma(t,s). 
\end{align*}
Since controls and their one-sided limits are increasing in $t$ and decreasing in $s$, the same holds for $\sigma$. Moreover, $\sigma$ is bounded on $\Delta_2$. Therefore, $\mu$ has growth bound with $\sigma$.

\textbf{Step 4. Proof of the estimate \eqref{sew01}.} 

By Step 1-3, all assumptions in Theorem  \ref{thm:mul01} are satisfied. Hence  there exists a unique  limit denoted by  $$I_{t,s}:=RRS-\lim\limits_{|\mathcal{P}|\to 0}\sum\limits_{[u,v]\in \mathcal{P}}S_{t,u}\xi_{v,u}.$$ 
And we denote $\varphi_{t,s}:=(S_{t,s}, I_{t,s})$. 
 Using \eqref{est001}, we have 
\begin{align*}
|I_{t,s}-S_{t,s}\xi_{t,s}|_{\gamma-\theta}&=d(\varphi_{t,s}, \mu_{t,s})\\
&\lesssim
\omega_1(t,s+)^{\alpha_1}|t-s|^\theta
+
\omega_2(t,s+)^{\alpha_2}
\omega_3(t-,s)^{\alpha_3}.
\end{align*}
Let $\mathcal{I}_t(\xi)=I_{t,0}$. From the multiplicativity of $\varphi$, we have $I_{t,0}=I_{t,s}+S_{t,s}I_{s,0}$. Therefore $I_{t,s}=\mathcal{I}_t(\xi)-S_{t,s}\mathcal{I}_s(\xi)=\delta^S \mathcal{I}_{t,s}(\xi).$ Finally, we obtain \eqref{sew01}. 

Furthermore, the first part in \eqref{sew01} contains $|t-s|$, which is already a regular control in view of Remark \ref{reg-control}. To obtain convergence in the $MRS$ sense, it suffices that $\omega_2$ be right-continuous in their second variable or that $\omega_3$ be left-continuous in their first variable.

  
\end{proof}
We introduce the following mixed 
$p$-variation space, which will be used to define the Young integral and the solution spaces for our well-posedness analysis. 

For $0<\alpha_1<\alpha_2\in \mathbb{R}$, $1\leq p_1< p_2\leq \infty$, we  
denote 
$$V^{p_{1},p_{2}}_{\alpha_{1}, \alpha_{2}}([0,T]):=V^{p_{1}}([0,T]; \mathcal{L}(\mathbb{R}^d, \mathcal{B}_{\alpha_{1}}))\cap V^{p_{2}}([0,T];  \mathcal{L}(\mathbb{R}^d, \mathcal{B}_{\alpha_{2}}))$$ with the norm 
\begin{equation*}
\|Y\|_{V^{p_{1},p_{2}}_{\alpha_{1}, \alpha_{2}}([0,T])}=
\begin{cases}
\|Y\|_{V^{p_{1}}([0,T];\mathcal{L}(\mathbb{R}^d, \mathcal{B}_{\alpha_{1}}))}+\|Y\|_{V^{p_{2}}([0,T]; \mathcal{L}(\mathbb{R}^d, \mathcal{B}_{\alpha_{2}}))}+\|Y_0\|_{\cL(\mathbb{R}^d, \cB_{\alpha_3})}, ~~&\text{if}~p_2<\infty;\\
\|Y\|_{V^{p_{1}}([0,T];\mathcal{L}(\mathbb{R}^d, \mathcal{B}_{\alpha_{1}}))}+\|Y\|_{V^{p_{2}}([0,T]; \mathcal{L}(\mathbb{R}^d, \mathcal{B}_{\alpha_{2}}))}, \quad &\text{if}~p_2=\infty,
\end{cases}
\end{equation*}
where $\alpha_3\geq \alpha_2$.
Since $\mathbb{R}^d$ is finite-dimensional, we identify
$\mathcal{L}(\mathbb{R}^d,\mathcal{B}_\alpha)$ with
$\mathcal{B}_\alpha^d$. For notational simplicity, we omit the superscript
$d$ and simply write $\mathcal{B}_\alpha$ throughout the rest of the manuscript.

 \begin{lemma}[\textbf{Integration in the Young case}]\label{YOU01} 
	Let $p_{1}, q\in [1,2)$,  $p_{2}\in (p_1, \infty]$, $\theta\in(0,1]$,   $\alpha\in \mathbb{R}$. Let $X\in V^{q}([0,T];  \mathbb{R}^d)$, $Y\in V^{p_{1},p_{2}}_{\alpha,\alpha+\theta}([0,T])$ with $\min\{\frac{1}{p_{1}}+\frac{1}{q},  \theta+\frac{1}{q}\}>1$.
	Then the limit
	\begin{align}
    \label{youint0}
\int_{0}^{t}S_{t,s}Y_{s}\textup{d}X_{s}=RRS-	\lim\limits_{|\mathcal{P}|\rightarrow 0}\sum_{(u,v)\in \mathcal{P}}S_{t,u} Y_{u}\delta X_{v,u}
	\end{align}
	exists in $\mathcal{B}_{\alpha}$. 
    Moreover, 
    we have the local estimate for all  $(s,t)\in \Delta_{2}:$
	\begin{align}\label{young-bound}
|\int_{s}^{t}S_{t,u}Y_{u}\textup{d}X_{u}-S_{t,s}Y_{s}\delta X_{t,s}|_{\alpha}
    \leq C \sum\limits_{j=1}^{2} \omega_{1,j}^{\alpha_{1,j}}(t,s+)\omega_{2,j}^{\alpha_{2,j}}(t-,s),
	\end{align}
    where $C$ depends on $S$, $\min\{\frac{1}{q}+\theta, \frac{1}{q}+\frac{1}{p_{1}}\}$, 
    \begin{align*}
&\omega_{1,1}(t,s+)=\|X\|^{q}_{V^{q}((s,t]; \mathbb{R}^d)}, \quad
\omega_{2,1}(t-,s)=\|Y\|^{p_{1}}_{V^{p_{1}}([s,t);\cL(\mathbb{R}^d,\mathcal{B}_{\alpha}))},\\
&\omega_{1,2}(t,s+)=[|Y_{0}|_{\alpha+\theta}+\|Y\|_{V^{p_{2}}([0,T];\cL(\mathbb{R}^d,\mathcal{B}_{\alpha+\theta}))}]^{q}\|X\|^{q}_{V^{q}((s,t];\mathbb{R}^d)}
,\quad \omega_{2,2}(t,s)=|t-s|,\\
&\alpha_{1,1}=\alpha_{1,2}=\frac{1}{q},~~\alpha_{2,1}=\frac{1}{p_{1}},~~\alpha_{2,2}=\theta.
    \end{align*}
    	Furthermore, the convergence in \eqref{youint0} holds in the MRS
	sense provided that  the control $\omega_{1,1}$ is
    right-continuous in $s$ or the control $\omega_{2,1}$ is
    left-continuous in $t$.
	\end{lemma}
\begin{proof}
	
	For $(t,s)\in \Delta_{2}$, $(t,u,s)\in \Delta_{3}$, we have 
	$$\xi_{t,s}=Y_{s}\delta X_{t,s}$$ and
\begin{align*}
	\delta \xi_{t,u,s}&=\xi_{t,s}-\xi_{t,u}-\xi_{u,s}\\
	&=Y_{s}(\delta X_{t,u}+\delta X_{u,s})-Y_{u}\delta X_{t,u}-Y_{s}\delta X_{u,s}\\
	&=-\delta Y_{u,s} \delta X_{t,u}.
\end{align*}
We verify that
\begin{align*}
|\xi_{t,s}|_{\alpha+\theta}&\leq |Y_{s}|_{\alpha+\theta}|X_{t}-X_{s}|\\
&\leq [|Y_{0}|_{\alpha+\theta}+\|Y\|_{V^{p_{2}}([0,T]; \cL(\mathbb{R}^d, \mathcal{B}_{\alpha+\theta}))}] \|X\|_{V^{q}([s,t]; \mathbb{R}^d)}
\end{align*}
as well as
\begin{align*}
	|\delta \xi_{t,u, s}|_{\alpha}&\leq |Y_{u}-Y_{s}|_{\alpha}|X_{t}-X_{u}|
	\\
	&\leq \|Y\|_{V^{p_1}([s,u];\cL(\mathbb{R}^d,\mathcal{B}_{\alpha}))}\|X\|_{V^{q}([u,t]; \mathbb{R}^d)}.
\end{align*}
Hence we have 
\begin{equation}\label{young-N2}
\begin{split}
	|\delta \xi_{t,u,s}+(S_{u ,s}-\textup{id})\xi _{t,u }|_{\alpha}
   & \leq \|Y\|_{V^{p_1}([s,u], \cL(\mathbb{R}^d, \mathcal{B}_{\alpha}))}\|X\|_{V^{q}([u,t]; \mathbb{R}^d)}\\
    &~~+[|Y_{0}|_{\alpha+\theta}+\|Y\|_{V^{p_{2}}([0,T];\cL(\mathbb{R}^d,\mathcal{B}_{\alpha+\theta}))}]\|X\|_{V^q([u,t]; \mathbb{R}^d)}|u-s|^{\theta}. 
    \end{split}
\end{equation}
Therefore, Theorem \ref{affyoung01} with $\gamma=\alpha+\theta$, $ \alpha_{1}=\alpha_{2}=\frac{1}{q}$ and $\alpha_3=\frac{1}{p_1}$ proves the statement. 
\textcolor{black}{ Equivalently, estimate \(\eqref{young-N2}\) can be written as
\begin{align*}
|\delta \xi_{t,u,s}  + (S_{u ,s}-\textup{id})\xi _{t,u }|_{\alpha}\leq C \sum\limits_{j=1}^{2} \omega_{1,j}^{\alpha_{1,j}}(t,u)\omega_{2,j}^{\alpha_{2,j}}(u,s),
\end{align*}
as stated in Theorem \ref{thm:mul01} for $N=2$. }

\begin{remark}\label{exp-young}
The estimate \eqref{young-bound} is obtained in $\mathcal B_\alpha$ and, in general, cannot be improved to a space of higher regularity. In the continuous case, the smoothing property of the propagator allows one to trade temporal regularity for spatial regularity. This argument is not available in the presence of jumps, since the control increments are not small on short intervals. 
\textcolor{black}{Thus, to obtain an estimate of the convolution in $\mathcal B_\alpha$, we assume that $Y\in V^{p_1, p_2}_{\alpha, \alpha+\theta}$, which satisfy the assumptions of Lemma \ref{YOU01}. Likewise, applying Theorem  \ref{affyoung01} with $\gamma=\alpha, \alpha_{1}=\alpha_{2}=\frac{1}{q}$ and $\alpha_3=\frac{1}{p_1}$, equivalently, with the spatial index $\alpha$ replaced by $\alpha-\theta$ in Lemma \ref{YOU01}, 
gives an estimate for the convolution in $\mathcal B_{\alpha-\theta}$ for  
 $Y\in V^{p_1,p_2}_{\alpha-\theta, \alpha}$. This is consistent with the continuous Young case. }
Finally, in the continuous case, it is known that the convolution improves the spatial regularity by a parameter which is strictly less than the H\"older exponent of the driver, see \cite[Corollary 4.6]{GHN21}. In the jump case, this is not possible, since we cannot rely on the smoothing effect of the propagator and use H\"older estimates to improve the spatial regularity. 
\end{remark}
	\end{proof}
In order to characterize the behavior of the integral when the upper or lower limit of integration approaches a jump point from the left or right, we introduce 
the following lemma, which is an infinite-dimensional version of \cite[Example 1.10]{FZ18}.
\begin{lemma}\label{conjump}
Let $p_1\in  [1,2)$, $p_2\in (p_1, \infty]$, $\alpha\in \mathbb{R}$, $\theta\in(0,1]$.
Suppose $Y\in V^{p_1,p_2}_{\alpha, \alpha+\theta}([0,T])$ and suppose  $X_t$ on $[0,T]$ is constant except at $t=\tau>0$. We have 
\begin{align*}
\int_0^T S_{T,r}Y_r\textup{d}X_r=S_{T,\tau}Y_{\tau-}\Delta_{\tau}^- X+S_{T,\tau}Y_\tau \Delta_\tau^+ X.
\end{align*}
\begin{proof}
It is clear that $X\in V^{1}([0,T];\mathbb{R}^d)$. 
 According to Lemma \ref{YOU01}, the integral is well-defined, 
\begin{align*}
\int_0^T S_{T,r}Y_r \textup{d}X_r= RRS-\lim\limits_{|\mathcal{P}|\to 0}\sum_{(u,v)\in \mathcal{P}}S_{T,u} Y_u \delta X_{v,u}. 
\end{align*}
For any $\varepsilon>0$, we choose a partition  
 $\mathcal{P}=\{0=t_0<t_1<\dots<t_i<t_{i+1}=\tau<t_{i+2}\dots<t_n=T\}$ such that 
\begin{align*}
\sum_{(u,v)\in \mathcal{P}}S_{T,u} Y_u \delta X_{v,u} =S_{T, t_i}Y_{t_i}\delta X_{\tau, t_i}+S_{T, \tau}Y_\tau \delta X_{t_{i+2},\tau}.
\end{align*}
We note
\begin{align*}
&S_{T, t_i}Y_{t_i}\delta X_{\tau, t_i}+S_{T, \tau}Y_\tau \delta X_{t_{i+2},\tau}-S_{T,\tau}Y_{\tau-}\Delta_{\tau}^- X-S_{T,\tau}Y_\tau \Delta_\tau^+ X\\
=&~~S_{T, t_i}Y_{t_i}\Delta_{\tau}^- X-S_{T,\tau}Y_{\tau-}\Delta_{\tau}^- X\\
=&~~S_{T, t_i}(Y_{t_i}-Y_{\tau^{-}})\Delta_{\tau}^{-}X+ (S_{T, t_i}-S_{T, \tau})Y_{\tau-}\Delta_{\tau}^- X.
\end{align*}
Since $Y_{t_i}\to Y_{\tau-}$ in $\cB_\alpha$ as $t_i\uparrow\tau$
and  the continuity of  $S$, the partition $\mathcal{P}$ can be chosen such that 
\begin{align*}
|\sum_{(u,v)\in \mathcal{P}}S_{T,u} Y_u \delta X_{v,u} -S_{T,\tau}Y_{\tau-}\Delta_{\tau}^- X-S_{T,\tau}Y_\tau \Delta_\tau^+ X|_{\alpha}<\varepsilon.
\end{align*}
For every refinement $\tilde{\mathcal{P}}$ of the partition $\mathcal{P}$,  the same estimate holds: 
\begin{align*}
|\sum_{(u,v)\in \tilde{\mathcal{P}}}S_{T,u} Y_u \delta X_{v,u} -S_{T,\tau}Y_{\tau-}\Delta_{\tau}^- X-S_{T,\tau}Y_\tau \Delta_\tau^+ X|_{\alpha}<\varepsilon.
\end{align*}
We have completed the proof.
\end{proof}
\end{lemma}
 The existence of the left and right limits with respect to the upper limit of integration ensures that the solution in Section \ref{sec:wellpose} is well defined in the mixed $p$-variation space $V^{p,p_1}_{\alpha, \alpha+\theta}$, $p<p_1\leq \infty$. Since we need to establish the existence of these left and right limits not only in $\cB_{\alpha}$, but also in $\cB_{\alpha+\theta}$, we  only provide the proof in $\cB_{\alpha+\theta}$. 
\begin{lemma} \label{llrr}
Let $p_1, q\in  [1,2)$, $p_2\in (p_1, \infty]$, $\alpha\in \mathbb{R}$, $\theta\in(0,1]$. Suppose that $Y_0\in \cB_{\alpha+2\theta}$ and $Y\in V^{p_1,p_2}_{\alpha+\theta,\alpha+2\theta}([0,T])$, $X\in V^q([0,T];\mathbb{R}^d)$ with $\min\{\frac{1}{p_{1}}+\frac{1}{q}, \theta+\frac{1}{q}\}>1$. Then for $0\leq s<t\leq T$, 

(1) the right-hand limit $\lim\limits_{u\downarrow s}\int_u^t S_{t,r}Y_r\textup{d}X_r:=\int_{s+}^{t}S_{t,r}Y_{r}\textup{d}X_{r}$ exists in $\cB_{\alpha+\theta}$. Moreover,
\begin{align*}
\int_{s+}^{t}S_{t,r}Y_{r}\textup{d}X_{r}=\int_s^tS_{t,r}Y_{r}\textup{d}X_{r}-S_{t,s}Y_s\Delta_s^+ X.
\end{align*}

(2) The left-hand limit $\lim \limits_{u\uparrow s}\int_u^t S_{t,r}Y_r\textup{d}X_r:=\int_{s-}^t S_{t,r}Y_r\textup{d}X_r$ exists in $\cB_{\alpha+\theta}$. Moreover,
\begin{align*}
\int_{s-}^t S_{t,r}Y_r\textup{d}X_r=\int_{s}^t S_{t,r}Y_r\textup{d}X_r+S_{t,s}Y_{s-}\Delta_{s}^-X.
\end{align*}

(3) The right-hand limit $\lim\limits_{u\downarrow t}\int_s^u S_{u,r}Y_r\textup{d}X_r:=\int_{s}^{t+}S_{t,r}Y_{r}\textup{d}X_{r}$ exists in $\cB_{\alpha+\theta}$. Moreover,
\begin{align*}
\int_s^{t+}S_{t,r} Y_r\textup{d}X_r=\int_s^{t}S_{t,r} Y_r\textup{d}X_r+Y_{t}\Delta_t^
+ X.
\end{align*}

(4) Furthermore, suppose $Y\in V^{p_1,p_2}_{\alpha+\theta+\varepsilon,\alpha+2\theta+\varepsilon}([0,T])$,  $\varepsilon>0$. The left-hand limit $\int_{s}^{t-}S_{t,r}Y_{r}\textup{d}X_{r}:=\lim\limits_{u\uparrow t} \int_s^u S_{u,r}Y_r\textup{d}X_r$ exists in $\cB_{\alpha+\theta}$. Moreover
\begin{align*}
\int_{s}^{t-}S_{t,r}Y_{r}\textup{d}X_{r}=  \int_s^tS_{t,r}Y_{r}\textup{d}X_{r}-Y_{t-}\Delta_t^- X.
\end{align*}
\end{lemma}
\begin{proof}
(1) In order to prove the existence of the right-hand limit  $\lim\limits_{u\downarrow s}\int_u^t S_{t,r}Y_r\textup{d}X_r$, we first remove the jump and therefore introduce an auxiliary process:
$$
\tilde X^1_m:=\left\{
\begin{array}{ll}
X_{s+}, & m=s, \\
	X_m,& m\neq s.
\end{array}
\right.
$$
$\tilde X^1$ is right-continuous at $s$. We claim that $\|\tilde X^1\|_{V^{q}([s,t];\mathbb{R}^d)}\leq \|X\|_{V^{q}((s,t];\mathbb{R}^d)}$. In fact, taking a partition of $(s,t]$, denoted by $\mathcal{P}:=\{s=t_0<t_1<t_2<\dots<t_n=t$\},  we have 
\begin{align*}
\sum_{i=0}^{n-1}|\tilde X^1_{t_{i+1}}-\tilde X^1_{t_i}|^{q}
&= |\tilde X^1_{t_1}-\tilde X^1_s|^{q}+\sum_{i=1}^{n-1}|\tilde X^1_{t_{i+1}}-\tilde X^1_{t_i}|^{q} \\
&=|X_{t_1}-X_{s+}|^{q}+\sum_{i=1}^{n-1}|X_{t_{i+1}}-X_{t_i}|^{q}\\
&=\lim\limits_{u\downarrow s}|X_{t_1}-X_u|^{q}+\sum_{i=1}^{n-1}|X_{t_{i+1}}-X_{t_i}|^{q}.
\end{align*}
Since $u\in (s,t_1)$,  the partition $\{u<t_1<t_2<\dots<t_n=t\}\subset \mathcal{P}$,  we have 
\begin{align*}
|X_{t_1}-X_u|^{q}+\sum_{i=1}^{n-1}|X_{t_{i+1}}-X_{t_i}|^{q}\leq \|X\|^{q}_{V^{q}((s,t];\mathbb{R}^d)}.
\end{align*}
Letting $u\downarrow s$ and then  taking the  supremum over $\mathcal{P}$ of $[s,t]$ gives
\begin{align}\label{tilde-}
\|\tilde X^1\|_{V^{q}([s,t];\mathbb{R}^d)}\leq \|X\|_{V^{q}((s,t];\mathbb{R}^d)}.
\end{align}
 According to  Lemma \ref{YOU01} and Lemma \ref{conjump}, for $h>0$, we derive 
\begin{align*}
&|\int_{s+h}^t S_{t,r}Y_r \textup{d}X_r-\int_{s}^t S_{t,r}Y_r \textup{d}X_r+S_{t,s}Y_s \Delta_s^+ X|_{\alpha+\theta}\\
\lesssim &~|-\int_{s}^{s+h}S_{s+h,r}Y_r \textup{d}X_r+S_{s+h,s}Y_s \Delta_s^
+X|_{\alpha+\theta}
\\
=&~|-\int_{s}^{s+h}S_{s+h,r}Y_r \textup{d}(X-\tilde X^1)_r-\int_{s}^{s+h}S_{s+h,r}Y_r \textup{d}\tilde X^1_r+S_{s+h,s}Y_s \Delta_s^+ X|_{\alpha+\theta}\\
=&~|-S_{s+h,s}Y_s \Delta_s^+ X-\int_{s}^{s+h}S_{s+h,r}Y_r \textup{d}\tilde X^1_r+S_{s+h,s}Y_s \Delta_s^+ X|_{\alpha+\theta}\\
\lesssim &~|\int_{s}^{s+h}S_{s+h,r}Y_r \textup{d}\tilde X^1_r- S_{s+h,s}Y_s\delta \tilde{X}^1_{s+h,s}|_{\alpha+\theta}+|S_{s+h,s}Y_s\delta \tilde{X}^1_{s+h,s}|_{\alpha+\theta}\\
\lesssim &~ \|Y\|_{V^{p_1}([s,s+h];\cB_{\alpha+\theta})}\|\tilde X^1\|_{V^{q}([s,s+h];\mathbb{R}^d)}+
(|Y_0|_{\alpha+2\theta}+\|Y\|_{V^{p_2}([0,s];\cB_{\alpha+2\theta})}) \|\tilde X^1\|_{V^{q}([s,s+h];\mathbb{R}^d)} h^\theta\\
&~~+(|Y_0|_{\alpha+2\theta}+\|Y\|_{V^{p_1}([0,s];\cB_{\alpha+\theta})})\|\tilde X^1\|_{V^{q}([s,s+h];\mathbb{R}^d)}\\
\lesssim &~  (1+h^\theta)(|Y_0|_{\alpha+2\theta}+\|Y\|_{V^{p_1,p_2}_{\alpha+\theta, \alpha+2\theta}([0,T])})\|\tilde X^1\|_{V^{q}([s,s+h];\mathbb{R}^d)}.
\end{align*}
Then from Remark \ref{small-OSC}, for $\varepsilon_0>0$, there exists $\delta>0$, such that $h<\delta$, 
\begin{align*}
|\int_{s+h}^t S_{t,r}Y_r \textup{d}X_r-\int_{s}^t S_{t,r}Y_r \textup{d}X_r+S_{t,s}Y_s \Delta_s^+ X|_{\alpha+\theta}
<\varepsilon_0.
\end{align*}
(2)
Introduce an  auxiliary process
$$
\tilde X^2_m:=\left\{
\begin{array}{ll}
X_{s-}, & m=s, \\
	X_m,& m\neq s.
\end{array}
\right.
$$
It is  clear that $\tilde X^2$ is left-continuous at $s$ and that for $u<s$,  $$
\|\tilde X^2\|_{V^q([u,s];\mathbb{R}^d)}\leq \|X\|_{V^q ([u,s);\mathbb{R}^d)}+|\Delta_s^- X|,
$$
see \cite[Lemma 2.6]{ALP21} for a detailed proof. 

According to  Lemma \ref{YOU01} and Lemma \ref{conjump}, for $h>0$, we derive
\begin{align*}
&|\int_{s-h}^t S_{t,r} Y_r \textup{d}X_r-\int_s^t S_{t,r} Y_r\textup{d}X_r-S_{t,s}Y_{s-}\Delta_s^- X|_{\alpha+\theta}\\
=&|\int_{s-h}^s S_{t,r}Y_r\textup{d}X_r-S_{t,s}Y_{s-}\Delta_s^- X|_{\alpha+\theta}\\
\leq &|S_{t,s}\int_{s-h}^s S_{s,r}Y_r\textup{d}(X-\tilde X^2)_r-S_{t,s}Y_{s-}\Delta_s^- X|_{\alpha+\theta}+ |S_{t,s}\int_{s-h}^ s S_{s,r}Y_r\textup{d} \tilde X^2_r|_{\alpha+\theta}\\
\lesssim &|\int_{s-h}^s S_{s,r}Y_r\textup{d}\tilde X^2_r-S_{s,s-h}Y_{s-h}\delta \tilde X^2_{s,s-h}|_{\alpha+\theta}+|S_{s,s-h}Y_{s-h}\delta \tilde X^2_{s,s-h}|_{\alpha+\theta}\\
\lesssim&  \|\tilde X^2\|_{V^{q}([s-h,s];\mathbb{R}^d)}\|Y\|_{V^{p_1}([s-h,s];\cB_{\alpha+\theta})}+[|Y_0|_{\alpha+2\theta}+\|Y\|_{V^{p_2}([0,s-h];\cB_{\alpha+2\theta})}]\|\tilde X^2\|_{V^{q}([s-h,s];\mathbb{R}^d)}h^\theta\\
~~+&(|Y_0|_{\alpha+2\theta}+\|Y\|_{V^{p_1}([0,s-h];\cB_{\alpha+\theta})})\|\tilde X^2\|_{V^{q}([s-h,s];\mathbb{R}^d)}\\
\lesssim & (1+h^\theta)(|Y_0|_{\alpha+2\theta}+\|Y\|_{V^{p_1,p_2}_{\alpha+\theta, \alpha+2\theta}([0,T])})\|\tilde X^2\|_{V^{q}([s-h,s];\mathbb{R}^d)}.
\end{align*}
Then for $\varepsilon_0>0$, there exists $\delta>0$, such that $h<\delta$, 
\begin{align*}
|\int_{s-h}^t S_{t,r} Y_r \textup{d}X_r-\int_s^t S_{t,r} Y_r\textup{d}X_r-S_{t,s}Y_{s-}\Delta_s^- X|_{\alpha+\theta}<\varepsilon_0.
\end{align*}
(3)
For $h>0$, 
\begin{align*}
&|\int_s^{t+h}S_{t+h,r}Y_r\textup{d}X_r-\int_{s}^{t}S_{t,r}Y_r\textup{d}X_r-Y_t\Delta_t^+ X|_{\alpha+\theta}\\
\leq& |(S_{t+h,t}-\textup{id})\int_s^t S_{t,r}Y_r \textup{d}X_r|_{\alpha+\theta}+|\int_t^{t+h}S_{t+h,r}Y_r\textup{d}X_r-Y_t\Delta_t^+ X |_{\alpha+\theta}\\
:=&I_1+I_2.
\end{align*}
From the continuity of the propagator in Remark \ref{dir}, we claim that $I_1\to 0$, as $h\to 0$.  
Similar to the proof in case (1), we introduce an auxiliary process, which is right-continuous at $t$:
$$
\tilde X^3_m:=\left\{
\begin{array}{ll}
X_{t+}, & m=t, \\
	X_m,& m\neq  t.
\end{array}
\right.
$$
to deal with $I_2$, which allows us to obtain the required result in (3).

(4) 
For $h>0$,
\begin{align*}
&|\int_s^{t-h} S_{t-h,r}Y_r \textup{d}X_r-\int_s^t S_{t,r} Y_r \textup{d}X_r+Y_{t-} \Delta_t^- X|_{\alpha+\theta}\\
\leq &|(\textup{id}-S_{t,t-h})\int_s^{t-h} S_{t-h,r}Y_r\textup{d}X_r|_{\alpha+\theta}
+
|\int_{t-h}^t S_{t,r}Y_r \textup{d}\tilde X^4_r|_{\alpha+\theta}\\
~~+&|\int_{t-h}^t S_{t,r}Y_r\textup{d}(X-\tilde X^4)_r-Y_{t-}\Delta_t^- X|_{\alpha+\theta}\\
:=& I_3+I_4+I_5.
\end{align*}
The estimate of 
$I_3$ is treated in a higher-regularity space under the influence of  the smoothing inequality in \eqref{smooth}. Therefore, in this case, we need to raise the regularity of $Y$ accordingly so that Lemma \ref{YOU01} can be applied.
In terms of $I_4$ and $I_5$, we can use the similar arguments in case (2) by introducing an auxiliary process, which is left-continuous at $t$: 
$$
\tilde X^4_m:=\left\{
\begin{array}{ll}
X_{t-}, & m=t, \\
	X_m,& m\neq t.
\end{array}
\right.
$$
\end{proof}
\section{Well-posedness}\label{sec:wellpose}
In this section, we study well-posedness of \eqref{main} driven by $p$-variation processes, where $p\in [1,2).$ We start with the local well-posedness in the space $V^{p,p_1}_{\alpha,\alpha+\theta}$ for finite $p_1$. This will be further extended to a global well-posedness in the larger space $V^{p,\infty}_{\alpha, \alpha+\theta}$.  The nonlinearity 
$f$ in Lemma \ref{fff1} and Lemma \ref{fff2} is chosen differently when facing different solution spaces, yet both choices require assumptions that enhance regularity. There are two reasons for this: one is to ensure the well-posedness of the solution in $p$-variation space; the other is to ensure that the estimates in Lemma \ref{YOU01} are applicable.
\subsection{Local well-posedness}
    \begin{lemma}\label{fff1}
Suppose $f(\cdot)\in C_b^2(\cB_{\alpha+i\theta}, \cL(\mathbb{R}^d,\mathcal{B}_{\alpha+(i+2)\theta+\varepsilon}))$, for some $\alpha, \theta \in \mathbb{R}$, $\varepsilon>0$ and let $x, y\in V^{p}([0,T];\cB_{\alpha+i\theta})$,  with $\|x\|_{V^{p}([0,T];\cB_{\alpha+i\theta})}$, $\|y\|_{V^{p}([0,T];\cB_{\alpha+i\theta})}<K$, $i=0,1$. 
Then we have
\begin{align*}
\|f(x)\|_{V^{p}([0,T];\cL(\mathbb{R}^d,\mathcal{B}_{\alpha+(i+2)\theta+\varepsilon}))}&\leq \|f\|_{C_b^2}\|x\|_{V^{p}([0,T];\mathcal{B}_{\alpha+i\theta})}, \\
\|f(x)-f(y)\|_{V^{p}([0,T];\cL(\mathbb{R}^d,\mathcal{B}_{\alpha+(i+2)\theta+\varepsilon}))}&\leq C_{p,K} \|f\|_{C_b^2} (|x_{0}-y_{0}|_{\alpha+i\theta}+\|x-y\|_{V^{p}([0,T];\mathcal{B}_{\alpha+i\theta})}).
\end{align*}
\end{lemma}
\begin{proof}
Denote $\|f\|_{C_b^2}:=\max\limits_{i=0,1}\|f\|_{C^2(\cB_{\alpha+i\theta},\mathcal{L}(\mathbb{R}^d, \cB_{\alpha+(i+2)\theta+\varepsilon}))}$.
We have that 
\begin{align*}
\|f(x)\|_{V^{p}([0,T]; \cL(\mathbb{R}^d,\mathcal{B}_{\alpha+(i+2)\theta+\varepsilon}))}&=
(\sup\limits_{\mathcal{P}}\sum_{[s,t]\in \mathcal{P}}|f(x_{t})-f(x_{s})|^{p}_{\cL(\mathbb{R}^d,\cB_{\alpha+(i+2)\theta+\varepsilon})})^{\frac{1}{p}}\\
&\leq 
(\sup\limits_{\mathcal{P}}\sum_{[s,t]\in \mathcal{P}}\|f\|^p_{C_b^2}|x_{t}-x_{s}|^{p}_{\alpha+i\theta})^{\frac{1}{p}}\\
&\leq 
\|f\|_{C_b^2}\|x\|_{V^{p}([0,T];\mathcal{B}_{\alpha+i\theta})}\\
&\leq K\|f\|_{C_b^2}.
\end{align*}
Using Taylor's formula, 
we derive 
\begin{align*}
&\|f(x)-f(y)\|_{V^{p}([0,T];\cL(\mathbb{R}^d,\mathcal{B}_{\alpha+(i+2)\theta+\varepsilon}))}\\
&=(\sup\limits_{\mathcal{P}}\sum_{[s,t]\in \mathcal{P}}|f(x)_{t,s}-f(y)_{t,s}|^p_{\cL(\mathbb{R}^d,\cB_{\alpha+(i+2)\theta+\varepsilon})})^{\frac{1}{p}}\\
&\leq (\sup\limits_{\mathcal{P}}\sum_{[s,t]\in \mathcal{P}}  |\int_{0}^{1}Df(y_{t}
+r(x_{t}-y_{t}))\textup{d}r (x_{t}-y_{t})
-\int_{0}^{1}Df(y_{t}
+r(x_{t}-y_{t}))\textup{d}r (x_{s}-y_{s})\\
&~~+\int_{0}^{1}Df(y_{t}
+r(x_{t}-y_{t}))\textup{d}r (x_{s}-y_{s})-\int_{0}^{1}Df(y_{s}
+r(x_{s}-y_{s}))\textup{d}r (x_{s}-y_{s})
|^{p}_{\cL(\mathbb{R}^d,\cB_{\alpha+i\theta})})^{\frac{1}{p}}\\
&\leq C_{p}\|f\|_{C_b^2}  (\sup\limits_{\mathcal{P}}\sum_{[s,t]\in \mathcal{P}}|(x-y)_{t,s}|_{\alpha+i\theta}^{p}+(|x_{t,s}|^{p}_{\alpha+i\theta}+|y_{t,s}|^{p}_{\alpha+i\theta})\|x-y\|^{p}_{V^{\infty}([0,T];\mathcal{B}_{\alpha+i\theta})})^{\frac{1}{p}}\\
&\leq C_{p,K}\|f\|_{C_b^2}(\|x-y\|_{V^{p}([0,T];\mathcal{B}_{\alpha+i\theta})}+|x_{0}-y_{0}|_{\alpha+i\theta}),
\end{align*}
where in the last line we used that 
\[ \|x-y\|_{V^{\infty}([0,T];\mathcal{B}_{\alpha+i\theta})} \leq (\|x-y\|_{V^{p}([0,T];\mathcal{B}_{\alpha+i\theta})}+|x_{0}-y_{0}|_{\alpha+i\theta}). \]
\end{proof}
\begin{lemma}\label{local-2p}
Let $\alpha\in \mathbb{R}$, $\varepsilon>0$, $p\in [1,2)$, 
$\theta\in [\frac{1}{p}, 1]$, $p_1\in (p, 2)$. 
Suppose $X\in V^{p}([0,T]; \mathbb{R}^d)$,  is  right-continuous at $t=0$, $X_{0}=0$ and  $f(\cdot)\in C_b^{2}(\cB_{\alpha+i\theta},\mathcal{L}(\mathbb{R}^d, \mathcal{B}_{\alpha+(i+2)\theta+\varepsilon}))$, $i=0,1$, $Y_{0}\in \mathcal{B}_{\alpha+2\theta}$. Then there exists a time horizon $\rho\in (0,T]$ and a unique local solution 
$Y\in V^{p,p_{1}}_{\alpha,\alpha+\theta}([0,\rho])$,   satisfying 
\begin{align*}
Y_{t}=S_{t,0}Y_{0}+\int_{0}^{t}S_{t,r}f(Y_{r})\textup{d}X_{r}, \quad t\in [0,\rho].
\end{align*}
\end{lemma}
\begin{proof}
Define a mapping $\mathcal{J}$ by
$$
(\mathcal{J}Y)
_{t}:=S_{t,0}Y_{0}+\int_{0}^{t}S_{t,r}f(Y_{r})\textup{d}X_{r},\quad t\in [0,\rho].$$
Rather than solving the equation directly in the target space $V^{p,p_1}_{\alpha, \alpha+\theta}([0,\rho])$, we work in a larger space and prove that $\mathcal{J}$ maps a suitable ball into itself and is contractive on that ball.
Fix $p'$ and $p_{1}'$, satisfying $p<p'<p_1<p_1'< 2$. Let $0<\rho<1$ and
denote 
$$
\mathscr{B}_{\rho}
:=\{Y\in V^{p',p'_{1}}_{\alpha,\alpha+\theta}([0,\rho]) ~~| ~~\|Y\|_{V^{p', p'_{1}}_{\alpha, \alpha+\theta}([0,\rho])}\leq 1+|Y_{0}|_{\alpha+2\theta}\}.$$

In what follows, we prove that $\mathcal{J}$ is a contractive self-mapping when $\rho$ is chosen sufficiently small. For $0\leq s<t\leq \rho$,
denote
\begin{align*}
Z_{t,s}:=(\mathcal{J}Y)_{t}-(\mathcal{J}Y)_{s}
:=I_{1}+I_{2}+I_{3}+I_{4},
\end{align*}
where 
\begin{align*}
I_{1}&=S_{t,0}Y_{0}-S_{s,0}Y_{0},\\
I_{2}&=\int_{s}^{t}S_{t,r}f(Y_{r})\textup{d}X_{r}-S_{t,s}f(Y_{s})\delta X_{t,s},\\
I_{3}&=S_{t,s}f(Y_{s})\delta X_{t,s},\\
I_{4}&=(S_{t,s}-\textup{id})\int_{0}^{s}S_{s,r}f(Y_{r})\textup{d}X_{r}.
\end{align*}
Firstly, we prove that  $Z\in V^{p'}([0,\rho]; \mathcal{B}_{\alpha})$. To this aim, from  \eqref{smooth}, we have
\begin{align*}
|I_{1}|_{\alpha}&=|(S_{t,0}-S_{s,0})Y_{0}|_{\alpha}\\
&=|(S_{t,s}-\textup{id})S_{s,0}Y_{0}|_{\alpha}\\
&\lesssim  |S_{s,0}Y_{0}|_{\alpha+\theta}|t-s|^{\theta}\\
&\lesssim |t-s|^{\theta}|Y_{0}|_{\alpha+2\theta}.
\end{align*}
In view of Lemma \ref{YOU01} and Lemma \ref{fff1}, we have 
\begin{align*}
|I_{2}|_{\alpha}&\lesssim \|X\|_{V^{p'}((s,t];\mathbb{R}^d)}
\|f(Y)\|_{V^{p'}([s,t);\cL(\mathbb{R}^d,\cB_{\alpha}))}\\
&~~+(\|f(Y_{0})\|_{\cL(\mathbb{R}^d, \cB_{\alpha+\theta})}+\|f(Y)\|_{V^{p_{1}'}([0,s];\cL(\mathbb{R}^d,\cB_{\alpha+\theta}))}) \|X\|_{V^{p'}((s,t];\mathbb{R}^d)}|t-s|^{\theta}\\
&\lesssim  \|f\|_{C_b^2}\|X\|_{V^{p'}((s,t];\mathbb{R}^d)}\|Y\|_{V^{p'}([s,t);\cB_{\alpha})}\\
&~~+(\|f\|_{C_b^2}+\|f\|_{C_b^2}\|Y\|_{V^{p_{1}'}([0,s];\cB_{\alpha+\theta})})\|X\|_{V^{p'}((s,t];\mathbb{R}^d)}|t-s|^{\theta}\\
&\lesssim  \|f\|_{C_b^2}\|X\|_{V^{p'}((s,t];\mathbb{R}^d)}[(1+\|Y\|_{V^{p_{1}'}([0,s];\cB_{\alpha+\theta})})|t-s|^{\theta}+\|Y\|_{V^{p'}([s,t);\cB_{\alpha})}].
\end{align*}
\begin{align*}
|I_{3}|_{\alpha}\lesssim \|f\|_{C_b^2}\|X\|_{V^{p'}([s,t];\mathbb{R}^d)}.
\end{align*}
Similar to the proof of $I_{1}$, we have 
\begin{align*}
|I_{4}|_{\alpha}&\lesssim [|\int_{0}^{s}S_{s,r}f(Y_{r})\textup{d}X_{r}-S_{s,0}f(Y_{0})X_{s}|_{\alpha+\theta}+|S_{s,0}f(Y_{0})X_{s}|_{\alpha+\theta}]|t-s|^{\theta}, 
\end{align*}
where 
\begin{align*}
&|\int_{0}^{s}S_{s,r}f(Y_{r})\textup{d}X_{r}-S_{s,0}f(Y_{0})X_{s}|_{\alpha+\theta}\\
\lesssim &~~ \|X\|_{V^{p'}([0,s];\mathbb{R}^d)}\|f(Y)\|_{V^{p'}([0,s);\cL(\mathbb{R}^d, \cB_{\alpha+\theta}))}\\
~~+ &~~[\|f(Y_{0})\|_{\cL(\mathbb{R}^d,\cB_{\alpha+2\theta})}+\|f(Y)\|_{V^{p'_{1}}([0,s];\cL(\mathbb{R}^d,\cB_{\alpha+2\theta}))}]\|X\|_{V^{p'}([0,s];\mathbb{R}^d)}s^{\theta}\\
\lesssim &~~ \|f\|_{C_b^2}\|X\|_{V^{p'}([0,s];\mathbb{R}^d)}\|Y\|_{V^{p'}([0,s);\cB_{\alpha})}\\
~~+&~~(\|f\|_{C_b^2}+\|f\|_{C_b^2}\|Y\|_{V^{p_{1}'}([0,s];\cB_{\alpha+\theta})})\|X\|_{V^{p'}([0,s];\mathbb{R}^d)}s^{\theta}\\
\lesssim & ~~\|f\|_{C_b^2}\|X\|_{V^{p'}([0,s];\mathbb{R}^d)}[(1+\|Y\|_{V^{p_{1}'}([0,s];\cB_{\alpha+\theta})})s^{\theta}+\|Y\|_{V^{p'}([0,s);\cB_{\alpha})}]
\end{align*}
and 
\begin{align*}
|S_{s,0}f(Y_{0})X_{s}|_{\alpha+\theta}&\lesssim \|X\|_{V^{p'}([0,s];\mathbb{R}^d)}\|f\|_{C_b^2}.
\end{align*}
Thus,
\begin{equation}
\begin{split}\label{Z1}
&\|Z\|^{p'}_{V^{p'}([0,\rho];\cB_{\alpha})}
=\sup\limits_{\mathcal{P}}\sum_{[s,t]\in \mathcal{P}}|Z_{t,s}|_{\alpha}^{p'}\\
&\lesssim \sup\limits_{\mathcal{P}}\sum_{[s,t]\in \mathcal{P}}\big[|t-s|^{\theta p'}|Y_{0}|^{p'}_{\alpha+2\theta}
+\|f\|^{p'}_{C_b^2}\|X\|^{p'}_{V^{p'}((s,t];\mathbb{R}^d)}[(1+\|Y\|_{V^{p_{1}'}([0,s];\cB_{\alpha+\theta})})|t-s|^{\theta}\\
&~~+\|Y\|_{V^{p'}([s,t);\cB_{\alpha})}]^{p'}
+\|f\|_{C_b^2}^{p'}\|X\|^{p'}_{V^{p'}([s,t];\mathbb{R}^d)}\\
&~~+|t-s|^{\theta p'}\|f\|^{p'}_{C_b^2}\|X\|^{p'}_{V^{p'}([0,s];\mathbb{R}^d)}[1+\|Y\|_{V^{p'}([0,s);\cB_{\alpha})}
+(1+\|Y\|_{V^{p_{1}'}([0,s];\cB_{\alpha+\theta})})s^{\theta}]^{p'}\\
&\lesssim (\rho^{\theta p'}+\|X\|^{p'}_{V^{p'}([0,\rho];\mathbb{R}^d)})(1+|Y_{0}|_{\alpha+2\theta})^{p'},
\end{split}
\end{equation}
Next, we prove that $Z\in V^{p'_{1}}([0,\rho]; \mathcal{B}_{\alpha+\theta})$.
Applying \eqref{smooth}, we have 
\begin{align*}
|I_{1}|_{\alpha+\theta}&=|(S_{t,0}-S_{s,0})Y_{0}|_{\alpha+\theta}\\
&=|(S_{t,s}-\textup{id})S_{s,0}Y_{0}|_{\alpha+\theta}\\
&\lesssim   |S_{s,0}Y_{0}|_{\alpha+2\theta}|t-s|^{\theta}\\
&\lesssim |t-s|^{\theta}|Y_{0}|_{\alpha+2\theta}.
\end{align*}
From Lemma \ref{YOU01} and Lemma \ref{fff1}, we get
\begin{align*}
|I_{2}|_{\alpha+\theta}&\lesssim  \|X\|_{V^{p_{1}'}((s,t];\mathbb{R}^d)}
\|f(Y)\|_{V^{p'}([s,t);\cL(\mathbb{R}^d,\cB_{\alpha+\theta}))}\\
&~~+(\|f(Y_{0})\|_{\cL(\mathbb{R}^d, \cB_{\alpha+2\theta})}+\|f(Y)\|_{V^{p_{1}'}([0,s];\cL(\mathbb{R}^d,\cB_{\alpha+2\theta}))}) \|X\|_{V^{p_{1}'}((s,t];\mathbb{R}^d)}|t-s|^{\theta}\\
&\lesssim \|f\|_{C_b^2}\|X\|_{V^{p_{1}'}((s,t];\mathbb{R}^d)}\|Y\|_{V^{p'}([s,t);\cB_{\alpha})}\\
&~~+(\|f\|_{C_b^2}+\|f\|_{C_b^2}\|Y\|_{V^{p_{1}'}([0,s];\cB_{\alpha+\theta})})\|X\|_{V^{p_{1}'}((s,t];\mathbb{R}^d)}|t-s|^{\theta}\\
&\lesssim \|f\|_{C_b^2}\|X\|_{V^{p_{1}'}((s,t];\mathbb{R}^d)}[(1+\|Y\|_{V^{p_1'}([0,s];\cB_{\alpha+\theta})})|t-s|^{\theta}+\|Y\|_{V^{p'}([s,t);\cB_{\alpha})}].
\end{align*}
\begin{align*}
|I_{3}|_{\alpha+\theta}\lesssim \|f\|_{C_b^2}\|X\|_{V^{p_{1}'}([s,t];\mathbb{R}^d)}.
\end{align*}
\begin{align*}
|I_{4}|_{\alpha+\theta}&\lesssim (|\int_{0}^{s}S_{s,r}f(Y_{r})\textup{d}X_{r}-S_{s,0}f(Y_{0})X_{s}|_{\alpha+2\theta}+|S_{s,0}f(Y_{0})X_{s}|_{\alpha+2\theta})|t-s|^{\theta}, 
\end{align*}
where 
\begin{align*}
&|\int_{0}^{s}S_{s,r}f(Y_{r})\textup{d}X_{r}-S_{s,0}f(Y_{0})X_{s}|_{\alpha+2\theta}\\
\lesssim & ~~\|X\|_{V^{p_{1}'}([0,s];\mathbb{R}^d)}\|f(Y)\|_{V^{p'}([0,s);\cL(\mathbb{R}^d, \cB_{\alpha+2\theta}))}\\
~~+&~~[\|f(Y_{0})\|_{\cL(\mathbb{R}^d,\cB_{\alpha+3\theta})}+\|f(Y)\|_{V^{p_{1}'}([0,s];\cL(\mathbb{R}^d,\cB_{\alpha+3\theta}))}]\|X\|_{V^{p'_{1}}([0,s];\mathbb{R}^d)}s^{\theta}\\
\lesssim & ~~\|f\|_{C_b^2}\|X\|_{V^{p_{1}'}([0,s];\mathbb{R}^d)}\|Y\|_{V^{p'}([0,s);\cB_{\alpha})}\\
~~+&~~(\|f\|_{C_b^2}+\|f\|_{C_b^2}\|Y\|_{V^{p_{1}'}([0,s];\cB_{\alpha+\theta})})\|X\|_{V^{p_{1}'}([0,s];\mathbb{R}^d)}s^{\theta}\\
\lesssim &~~ \|f\|_{C_b^2}\|X\|_{V^{p_{1}'}([0,s];\mathbb{R}^d)}[(1+\|Y\|_{V^{p_{1}'}([0,s];\cB_{\alpha+\theta})})s^{\theta}+\|Y\|_{V^{p'}([0,s);\cB_{\alpha})}].
\end{align*}
and 
\begin{align*}
|S_{s,0}f(Y_{0})X_{s}|_{\alpha+2\theta}&\lesssim \|X\|_{V^{p_{1}'}([0,s];\mathbb{R}^d)}\|f\|_{C_b^2}.
\end{align*}
Similar to the proof in \eqref{Z1}, we get
\begin{equation}\label{Z2}
\begin{split}
&\|Z\|^{p_{1}'}_{V^{p'_{1}}([0,\rho];\cB_{\alpha+\theta})}
=\sup\limits_{\mathcal{P}}\sum_{(s,t)\in \mathcal{P}}|Z_{t,s}|_{\alpha+\theta}^{p_{1}'}\\
&\lesssim \sup\limits_{\mathcal{P}}\sum_{[s,t]\in \mathcal{P}}\big[|t-s|^{\theta p_{1}'}|Y_{0}|^{p_{1}'}_{\alpha+2\theta}
+\|f\|^{p_{1}'}_{C_b^2}\|X\|^{p_{1}'}_{V^{p_{1}'}((s,t];\mathbb{R}^d)}[(1+\|Y\|_{V^{p_{1}'}([0,s];\cB_{\alpha+\theta})})|t-s|^{\theta}\\
&~~
+\|Y\|_{V^{p'}([s,t);\cB_{\alpha})}]^{p_{1}'}
+\|f\|_{C_b^2}^{p_{1}'}\|X\|^{p_{1}'}_{V^{p_{1}'}([s,t];\mathbb{R}^d)}\\
&~~+|t-s|^{\theta p'_{1}}\big(\|f\|^{p_{1}'}_{C_b^2}\|X\|^{p_{1}'}_{V^{p_{1}'}([0,s];\mathbb{R}^d)}[(1+\|Y\|_{V^{p_{1}'}([0,s];\cB_{\alpha+\theta})})s^{\theta}
+\|Y\|_{V^{p'}([0,s);\cB_{\alpha})}+1]^{p_{1}'}
\big]\\
&\lesssim  (\rho^{\theta p'_{1}}+\|X\|^{p'_{1}}_{V^{p'_{1}}([0,\rho];\mathbb{R}^d)})(1+|Y_{0}|_{\alpha+2\theta})^{p'_{1}}.
\end{split}
\end{equation}
Note that for any $\tilde p>p$, $$\|X\|^{\tilde p}_{V^{\tilde p}([0,\rho]; \mathbb{R}^d)}\leq 
\|X\|^{p}_{V^{p}([0,\rho];\mathbb{R}^d)}Osc(X; [0,\rho])^{\tilde p-p}.$$
Due to the right continuity of $X$ at $t=0$, we can choose $\rho$ sufficiently small, such that 
\begin{align}\label{rhoC}
C\big(\rho^{\theta}+\|X\|^{\frac{p}{p'}}_{V^{p}([0,\rho];\mathbb{R}^d)}Osc(X;[0,\rho])^{\frac{p'-p}{p'}}+\|X\|^{\frac{p}{p'_{1}}}_{V^{p}([0,\rho];\mathbb{R}^d)}Osc(X;[0,\rho])^{\frac{p'_{1}-p}{p_{1}'}}\big)\leq 1, 
\end{align}
where $C$ depends on $S, f, p', p_1', \theta$.
Combining with \eqref{Z1}, \eqref{Z2} and \eqref{rhoC}, we derive 
\begin{align*}
\|Z\|_{V^{p',p_{1}'}_{\alpha,\alpha+\theta}([0,\rho])}\leq 1+|Y_0|_{\alpha+2\theta}.
\end{align*}

It remains to show the contraction property of the mapping $\mathcal{J}$. For any $Y_{1}$, $Y_{2}\in \mathscr{B}_{\rho}$, $Y^{1}_{0}=Y_{0}^{2}$, we have
\begin{align*}
|Z^{1}_{t,s}-Z_{t,s}^{2}|_{\alpha}&=|\mathcal{J}(Y^{1})_{t,s}-\mathcal{J}(Y^{2})_{t,s}|_{\alpha}\\
&=|\int_{s}^{t}S_{t,r}f(Y^{1}_{r})-f(Y^{2}_{r})\textup{d}X_{r}+(S_{t,s}-\textup{id})\int_{0}^{s}S_{s,r}f(Y^{1}_{r})-f(Y^{2}_{r})\textup{d}X_{r}|_{\alpha}\\
&\leq |\int_{s}^{t}S_{t,r}(f(Y_{r}^{1})-f(Y_{r}^{2}))\textup{d}X_{r}-S_{t,s}(f(Y^{1}_{s})-f(Y_{s}^{2}))\delta X_{t,s}|_{\alpha}\\
&~~+|S_{t,s}(f(Y^{1}_{s})-f(Y^{2}_{s}))\delta X_{t,s}|_{\alpha}+|(S_{t,s}-\textup{id})\int_{0}^{s}S_{s,r}(f(Y^{1}_{r})-f(Y^{2}_{r}))\textup{d}X_{r}|_{\alpha}\\
&:=I_{5}+I_{6}+I_{7}.
\end{align*}
Applying \eqref{smooth}, Lemma \ref{YOU01}, and Lemma \ref{fff1} gives the following  estimates for 
$I_5, I_6, I_7$ respectively.  
\begin{align*}
I_{5}&\lesssim \|X\|_{V^{p'}((s,t];\mathbb{R}^d)}\|f(Y^{1})-f(Y^{2})\|_{V^{p'}([s,t);\cL(\mathbb{R}^d,\mathcal{B}_{\alpha}))}\\
&~~+
\|f(Y^{1})-f(Y^{2})\|_{V^{p_{1}'}([0,s];\cL(\mathbb{R}^d,\mathcal{B}_{\alpha+\theta}))}
\|X\|_{V^{p'}((s,t];\mathbb{R}^d)}|t-s|^{\theta}\\
&\lesssim \|X\|_{V^{p'}((s,t];\mathbb{R}^d)} \|f\|_{C_b^2}\|Y^{1}-Y^{2}\|_{V^{p'}([s,t);\mathcal{B}_{\alpha})}\\
&~~+\|f\|_{C_b^2}\|Y^{1}-Y^{2}\|_{V^{p_{1}'}([0,s];\mathcal{B}_{\alpha+\theta})}\|X\|_{V^{p'}((s,t];\mathbb{R}^d)}|t-s|^{\theta}\\
&\lesssim \|X\|_{V^{p'}((s,t];\mathbb{R}^d)} \|f\|_{C_b^2}\big(
\|Y^{1}-Y^{2}\|_{V^{p'}([s,t);\mathcal{B}_{\alpha})}+\|Y^{1}-Y^{2}\|_{V^{p_{1}'}([0,s];\mathcal{B}_{\alpha+\theta})}|t-s|^{\theta}
\big),
\end{align*}
\begin{align*}
I_{6}&\lesssim \|f\|_{C_b^2}\|Y_{s}^{1}-Y_{s}^{2}\|_{\alpha}\|X\|_{V^{p'}((s,t];\mathbb{R}^d)}\\
&\lesssim \|f\|_{C_b^2}\|Y^{1}-Y^{2}\|_{V^{p'}([0,s];\mathcal{B}_{\alpha})}\|X\|_{V^{p'}((s,t];\mathbb{R}^d)},
\end{align*}
\begin{align*}
I_{7}&\lesssim |\int_{0}^{s}S_{s,r}(f(Y^{1}_{r})-f(Y^{2}_{r}))\textup{d}X_{r}|_{\alpha+\theta}|t-s|^{\theta}\\
&\lesssim [\|X\|_{V^{p'}([0,s];\mathbb{R}^d)}\|f(Y^{1})-f(Y^{2}))\|_{V^{p'}([0,s);\cL(\mathbb{R}^d,\mathcal{B}_{\alpha+\theta}))}\\
&~~+\|f(Y^{1})-f(Y^{2})\|_{V^{p_{1}'}([0,s];\cL(\mathbb{R}^d,\mathcal{B}_{\alpha+2\theta}))}\|X\|_{V^{p'}([0,s];\mathbb{R}^d)}s^{\theta}]|t-s|^{\theta}\\
&\lesssim \big[\|X\|_{V^{p'}([0,s];\mathbb{R}^d)}\|f\|_{C_b^2
}\big(\|Y^{1}-Y^{2}\|_{V^{p'}([0,s);\mathcal{B}_{\alpha})}+\|Y^{1}-Y^{2}\|_{V^{p_{1}'}([0,s];\mathcal{B}_{\alpha+\theta})}s^{\theta}\big)\big]|t-s|^{\theta}.
\end{align*}
Thus, we derive 
\begin{equation}\label{local-self}
\begin{split}
&\sup\limits_{\mathcal{P}}\sum_{[s,t]\in \mathcal{P}}|Z^{1}_{t,s}-Z^{2}_{t,s}|^{p'}_{\alpha}  \\
&\lesssim \sup\limits_{\mathcal{P}}\sum_{[s,t]\in \mathcal{P}} \{\|X\|^{p'}_{V^{p'}((s,t];\mathbb{R}^d)} \|f\|^{p'}_{C_b^2}(\|Y^{1}-Y^{2}\|_{V^{p'}([s,t);\mathcal{B}_{\alpha})}
+\|Y^{1}-Y^{2}\|_{V^{p_{1}'}([0,s];\mathcal{B}_{\alpha+\theta})}|t-s|^{\theta})^{p'} \\
&~~+\|X\|^{p'}_{V^{p'}((s,t];\mathbb{R}^d)} \|f\|^{p'}_{C_b^2}\|Y^{1}-Y^{2}\|^{p'}_{V^{p_{1}'}([0,s];\mathcal{B}_{\alpha+\theta})}
+\|X\|^{p'}_{V^{p'}([0,s];\mathbb{R}^d)} \|f\|^{p'}_{C^2}|t-s|^{\theta p'} \\
&~~(\|Y^{1}-Y^{2}\|_{V^{p'}([0,s);\mathcal{B}_{\alpha})}+\|Y^{1}-Y^{2}\|_{V^{p_{1}'}([0,s];\mathcal{B}_{\alpha+\theta})}s^{\theta})^{p'}
\} \\
&\lesssim \|f\|_{C_b^2}^{p'}(\|X\|^{p'}_{V^{p'}([0,\rho];\mathbb{R}^d)} +\|X\|^{p'}_{V^{p'}([0,T];\mathbb{R}^d)}\rho^{\theta p'}))\|Y^{1}-Y^{2}\|_{V_{\alpha,\alpha+\theta}^{p',p_{1}'}([0,\rho])}^{p'}  \\
&\lesssim (\|X\|^{p'}_{V^{p'}([0,\rho];\mathbb{R}^d)} +\rho^{\theta p'}))\|Y^{1}-Y^{2}\|_{V_{\alpha,\alpha+\theta}^{p',p_{1}'}([0,\rho])}^{p'} \\
&\lesssim (\|X\|^{p}_{V^{p}([0,\rho];\mathbb{R}^d)}Osc(X;[0,\rho])^{p'-p} +\rho^{\theta p'}))\|Y^{1}-Y^{2}\|_{V_{\alpha,\alpha+\theta}^{p',p_{1}'}([0,\rho])}^{p_{1}'}.
\end{split}
\end{equation}
Similarly, we have 
\begin{align*}
\sup\limits_{\mathcal{P}}\sum_{[s,t]\in \mathcal{P}}|Z^{1}_{t,s}-Z^{2}_{t,s}|^{p_{1}'}_{\alpha+\theta}
\lesssim (\|X\|^{p}_{V^{p}([0,\rho];\mathbb{R}^d)}Osc(X;[0,\rho])^{p_{1}'-p}  +\rho^{\theta p_{1}'}))\|Y^{1}-Y^{2}\|_{V_{\alpha,\alpha+\theta}^{p',p_{1}'}([0,\rho])}^{p_{1}'}.
\end{align*}
Due to the right continuity of $X$ at $t=0$, one can choose $\rho$ sufficiently small, such that 
\begin{align*}
C(\|X\|^{p/p'}_{V^{p}([0,\rho];\mathbb{R}^d)}Osc(X;[0,\rho])^{\frac{p'-p}{p'}}+\|X\|^{p/p_{1}'}_{V^{p}([0,\rho];\mathbb{R}^d)}Osc(X;[0,\rho])^{\frac{p_{1}'-p}{p_{1}'}}+\rho^{\theta})<1.
\end{align*}
Thus we derive the contraction property of $\mathcal{J}$, i.e.
\begin{align*}
\|Z^{1}-Z^{2}\|_{V^{p', p_{1}'}_{\alpha, \alpha+\theta}([0,\rho])}<\|Y^{1}-Y^{2}\|_{V^{p', p_{1}'}_{\alpha, \alpha+\theta}([0,\rho])}.
\end{align*}
By applying Banach's fixed point theorem, we conclude that there exists a unique local solution $Y\in V^{p',p_{1}'}_{\alpha,\alpha+\theta}([0,\rho])$. 
By similar computation as \eqref{Z1} and \eqref{Z2},  we have 
\begin{align*}
\|Y\|_{V^{p,p_{1}}_{\alpha,\alpha+\theta}([0,\rho])}\lesssim (1+\|X\|_{V^{p}([0,\rho];\mathbb{R}^d)})(1+|Y_{0}|_{\alpha+2\theta}),
\end{align*}   
which yields the existence and uniqueness of the solution in $V^{p,p_1}_{\alpha,\alpha+\theta}([0,\rho])$. 
\end{proof}
\begin{remark}\label{localpp}
Let $X\in V^p ([0,T];\mathbb{R}^d)$, $p\in [1,2)$. For $t\in[0,t_0]$, consider
\begin{align}\label{eqYY}
Y_t= S_{t,0} Y_0+ \int_0^t S_{t,r} f(Y_r)\textup{d}X_r, 
\end{align}
with initial value $Y_0\in \cB_{\alpha+2\theta}$. Suppose that  $Y$
is a 
solution of \eqref{eqYY}.  By the denifition of the integral with jumps at the initial point (see Lemma \ref{llrr} (1)), we have 
\begin{align*}
\lim\limits_{t\downarrow 0} \int_0^t S_{t,r} f(Y_r)\textup{d} X_r= S_{t,0} f(Y_0)\Delta_0^+ X. 
\end{align*}
Define 
$Y_{0+}= Y_0+ f(Y_0) \Delta_0^+ X$, 
$$
\tilde Y_t:=\left\{
\begin{array}{ll}
Y_{0+}, & t=0 \\
Y_t,&  t>0, 
\end{array}
\right.
$$
and 
$$
\tilde X_t:=\left\{
\begin{array}{ll}
X_{0+}, & t=0 \\
X_t,&  t>0.
\end{array}
\right.
$$
It is clear that $\tilde X$
 is right-continuous at 0. It holds for every $t\in (0,t_0]$, 
 $$
 \int_0^t S_{t,r} f(Y_r)\textup{d}X_r= S_{t,0} f(Y_0) \Delta_0^+ X+ \int_0^t S_{t,r} f(\tilde Y_r)\textup{d}\tilde X_r.
 $$
 Then,
 \begin{align*}
 \tilde Y_t=Y_t&=S_{t,0}Y_0+ \int_0^t S_{t,r} f(Y_r)\textup{d}X_r\\
 &= S_{t,0}Y_0+  S_{t,0} f(Y_0) \Delta_0^+ X+ \int_0^t S_{t,r} f(\tilde Y_r)\textup{d}\tilde X_r\\
 &= S_{t,0} Y_{0+}+\int_0^t S_{t,r} f(\tilde Y_r)\textup{d}\tilde X_r. 
 \end{align*}
 The identity also holds for $t=0$. Hence $\tilde Y$ solves 
\begin{align}\label{eqtY}
\tilde Y_t =S_{t,0} Y_{0+}+ \int_0^t S_{t,r} f(\tilde Y_r)\textup{d}\tilde X_r, \quad t\in [0,t_0]. 
\end{align}
Conversly, suppose that $\tilde Y$  solves \eqref{eqtY}, where $
Y_{0+}= Y_0+ f(Y_0)\Delta_0^+ X.$ 
Define  
$$
Y_t:=\left\{
\begin{array}{ll}
Y_{0}, & t=0 \\
\tilde Y_t,&  t>0, 
\end{array}
\right.
$$
Then 
\begin{align*}
S_{t,0} Y_0+\int_0^t S_{t,r} f(Y_r) \textup{d}X_r &= S_{t,0} Y_0+  S_{t,0} f(Y_0) \Delta_0^+ X+ \int_0^t S_{t,r} f(\tilde Y_r)\textup{d}\tilde X_r\\
&= S_{t,0} Y_{0+} + \int_0^t S_{t,r} f(\tilde Y_r)\textup{d}\tilde X_r\\
&= \tilde Y_t =Y_t. 
\end{align*}
Therefore $Y$
 solves 
 \eqref{eqYY}. 
The preceding local result requires initial data in
$\mathcal B_{\alpha+2\theta}$. At later times, however, the solution
is only known to belong to $\mathcal B_{\alpha+\theta}$. Thus, the
local argument cannot in general be restarted to obtain a global
solution in $V_{\alpha,\alpha+\theta}^{p,p_1}([0,T])$.
 \end{remark}
\subsection{Global well-posedness}
In this subsection, we  prove the existence and uniqueness of a global solution in space $V^{p,\infty}_{\alpha,\alpha+\theta}$.
Note that 
for every finite $p_1\geq1$, one has $V^{p,p_1}_{\alpha, \alpha+\theta}([0,T]) \hookrightarrow V^{p,\infty}_{\alpha,\alpha+\theta}([0,T])$. In contrast to the last subsection,  working in this larger space not only lowers the initial regularity requirement from 
\(\mathcal B_{\alpha+2\theta}\) to
\(\mathcal B_{\alpha+\theta}\), 
but also reduces the regularity gain
required of \(f\) from \(2\theta+\varepsilon\) to
\(\theta+\varepsilon\). Consequently, 
 the corresponding local arguments can be iterated to obtain a global result. 

\begin{lemma}\label{fff2}
Suppose $f(\cdot)\in C_b^2(\cB_{\alpha+i\theta}, \cL(\mathbb{R}^d,\mathcal{B}_{\alpha+(i+1)\theta+\varepsilon}))$, for some $\alpha,\theta  \in \mathbb{R}$, $\varepsilon>0$ and let $x, y\in V^{p,\infty}_{\alpha+i\theta, \alpha+i\theta}([0,T])$, $i=0,1$, with $\|x\|_{V^{p,\infty}_{\alpha+i\theta, \alpha+i\theta}([0,T])}$, $\|y\|_{V^{p,\infty}_{\alpha+i\theta, \alpha+i\theta}([0,T])}<K$. Then one has
\begin{align*}
\|f(x)\|_{V^{p}([0,T];\cL(\mathbb{R}^d,\mathcal{B}_{\alpha+(i+1)\theta+\varepsilon}))}&\leq \|f\|_{C_b^2}\|x\|_{V^{p}([0,T];\mathcal{B}_{\alpha+i\theta})}, \\
\|f(x)\|_{V^{\infty}([0,T];\cL(\mathbb{R}^d,\mathcal{B}_{\alpha+(i+1)\theta+\varepsilon}))}&\leq \|f\|_{C_b^2}  \|x\|_{V^{\infty}([0,T]; \mathcal{B}_{\alpha+i\theta})},\\
\|f(x)-f(y)\|_{V^{p}([0,T];\cL(\mathbb{R}^d,\mathcal{B}_{\alpha+(i+1)\theta+\varepsilon}))}&\leq C_{p,K} \|f\|_{C_b^2} (|x_{0}-y_{0}|_{\alpha+i\theta}+\|x-y\|_{V^{p}([0,T];\mathcal{B}_{\alpha+i\theta})}),\\
\|f(x)-f(y)\|_{V^{\infty}([0,T];\cL(\mathbb{R}^d,\mathcal{B}_{\alpha+(i+1)\theta+\varepsilon}))}&\leq C_{K}\|f\|_{C_b^2}(|x_{0}-y_{0}|_{\alpha+i\theta}+\|x-y\|_{V^{\infty}([0,T];\mathcal{B}_{\alpha+i\theta})}).
\end{align*}
\end{lemma}
\begin{proof}
Denote $\|f\|_{C_b^2}:=\max\limits_{i=0,1}\|f\|_{C^2(\cB_{\alpha+i\theta},\mathcal{L}(\mathbb{R}^d, \cB_{\alpha+(i+1)\theta+\varepsilon}))}$.
For $p\in [1,\infty), $  the estimates in $V^p$ are similar to that in Lemma \ref{fff1}. 
When $p=\infty$, we have 
\begin{align*}
\|f(x)\|_{V^{\infty}([0,T];\cL(\mathbb{R}^d,\mathcal{B}_{\alpha+(i+1)\theta+\varepsilon}))}&=\sup\limits_{0\leq t\leq T}\|f(x_{t})-f(x_{0})\|_{\cL(\mathbb{R}^d,\mathcal{B}_{\alpha+(i+1)\theta+\varepsilon})}\\
&\leq \|f\|_{C_b^2} \sup\limits_{0\leq t\leq T}|x_{t}-x_{0}|_{\alpha+i\theta}\\
&\leq  \|f\|_{C_b^2}  \|x\|_{V^{\infty}([0,T]; \mathcal{B}_{\alpha+i\theta})}\\
&\leq K\|f\|_{C_b^2}.
\end{align*}
Using Taylor's formula, we derive
\begin{align*}
&\|f(x)-f(y)\|_{V^{\infty}([0,T];\cL(\mathbb{R}^d,\mathcal{B}_{\alpha+(i+1)\theta+\varepsilon}))}\\
&=\sup\limits_{0\leq t\leq T}|f(x_{t})-f(y_{t})-f(x_{0})+f(y_{0})|_{\alpha+(i+1)\theta+\varepsilon}\\
&\leq  \sup\limits_{0\leq t\leq T} |\int_{0}^{1}Df(y_{t}+r(x_{t}-y_{t}))\textup{d}r(x_{t}-y_{t})-\int_{0}^{1}Df(y_{t}+r(x_{t}-y_{t}))\textup{d}r(x_{0}-y_{0})|_{\alpha+i\theta}\\
&~~+|\int_{0}^{1}Df(y_{t}+r(x_{t}-y_{t}))\textup{d}r(x_{0}-y_{0})-\int_{0}^{1}Df(y_{0}+r(x_{0}-y_{0}))\textup{d}r(x_{0}-y_{0})|_{\alpha+i\theta}\\
&\leq \|f\|_{C_b^2}\big[\|x-y\|_{V^{\infty}([0,T];\mathcal{B}_{\alpha+i\theta})}+|x_{0}-y_{0}|_{\alpha+i\theta}\\
&~~\times(\|y\|_{V^{\infty}([0,T];\mathcal{B}_{\alpha+i\theta})}+\|x\|_{V^{\infty}([0,T];\mathcal{B}_{\alpha+i\theta})})\big]\\
&\leq C_{K}\|f\|_{C_b^2}(\|x-y\|_{V^{\infty}([0,T];\mathcal{B}_{\alpha+i\theta})}+|x_{0}-y_{0}|_{\alpha+i\theta}).
\end{align*}
\end{proof}
\begin{lemma}[Local well-posedness of \eqref{main}]\label{locyoung}
Let $\alpha\in \mathbb{R}$, $\varepsilon>0$, $p\in [1,2)$, $\theta\in [\frac{1}{p}, 1]$. 
Suppose $X\in V^{p}([0,T]; \mathbb{R}^d)$,  is  right-continuous at $t=0$, $X_{0}=0$ and $f(\cdot)\in C_b^2 (\cB_{\alpha+i\theta}, \cL(\mathbb{R}^d,\mathcal{B}_{\alpha+(i+1)\theta+\varepsilon}))$, $i=0,1,$ $Y_0\in \mathcal{B}_{\alpha+\theta}$. Then there exists a time horizon $\rho\in (0,T)$ and a unique local solution 
$Y\in V^{p, \infty}_{\alpha, \alpha+\theta}([0,\rho])$ satisfying 
\begin{align}\label{eqyoung}
Y_{t}=S_{t,0}Y_{0}+\int_{0}^{t}S_{t,r}f(Y_{r})\textup{d}X_{r}, \quad t\in [0,\rho].
\end{align}
\end{lemma}
\begin{proof}
 Recall that by \eqref{smooth}, $|S_{t,s}x|_{\alpha}\leq C_S |x|_\alpha.$
Define a mapping $\mathcal{J}$ by 
$$
(\mathcal{J}Y)
_{t}:=S_{t,0}Y_{0}+\int_{0}^{t}S_{t,r}f(Y_{r})\textup{d}X_{r},\quad t\in [0,\rho].$$
The fixed-point argument is not conducted in $V^{p,\infty}_{\alpha,\alpha+\theta}([0,\rho])$ but in a larger space.
To this aim, fix $p<p'<2$, let $0<\rho<1$, and
denote 
$$
\mathscr{B}_{\rho}(Y_0)
:=\{Y\in V^{p', \infty}_{\alpha, \alpha+\theta}([0,\rho]) ~~| ~~\|Y\|_{V^{p', \infty}_{\alpha, \alpha+\theta}([0,\rho])}\leq (C_S+2)(1+|Y_{0}|_{\alpha+\theta})\}.$$ 
In the following, we prove that $\mathcal{J}$ is a contractive self-mapping on $\mathscr{B}_\rho$ choosing $\rho$ sufficiently small.
For $0<s<t<\rho$,
denote $
Z_{t,s}:=(\mathcal{J}Y)_{t}-(\mathcal{J}Y)_{s}.$
The proof of $Z\in V^{p'}([0,\rho];\cB_\alpha)$ follows
from the similar arguments as those used in the proof of
Lemma~\ref{local-2p}, so we omit the proof
and claim that 
\begin{align*}
\|Z\|^{p'}_{V^{p'}([0,\rho];\mathcal{B}_{\alpha})}
\leq C ( \rho^{\theta p'}+\|X\|^{p'}_{V^{p'}([0,\rho];\mathbb{R}^d)})(|Y_{0}|_{\alpha+\theta}+1)^{p'},
\end{align*}
where $C$ depends on $S, p', \theta, f$.
Note that $\|X\|^{p'}_{V^{p'}([0,\rho]; \mathbb{R}^d)}\leq 
\|X\|^{p}_{V^{p}([0,\rho];\mathbb{R}^d)}Osc(X, [0,\rho])^{p'-p}$. 
Due to the right continuity of $X$ at $t=0$, we can choose $\rho$ sufficiently small, such that 
\begin{align*}
\|Z\|_{V^{p'}([0,\rho];\mathcal{B}_{\alpha})}\leq \frac{1}{2}(1+|Y_{0}|_{\alpha+\theta}).
\end{align*}
Next, we prove that $Z\in V^{\infty} ([0,\rho];\mathcal{B}_{\alpha+\theta})$.
It holds 
\begin{align*}
\sup\limits_{t\in [0,\rho]}|Z_{t}-Z_{0}|_{\alpha+\theta}&=\sup\limits_{t\in [0,\rho]}|S_{t,0}Y_{0}-Y_{0}+\int_{0}^{t}S_{t,r}f(Y_{r})\textup{d}X_{r}|_{\alpha+\theta}\\
&\leq (C_S+1)|Y_{0}|_{\alpha+\theta}+\sup\limits_{t\in[0,\rho]}\{|\int_{0}^{t}S_{t,r}f(Y_{r})\textup{d}X_{r}-S_{t,0}f(Y_{0})X_{t}|_{\alpha+\theta}
+|S_{t,0}f(Y_{0})X_{t}|_{\alpha+\theta}\}\\
&\leq (C_S+1)|Y_{0}|_{\alpha+\theta}+C\sup\limits_{t\in[0,\rho]}
\big(\|X\|_{V^{p'}([0,t];\mathbb{R}^d)}\|Y\|_{V^{p'}([0,t);\cL(\mathbb{R}^d,\mathcal{B}_{\alpha}))}\\
&~~+ (1+\|Y\|_{V^{\infty}
([0,t];\cL(\mathbb{R}^d, \mathcal{B}_{\alpha+\theta})})\|X\|_{V^{p'}([0,t];\mathbb{R}^d)}t^{\theta}\big)
+ C\sup\limits_{t\in [0,\rho]} \|X\|_{V^{p'}([0,t];\mathbb{R}^d)}\\
&\leq (C_S+1) |Y_{0}|_{\alpha+\theta}+C\|X\|_{V^{p'}([0,\rho];\mathbb{R}^d)}(|Y_{0}|_{\alpha+\theta}+1).
\end{align*}
By choosing $\rho$ sufficiently small, we have
\begin{align*}
\|Z\|_{V^{\infty}([0,\rho];\mathcal{B}_{\alpha+\theta})}\leq (C_S+1)|Y_{0}|_{\alpha+\theta}+\frac{1}{2}(1+|Y_{0}|_{\alpha+\theta}).
\end{align*}
Then,  we can prove that $\mathcal{J}$ is a self-mapping by which 
\begin{align*} \|Z\|_{V^{p'}([0,\rho];\mathcal{B}_{\alpha})}+ 
\|Z\|_{V^{\infty}([0,\rho];\mathcal{B}_{\alpha+\theta})}\leq (C_S+2)(1+|Y_{0}|_{\alpha+\theta}).
\end{align*} 

It remains to show that $\mathcal{J}$ is a contraction mapping. For any $Y^{1}, Y^{2}\in \mathscr{B}_{\rho}$, $Y_{0}^{1}=Y_{0}^{2}$, the estimate bewteen $Z^1:=\mathcal{J}(Y^1)$ and $Z^2:=\mathcal{J}(Y^2)$ in $\cB_\alpha$ is similar to \eqref{local-self} in Lemma \ref{local-2p}. We claim that 
\begin{align*}
\sup\limits_{\mathcal{P}}\sum_{[s,t]\in \mathcal{P}}|Z^{1}_{t,s}-Z^{2}_{t,s}|^{p'}_{\alpha}
\lesssim
(\|X\|^{p}_{V^{p}([0,\rho];\mathbb{R}^d)}Osc(X;[0,\rho])^{p'-p} +\rho^{\theta p'})\|Y^{1}-Y^{2}\|_{V_{\alpha,\alpha+\theta}^{p',\infty}([0,\rho])}^{p'}. 
\end{align*}
Besides, we have 
\begin{align*}
\sup\limits_{t\in[0,\rho]}|Z^{1}_{t,0}-Z^{2}_{t,0}|_{\alpha+\theta}&=\sup\limits_{t\in[0,\rho]}|\int_{0}^{t}S_{t,r}(f(Y^{1}_{r})-f(Y^{2}_{r}))\textup{d}X_{r}|_{\alpha+\theta}\\
&\lesssim  \sup\limits_{t\in[0,\rho]}\big(\|X\|_{V^{p'}([0,t];\mathbb{R}^d)}\|f(Y^{1})-f(Y^{2})\|_{V^{p'}([0,t);\cL(\mathbb{R}^d,\mathcal{B}_{\alpha+\theta}))}\\
&~~+\|f(Y^{1})-f(Y^{2})\|_{V^{\infty}([0,t];\cL(\mathbb{R}^d,\mathcal{B}_{\alpha+2\theta}))}\|X\|_{V^{p'}([0,t];\mathbb{R}^d)}t^{\theta}\big)\\
&\lesssim  \|f\|_{C_b^2}\|X\|^{\frac{p}{p'}}_{V^{p}([0,\rho];\mathbb{R}^d)}Osc(X;[0,\rho])^{\frac{p'-p}{p'}}\big[\|Y^{1}-Y^{2}\|_{V^{p'}([0,\rho);
\mathcal{B}_{\alpha})}\\
&~~+\|Y^{1}-Y^{2}\|_{V^{\infty}([0,\rho];\mathcal{B}_{\alpha+\theta})}\rho^{\theta}\big]\\
&\lesssim \|f\|_{C_b^2}\|X\|^{\frac{p}{p'}}_{V^{p}([0,\rho];\mathbb{R}^d)}Osc(X;[0,\rho])^{\frac{p'-p}{p'}}\|Y^{1}-Y^{2}\|_{V^{p',\infty}_{\alpha,\alpha+\theta}([0,\rho])}.
\end{align*}
Due to the right continuity of $X$ at $t=0$, we can choose sufficiently small $\rho$ such that 
\begin{align*}
 C(\|X\|^{\frac{p}{p'}}_{V^{p}([0,\rho];\mathbb{R}^d)}Osc(X,[0,\rho])^{\frac{p'-p}{p'}} +\rho^{\theta })
<1.
\end{align*}
Hence we have 
\begin{align*}
\|Z^{1}-Z^{2}\|_{V^{p',\infty}_{\alpha,\alpha+\theta}([0,\rho])}<\|Y^{1}-Y^{2}\|_{V^{p',\infty}_{\alpha,\alpha+\theta}([0,\rho])}.
\end{align*}
Therefore, there exists a unique solution $Y\in V^{p',\infty}_{\alpha,\alpha+\theta}([0,\rho])$ by applying  Banach's fixed point theorem. The required result can be obtained by the following  bound  
\begin{align}\label{loc-bound}
\|Y\|_{V^{p,\infty}_{\alpha,\alpha+\theta}([0,\rho])}&\lesssim (1+\|X\|_{V^{p}([0,\rho];\mathbb{R}^d)})(1+|Y_{0}|_{\alpha+\theta}).
\end{align}
\end{proof}
\begin{lemma}[Local stability of the solution of \eqref{main}]\label{loc-sta}
Suppose  $X, \tilde X\in V^{p}([0,T]; \mathbb{R}^d)$,  $p\in [1,2)$, are right-continuous at $t=0$, $X_0=\tilde X_0=0$. Let $Y, \tilde{Y}$ solve the equations \eqref{eqyoung} driven by $X$, $\tilde{X}$, respectively.  Then, for $\max\{|Y_{0}|_{\alpha+\theta}, |\tilde{Y}_{0}|_{\alpha+\theta}\}<K$, there exists $t\in(0,t_1)$, such that $\max\{\|X\|_{V^{p}([0,t];\mathbb{R}^d)}, \|\tilde{X}\|_{V^{p}([0,t];\mathbb{R}^d)}\}<\delta_{S, 
\theta, p, K, f}$ and 
\begin{align*}
\|Y-\tilde{Y}\|_{V^{p,\infty}_{\alpha,\alpha+\theta}([0,t])}\leq C(\|X-\tilde{X}\|_{V^{p}([0,t];\mathbb{R}^d)}+|Y_{0}-\tilde{Y}_{0}|_{\alpha+\theta}),
\end{align*}
where  $C$ depends on $S, \theta, p, K, f.$
\end{lemma}
\begin{proof}
Recall that $Y\in \mathscr{B}_{t}$ and $\tilde{Y}\in \tilde{\mathscr{B}}_{t}$ satisfying 
\begin{align*}
Y_{t}&=S_{t,0}Y_{0}+\int_{0}^{t}S_{t,r}f(Y_{r})\textup{d}X_{r},\\
\tilde{Y}
_{t}&=S_{t,0}\tilde{Y}_{0}+\int_{0}^{t}S_{t,r}f(\tilde{Y}_{r})\textup{d}\tilde{X}_{r},
\end{align*}
respectively, where
\begin{align*}
\mathscr{B}_{t}
&:=\{Y\in V^{p, \infty}_{\alpha, \alpha+\theta}([0,t]) ~~| ~~\|Y\|_{V^{p, \infty}_{\alpha, \alpha+\theta}([0,t])}\leq C(1+|Y_{0}|_{\alpha+\theta})\},\\
\tilde{\mathscr{B}}_{t}
&:=\{\tilde{Y}\in V^{p, \infty}_{\alpha, \alpha+\theta}([0,t]) ~~| ~~\|\tilde{Y}\|_{V^{p, \infty}_{\alpha, \alpha+\theta}([0,t])}\leq C(1+|\tilde{Y}_{0}|_{\alpha+\theta})\}.
\end{align*}
Denote $\Delta Y=Y-\tilde{Y}$, $\Delta X=X-\tilde{X}$, $\Delta Y_{0}=Y_{0}-\tilde{Y}_{0}$. 
Let $\mathcal{P}$ be a partition of $[0,t]$, for any $(u,v)\in \mathcal{P}$, 
\begin{align*}
\delta(\Delta Y)_{v,u}&=(S_{v,u}-\textup{id})S_{u,0}\Delta Y_0 \\
&\quad+
(S_{v,u}-\textup{id})
\int_0^u
S_{u,r}\bigl(f(Y_r)-f(\tilde Y_r)\bigr)\,\mathrm d X_r \\
&\quad+
(S_{v,u}-\textup{id})
\int_0^u
S_{u,r}f(\tilde Y_r)\,\mathrm d (\Delta X)_r \\
&\quad+
\int_u^v
S_{v,r}\bigl(f(Y_r)-f(\tilde Y_r)\bigr)\,\mathrm d X_r \\
&\quad+
\int_u^v
S_{v,r}f(\tilde Y_r)\,\mathrm d (\Delta X)_r\\
&\quad :=I_{1}+I_{2}+I_{3}+I_{4}+I_{5}.
\end{align*}
According to \eqref{smooth}, Lemma \ref{YOU01} and Lemma \ref{fff2},
we may estimate the previous terms as follows:
\begin{align*}
|I_{1}|_{\alpha}&=|(S_{v,u}-\textup{id})S_{u,0}\Delta Y_{0}|_{\alpha}\\
&\leq C_S|\Delta Y_{0}|_{\alpha+\theta}|v-u|^{\theta},
\end{align*}
\begin{align*}
|I_{2}|_{\alpha}&=|(S_{v,u}-\textup{id})
\int_0^u
S_{u,r}\bigl(f(Y_r)-f(\tilde Y_r)\bigr)\,\mathrm d X_r |_{\alpha}\\
&\leq C_S|\int_0^u
S_{u,r}\bigl(f(Y_r)-f(\tilde Y_r)\bigr)\,\mathrm d X_r|_{\alpha+\theta}|v-u|^{\theta},
\end{align*}
where 
\begin{align*}
&|\int_0^u
S_{u,r}\bigl(f(Y_r)-f(\tilde Y_r)\bigr)\,\mathrm d X_r|_{\alpha+\theta}\\
&\leq |\int_0^u
S_{u,r}\bigl(f(Y_r)-f(\tilde Y_r)\bigr)\,\mathrm d X_r-S_{u,0}\big(f(Y_0)-f(\tilde{Y}_0)\big)X_{u}|_{\alpha+\theta}+|S_{u,0}\big(f(Y_0)-f(\tilde{Y}_0)\big)X_{u}|_{\alpha+\theta}\\
&\leq C_{S, \theta, p}[\|X\|_{V^{p}([0,u];\mathbb R^d)}\|f(Y)-f(\tilde Y)\|_{V^{p}([0,u];\cB_{\alpha+\theta})}+(|f(Y_0)-f(\tilde Y_0)|_{\alpha+\theta}\\
&~~+\|f(Y)-f(\tilde Y)\|_{V^{\infty}([0,u];\cB_{\alpha+2\theta})})\|X\|_{V^{p}([0,u];\mathbb R^d)}u^{\theta}]+C_S|[f(Y_0)-f(\tilde Y_0)]X_u|_{\alpha+\theta}\\
&\leq C_{S, \theta, p, K}[\|X\|_{V^{p}([0,u];\mathbb{R}^d)}\|f\|_{C_b^2}(\|\Delta Y\|_{V^{p}([0,u];\cB_{\alpha})}+|\Delta Y_0|_{\alpha+\theta})\\
&~~+\|f\|_{C_b^2}(|\Delta Y_{0}|_{\alpha+\theta}+\|\Delta Y\|_{V^{\infty}([0,u];\cB_{\alpha+\theta})})\|X\|_{V^{p}([0,u];\mathbb{R}^d)}u^{\theta}\\
&~~+\|f\|_{C_b^2}\|X\|_{V^{p}([0,u];\mathbb{R}^d)}|\Delta Y_{0}|_{\alpha+\theta}]\\
&\leq C_{S, \theta, p, K}\|f\|_{C_b^2}\|X\|_{V^{p}([0,u];\mathbb{R}^d)}(|\Delta Y_{0}|_{\alpha+\theta}+\|\Delta Y\|_{V^{p,\infty}_{\alpha,\alpha+\theta}([0,u])}). 
\end{align*}
Similarly, we have 
\begin{align*}
|I_{3}|_{\alpha}&=|(S_{v,u}-\textup{id})
\int_0^u
S_{u,r}f(\tilde Y_r)\,\mathrm d (\Delta X)_r|_{\alpha}\\
&\leq C_S |\int_0^u
S_{u,r}f(\tilde Y_r)\,\mathrm d (\Delta X)_r|_{\alpha+\theta}|v-u|^{\theta}\\
&\leq C_S|v-u|^{\theta}[|\int_0^u
S_{u,r}f(\tilde Y_r)\,\mathrm d (\Delta X)_r-S_{u,0}f(\tilde{Y}_0 )(\Delta X)_u|_{\alpha+\theta}+|S_{u,0}f(\tilde{Y}_0 )(\Delta X)_u|_{\alpha+\theta}]\\
&\leq C_{S, \theta, p}|v-u|^{\theta}[
\|\Delta X\|_{V^{p}([0,u];\mathbb{R}^d)}\|f(\tilde{Y})\|_{V^{p}([0,u];\cL(\mathbb{R}^d,\cB_{\alpha+\theta}))}\\
&~~+(\|f(\tilde{Y}_0)\|_{\cL(\mathbb{R}^d,\cB_{\alpha+2\theta})}+\|f(\tilde{Y})\|_{V^{\infty}([0,u];\cL(\mathbb{R}^d,\cB_{\alpha+2\theta}))})\|\Delta X\|_{V^{p}([0,u];\mathbb{R}^d)}u^{\theta}\\
&~~+\|f(\tilde{Y}_{0})\|_{\cL(\mathbb{R}^d,\cB_{\alpha+\theta})}\|\Delta X\|_{V^{p}([0,u];\mathbb{R}^d)}
]\\
&\leq C_{S, \theta, p} |v-u|^{\theta}[\|f\|_{C_b^2}\|\Delta X\|_{V^{p}([0,u];\mathbb{R}^d)}\|\tilde Y\|_{V_{\alpha, \alpha+\theta}^{p,\infty}([0,u])}\\
&~~+\|f\|_{C_b^2}\|\Delta X\|_{V^{p}([0,u];\mathbb{R}^d)}(|\tilde Y_0|_{\alpha+\theta}+\|\tilde Y\|_{V_{\alpha, \alpha+\theta}^{p,\infty}([0,u])})u^\theta\\
&~~+
\|f\|_{C_b^2}\|\Delta X\|_{V^{p}([0,u];\mathbb{R}^d)}]\\
&\leq C_{S, \theta, p, K} |v-u|^{\theta}\|f\|_{C_b^2}\|\Delta X\|_{V^{p}([0,u];\mathbb{R}^d)}.
\end{align*}
\begin{align*}
|I_{4}|_{\alpha}&=|\int_u^v
S_{v,r}\bigl(f(Y_r)-f(\tilde Y_r)\bigr)\,\mathrm d X_r|_{\alpha}\\
&\leq |\int_u^v
S_{v,r}\bigl(f(Y_r)-f(\tilde Y_r)\bigr)\,\mathrm d X_r-S_{v,u}\big(f(Y_{u})-f(\tilde Y_u)\big)\delta X_{v,u}|_{\alpha}+|S_{v,u}\big(f(Y_{u})-f(\tilde Y_u)\big)\delta X_{v,u}|_{\alpha}\\
&\leq C_{S, \theta, p}\big[
\|X\|_{V^{p}([u,v];\mathbb{R}^d)}\|f(Y)-f(\tilde Y)\|_{V^{p}([u,v];\cL(\mathbb{R}^d,\cB_{\alpha})}\\
&~~+\big(\|f(Y_0)-f(\tilde Y_0)\|_{\cL(\mathbb{R}^d,\cB_{\alpha+\theta})}+\|f(Y)-f(\tilde Y)\|_{V^\infty([0,u];\cL(\mathbb{R}^d,\cB_{\alpha+\theta}))}\big)\|X\|_{V^{p}([u,v];\mathbb{R}^d)}|v-u|^\theta\\
&~~+\|f(Y_u)-f(\tilde Y_u)\|_{\cL(\mathbb{R}^d,\cB_{\alpha})}\|X\|_{V^{p}([u,v];\mathbb{R}^d)}
\big]\\
&\leq C_{S, \theta, p, K}  \|f\|_{C_b^2}\|X\|_{V^{p}([u,v];\mathbb{R}^d)}(\|\Delta Y\|_{V^p([u,v];\cB_{\alpha})}+|\Delta Y_0|_{\alpha+\theta}+\|\Delta Y\|_{V^{p,\infty}_{\alpha,\alpha+\theta}([0,u])}).
\end{align*}
\begin{align*}
|I_{5}|_{\alpha}&=|\int_u^v
S_{v,r}f(\tilde Y_r)\,\mathrm d (\Delta X)_r|_{\alpha}\\
&\leq |\int_u^v
S_{v,r}f(\tilde Y_r)\,\mathrm d (\Delta X)_r-S_{v,u}f(\tilde Y_u)\delta(\Delta X)_{v,u}|_{\alpha}+|S_{v,u}f(\tilde Y_u)\delta(\Delta X)_{v,u}|_{\alpha}\\
&\leq C_{S, \theta, p}\big[
\|\Delta X\|_{V^p([u,v];\mathbb{R}^d)}\|f(\tilde Y)\|_{V^p([u,v];\cL(\mathbb{R}^d,\cB_\alpha))}
\\
&~~+(\|f(\tilde Y_0)\|_{\cL(\mathbb{R}^d,\cB_{\alpha+\theta})}+\|f(\tilde Y)\|_{V^\infty([0,u];\cL(\mathbb{R}^d,\cB_{\alpha+\theta})})\|\Delta X\|_{V^p([u,v];\mathbb{R}^d)}|u-v|^\theta
\\
&~~+\|f(\tilde Y_u)\|_{\cL(\mathbb{R}^d,\cB_\alpha)}\|\Delta X\|_{V^p([u,v];\mathbb{R}^d)}\big]\\
&\leq C_{S, \theta, p}\big[\|f\|_{C^2}
\|\Delta X\|_{V^p([u,v];\mathbb{R}^d)}\|\tilde Y\|_{V^p([u,v];\cB_\alpha)}\\
&~~+\|f\|_{C_b^2}(1+\|\tilde Y\|_{V^\infty([0,u];\cB_{\alpha+\theta})})\|\Delta X\|_{V^p([u,v];\mathbb{R}^d)}|u-v|^\theta)\\
&~~+\|f\|_{C_b^2}\|\Delta X\|_{V^p([u,v];\mathbb{R}^d)}
\big]\\
&\leq C_{S, \theta, p, K}\|f\|_{C_b^2}\|\Delta X\|_{V^p([u,v];\mathbb{R}^d)}. 
\end{align*}
Combining  the above estimates together gives 
\begin{align*}
&\|\Delta Y\|_{V^p([0,t];\cB_{\alpha})}=(\sup\limits_{\mathcal{P}}\sum_{[u,v]\in\mathcal{P}}|\delta(\Delta Y)_{v,u}|^{p}_{\alpha})^{\frac 1p}\\
&\leq  C_{S, \theta, p, K} \big[\sup\limits_{\mathcal{P}}\sum_{[u,v]\in\mathcal{P}} \{ |\Delta Y_0|_{\alpha+\theta}^p|v-u|^{\theta p}\\
&~~+|v-u|^{\theta p}\|f\|_{C_b^2}^p\|X\|^p_{V^p([0,u];\mathbb{R}^d)}(|\Delta Y_0|^p_{\alpha+\theta}+\|\Delta Y\|_{V^{p,\infty}_{\alpha,\alpha+\theta}([0,u])}^p)\\
&~~+|v-u|^{\theta p}\|f\|^p_{C_b^2}\|\Delta X\|^p_{V^p([0,u];\mathbb R^d)}\\
&~~+\|f\|^p_{C_b^2}\|X\|^p_{V^{p}([u,v];\mathbb R^d)}(\|\Delta Y\|^p_{V^p([u,v];\cB_\alpha)}+|\Delta Y_0|^p_{\alpha+\theta}+\|\Delta Y\|^p_{V^{p,\infty}_{\alpha, \alpha+\theta}([0,u])})\\
&~~+\|f\|^p_{C_b^2}\|\Delta X\|^p_{V^{p}([u,v];\mathbb R^d)}\}\big]^{\frac 1p}\\
&\leq  C_{S, \theta, p, K} \big[ |\Delta Y_0|_{\alpha+\theta}+\|f\|_{C_b^2}\|X\|_{V^p([0,t];\mathbb{R}^d)}(|\Delta Y_0|_{\alpha+\theta}+\|\Delta Y\|_{V^{p,\infty}_{\alpha,\alpha+\theta}([0,t])})\\
&~~+\|f\|_{C_b^2} \|\Delta X\|_{V^p([0,t];\mathbb R^d)}\big].
\end{align*}

It remains to show the estimate in $V^\infty$. Applying \eqref{smooth}, Lemma \ref{YOU01} and Lemma \ref{fff2} again, we have 
\begin{align*}
&\|\Delta Y\|_{V^{\infty}([0,t];\cB_{\alpha+\theta})}=\sup\limits_{0\leq s\leq t}|\Delta Y_{s}-\Delta Y_0|_{\alpha+\theta}\\&=\sup\limits_{0\leq s\leq t}|S_{s,0}Y_0+\int_0^sS_{s,r}f(Y_r)\textup{d}X_r-S_{s,0}\tilde Y_0-\int_0^sS_{s,r}f(\tilde Y_r)\textup{d}\tilde X_r-Y_0+\tilde Y_0|_{\alpha+\theta}\\
&\leq \sup\limits_{0\leq s\leq t}
\big(
|(S_{s,0}-\textup{id})\Delta Y_0|_{\alpha+\theta}
+|\int_0^sS_{s,r}(f(Y_r)-f(\tilde Y_r))\textup{d}X_r|_{\alpha+\theta}
+|\int_0^sS_{s,r}f(\tilde Y_r)\textup{d}(\Delta X)_r|_{\alpha+\theta}\big)
\\
&\leq C_S|\Delta Y_0|_{\alpha+\theta}
+\sup\limits_{0\leq s\leq t}\big(|\int_0^sS_{s,r}(f(Y_r)-f(\tilde Y_r))\textup{d}X_r-S_{s,0}(f(Y_0)-f(\tilde Y_0))X_{s}|_{\alpha+\theta}\\
&~~+|S_{s,0}(f(Y_0)-f(\tilde Y_0))X_{s}|_{\alpha+\theta}
+|\int_0^sS_{s,r}f(\tilde Y_r)\textup{d}(\Delta X)_r-S_{s,0}f(\tilde Y_0)\Delta X_s|_{\alpha+\theta}
+|S_{s,0}f(\tilde Y_0)\Delta X_s|_{\alpha+\theta}\big)\\
&\leq C_S|\Delta Y_0|_{\alpha+\theta}
+C_{S, \theta, p, K}\sup\limits_{0\leq s\leq t}\{
\big[\|X\|_{V^{p}([0,s];\mathbb{R}^d)}\|f\|_{C_b^2}\big(\|\Delta Y\|_{V^{p}([0,s];\cB_{\alpha})}+|\Delta Y_0|_{\alpha+\theta}\\
&~~+(|\Delta Y_0|_{\alpha+\theta}+\|\Delta Y\|_{V^\infty([0,s];\cB_{\alpha+\theta})})s^{\theta}\big)
+\|f\|_{C_b^2}|\Delta Y_0|_{\alpha+\theta}\|X\|_{V^{p}([0,s];\mathbb{R}^d)}\big]
+[\|\Delta X\|_{V^p([0,s];\mathbb{R}^d)}\|f\|_{C^2}
]
\}\\
&\leq C_{S, \theta, p, K}[
|\Delta Y_0|_{\alpha+\theta}+\|f\|_{C_b^2}\|X\|_{V^p([0,t];\mathbb{R}^d)}(\|\Delta Y\|_{V^{p,\infty}_{\alpha,\alpha+\theta}([0,t])}+|\Delta Y_0|_{\alpha+\theta})+\|\Delta X\|_{V^p([0,t];\mathbb{R}^d)}\|f\|_{C_b^2}
]. 
\end{align*}
Thus we have 
\begin{align*}
\|\Delta Y \|_{V^{p,\infty}_{\alpha,\alpha+\theta}([0,t])}&\leq  C_{S, \theta, p, K}[
|\Delta Y_0|_{\alpha+\theta}+\|f\|_{C_b^2}\|X\|_{V^p([0,t];\mathbb{R}^d)}(\|\Delta Y\|_{V^{p,\infty}_{\alpha,\alpha+\theta}([0,t])}+|\Delta Y_0|_{\alpha+\theta})\\
&~~+\|\Delta X\|_{V^p([0,t];\mathbb{R}^d)}\|f\|_{C_b^2}
]. 
\end{align*}
Choosing $t$ sufficiently small, such that $C_{S, \theta, p, K, f} \|X\|_{V^p([0,t];\mathbb{R}^d)}<\frac 12$. Thus, we derive 
\begin{align*}
\|\Delta Y \|_{V^{p,\infty}_{\alpha,\alpha+\theta}([0,t])}\leq  C_{S, \theta, p, K, f}[|\Delta Y_0|_{\alpha+\theta}+\|\Delta X\|_{V^p([0,t];\mathbb{R}^d)}].
\end{align*}
\end{proof}
\begin{theorem}[Global  solution of \eqref{main} driven by  c\`adl\`ag paths]\label{cadlagGLOBAL}
Suppose $X\in V^{p}([0,T]; \mathbb{R}^d)$, $p\in[1,2)$,  is  c\`adl\`ag, $X_{0}=0$ and  let $f(\cdot)\in C_b^2(\cB_{\alpha+i\theta}, \cL(\mathbb{R}^d,\mathcal{B}_{\alpha+(i+1)\theta+\varepsilon}))$, $i=0,1,$ $\alpha \in\mathbb{R}$, $\theta\in[\frac 1p,1]$,  $\varepsilon>0$. Then, for $Y_0\in \mathcal{B}_{\alpha+\theta}$, there exists  a unique solution 
$Y\in V^{p, \infty}_{\alpha, \alpha+\theta}([0,T])$ satisfying 
\begin{align}\label{cady}
Y_{t}=S_{t,0}Y_{0}+\int_{0}^{t}S_{t,r}f(Y_{r})\textup{d}X_{r}, \quad t\in [0,T].
\end{align}
Furthermore, let $X,\tilde X\in V^p([0,T];\mathbb R^d)$ be  c\`adl\`ag paths 
with $X_0=\tilde X_0=0$, and suppose that
$\max \{|Y_0|_{\alpha+\theta}, |\tilde Y_0|_{\alpha+\theta}\}<R$, $\max\{\|X\|_{V^{p}([0,T]);\mathbb{R}^d}, \|\tilde X\|_{V^{p}([0,T]);\mathbb{R}^d}\}<L$.  Let $Y$ and $\tilde Y$   be the corresponding solutions of
\eqref{main} on $[0,T]$,  with initial values $Y_0$ and
$\tilde Y_0$,  driven by $X$, $\tilde X$,  respectively. Then 
\begin{align*}
\|Y-\tilde Y\|_{V^{p,\infty}_{\alpha, \alpha+\theta}([0,T])}\leq M(2C (L\vee 1))^{M+1} [|Y_0-\tilde Y_0|_{\alpha+\theta}+\|X-\tilde X\|_{V^p([0,T];\mathbb{R}^d)}].
\end{align*}
where $M$ is a constant bounded by $C_L^{p}+1$, $C_L$ is a constant depending on $S, \theta, p, f, R, L$. 
\end{theorem}
\begin{proof}
According to the local solution established in  Lemma \ref{locyoung},  given initial value $y\in \cB_{\alpha+\theta}$,  there exists a constant $\delta_{1}$, such that if
$\|X\|_{V^{p}([u,v);\mathbb{R}^d)}<\delta_{1},$
then there exists a unique solution $Y\in V^{p,\infty}_{\alpha,\alpha+\theta}([u,v))$ to 
\begin{align}\label{eqy01}
Y_{t}=S_{t,u}y+\int_{u}^{t}S_{t,r}f(Y_{r})\textup{d}X_{r}, \quad t\in [u,v).
\end{align}
In view of Lemma \ref{partition}, there exists a partition $\mathcal{P}:=\{0=t_{0}<t_{1}<\dots<t_{N}=T\}$, such that for any interval $(t_{i},t_{i+1})$, $\omega_{X,p}(t_{i+1}-,t_{i}+)<\delta_{1}^{p}$. The  right-continuity ensures that,  given initial data $y_{i}\in \cB_{\alpha+\theta}$, $i=0,\dots, N-1$, one can solve \eqref{eqy01} on $[t_{i}, t_{i+1})$.  Denote $y_0:=Y_0$ and construct the initial data
\begin{align}\label{gouzaoini}
y_{i+1}:=Y^{y_{i}}_{t_{i+1-}}+f(Y^{y_{i}}_{t_{i+1-}})\Delta_{t_{i+1}}^- X, 
\end{align}
where $Y_{t_{i+1-}}
\in \cB_{\alpha+\theta}$, $f(Y^{y_i}_{t_{i+1-}})\Delta_{t_{i+1}}^- X \in \cB_{\alpha+2\theta+\varepsilon}\subset \cB_{\alpha+\theta}$. Thus $y_{i+1}\in \cB_{\alpha+\theta}$, which allows us to solve the equation on the next interval $[t_{i+1}, t_{i+2})$. Iterating this procedure, we obtain a global solution  defined by  $Y_t:=Y_t^{y_{i}}$, $t\in [t_i,t_{i+1})$,  $i=0,\dots, N-1$, $Y_T:=y_N$. By construction,  $Y$ satisfies \eqref{cady}. The uniqueness follows by construction.

Next, we prove the global stability of the solution. 
Set $\Delta Y:=Y-\tilde Y$, $\Delta X:=X-\tilde X$. 
As shown in Lemma \ref{loc-sta}, there exists a constant $\delta_2$ depending on $S, \theta, p, f, R$, such that $\|X\|_{V^{p}([u,v);\mathbb{R}^d)} \vee  \|\tilde X\|_{V^{p}([u,v);\mathbb{R}^d)}<\delta_2<1$ 
and the local Lipschitz estimate holds
\begin{align}\label{loc-2Y}
\|Y-\tilde Y\|_{V^{p,\infty}_{\alpha,\alpha+\theta}([u,v))}\leq C[|\Delta Y_u|_{\alpha+\theta}+\|\Delta X\|_{V^p([u,v);\mathbb{R}^d)}], 
\end{align}
where $C$ depends on $S, \theta, p, f, R$.
Since $\omega_{X, p}(t,0)$, $\omega_{\tilde X, p}(t, 0)$ are increasing and bounded by $L^p$,  one may 
construct  a partition separately for $X$ and $\tilde X$ with at most $[\frac{L^p}{\delta_2^p}]+2$ endpoints. Then let $\mathcal{Q}:=\{0=s_0<s_1<\cdots<s_M=T\}$ be their common refinement with  $M<2\frac{L
^p}{\delta_2 ^p}+1$, such that 
$\omega_{X,p}(s_{i+1}-, s_i+)\vee \omega_{\tilde X, p}(s_{i+1}-, s_i+)<\delta_2 ^p$. On every time interval $[s_i, s_{i+1})$, similar to the estimate in \eqref{loc-2Y}, we have 
\begin{align}\label{loc-2Y'}
\|Y-\tilde Y\|_{V^{p,\infty}_{\alpha,\alpha+\theta}([s_i,s_{i+1}))}\leq C[|\Delta Y_{s_i}|_{\alpha+\theta}+\|\Delta X\|_{V^p([s_i,s_{i+1});\mathbb{R}^d)}].
\end{align}
Since 
$
\|g(\cdot)\|_{V^p([u,v];
\cB_\alpha)}\leq \|g(\cdot)\|_{V^{p}([u,v);\cB_\alpha)}+|g(v)-g(v-)|_\alpha
$
and 
\begin{align*}
\|g(\cdot)\|_{V^\infty([u,v];\cB_{\alpha+\theta})}&=\max\{\|g(\cdot)\|_{V^\infty([u,v);\cB_{\alpha+\theta})}, |g(v)-g(u)|_{\alpha+\theta}\}\\
&\leq
\|g\|_{V^\infty([u,v);\mathcal B_{\alpha+\theta})}
+
|g(v)-g(v-)|_{\alpha+\theta},
\end{align*}
we have 
\begin{align}\label{buqi}
\|g(\cdot)\|_{V^{p,\infty}_{\alpha, \alpha+\theta}([u,v])}\lesssim \|g(\cdot)\|_{V^{p,\infty}_{\alpha, \alpha+\theta}([u,v))}+|g(v)-g(v-)|_{\alpha+\theta}.
\end{align}
In view of \eqref{gouzaoini}, we have the equality
\begin{align}
&|Y_{s_{i+1}}-Y_{s_{i+1}-}-(\tilde Y_{s_{i+1}}-\tilde Y_{s_{i+1}-})|_{\alpha+\theta}\nonumber \\
&=|f(Y_{s_{i+1}-})\Delta_{s_{i+1}}^{-}X- f(\tilde Y_{s_{i+1}-})\Delta_{s_{i+1}}^-\tilde X|_{\alpha+\theta} \nonumber\\
&\leq \|f\|_{C_b^2} |\Delta Y_{s_{i+1}-}|_\alpha |\Delta_{s_{i+1}}^-X|+\|f(\tilde Y_{s_{i+1}-})\|_{\cL(\mathbb{R}^d,\cB_{\alpha+\theta})}|\Delta_{s_{i+1}}^-(\Delta X)|\nonumber \\
&\leq \|f\|_{C_b^2} (|\Delta Y_{s_{i+1}-}|_\alpha  \|X\|_{V^p([s_i, s_{i+1}];\mathbb{R}^d)}+ \|\Delta X\|_{V^p([s_i, s_{i+1}];\mathbb{R}^d)})\nonumber \\
&\leq C\{[|\Delta Y_{s_i}|_{\alpha+\theta}+\|\Delta Y\|_{V^{p,\infty}_{\alpha,\alpha+\theta}([s_i, s_{i+1}))}]L+\|\Delta X\|_{V^p([s_i, s_{i+1}];\mathbb{R}^d)})\}\nonumber \\
&\leq C(L\vee 1)[|\Delta Y_{s_i}|_{\alpha+\theta}+\|\Delta X\|_{V^p([s_i, s_{i+1}];\mathbb{R}^d)})].\label{buqi0}
\end{align}
Using the estimates  \eqref{buqi} and \eqref{buqi0}, we extend the local estimate  \eqref{loc-2Y'} to the  closed interval $[s_i, s_{i+1}]$, 
\begin{align}\label{local-bound-D}
\|\Delta Y\|_{V^{p,\infty}_{\alpha,\alpha+\theta}([s_i, s_{i+1}])}\leq  C(L\vee 1)[|\Delta Y_{s_i}|_{\alpha+\theta}+\|\Delta X\|_{V^p([s_i, s_{i+1}];\mathbb{R}^d)})],
\end{align}
which implies 
\begin{align*}
|\Delta Y_{s_{i+1}}|_{\alpha+\theta}&\leq |\Delta Y_{s_i}|_{\alpha+\theta}+\|\Delta Y\|_{V^{p,\infty}_{\alpha, \alpha+\theta}([s_i, s_{i+1}])}\\
&\leq C(L\vee 1)[|\Delta Y_{s_i}|_{\alpha+\theta}+\|\Delta X\|_{V^p([s_i, s_{i+1}];\mathbb{R}^d)}]. 
\end{align*}
By iterating, we have 
\begin{align*}
|\Delta Y_{s_{i+1}}|^p_{\alpha+\theta}&\leq (2C(L\vee 1))^p [|\Delta Y_{s_i}|^p_{\alpha+\theta}+\|\Delta X\|^p_{V^p([s_i,s_{i+1}];\mathbb{R}^d)}]\\
&\leq (2C(L\vee 1))^p\big[ (2C(L\vee 1))^p(|\Delta Y_{s_{i-1}}|^p_{\alpha+\theta}+\|\Delta X\|^p_{V^{p}([s_{i-1}, s_i]; \mathbb R^d)})+\|\Delta X\|^p_{V^p([s_i,s_{i+1}];\mathbb{R}^d)}\big]\\
&\leq (2C(L\vee 1))^{pM}[|\Delta Y_0|^p_{\alpha+\theta}+\sum_{i=0}^{M-1}\|\Delta X\|^p_{V^p([s_i,s_{i+1}];\mathbb{R}^d)}]
\end{align*}
which yields
\begin{align}\label{loc-Ys-D}
|\Delta Y_{s_{i+1}}|_{\alpha+\theta}\leq (2C(L\vee 1))^{M}[|\Delta Y_0|_{\alpha+\theta}+\|\Delta X\|_{V^{p}([0,T]; \mathbb{R}^d)}].
\end{align}
Inserting  \eqref{loc-Ys-D} into  \eqref{local-bound-D}, we have 
\begin{align*}
\|\Delta Y\|_{V^{p, \infty}_{\alpha, \alpha+\theta}([s_i, s_{i+1}])}\leq (2C(L\vee 1))^{M+1}[|\Delta Y_0|_{\alpha+\theta}+\|\Delta X\|_{V^p([0,T];\mathbb{R}^d)}].
\end{align*}
Using the facts that 
$
\|\cdot\|_{V^{p}([0,T];\cB_\alpha)}\leq \sum\limits_{i=0}^{M-1}\|\cdot\|_{V^{p}([s_i,s_{i+1}];\cB_{\alpha})}
$ and $\|\cdot\|_{V^\infty([0,T];\cB_{\alpha+\theta})}\leq \sum\limits_{i=0}^{M-1}\|\cdot\|_{V^\infty([s_i,s_{i+1}];\cB_{\alpha+\theta})}$, we have 
$\|\cdot\|_{V^{p,\infty}_{\alpha,\alpha+\theta}([0,T])}\leq \sum\limits_{i=0}^{M-1}\|\cdot\|_{V^{p,\infty}_{\alpha,\alpha+\theta}([s_i, s_{i+1}])}$ and 
the required bound is obtained. 
\end{proof}

\begin{theorem}[Global solution of \eqref{main} driven by jump processes]\label{Glo-young}
Suppose $X\in V^{p}([0,T]; \mathbb{R}^d)$, $p\in[1,2)$, $X_{0}=0$ and   let $f(\cdot)\in C_b^2(\cB_{\alpha+i\theta}, \cL(\mathbb{R}^d,\mathcal{B}_{\alpha+(i+1)\theta+\varepsilon}))$, $i=0,1,$ $\alpha \in\mathbb{R}$, $\theta\in [\frac{1}{p},1]$,   $\varepsilon>0$.  Then, for $Y_0\in \mathcal{B}_{\alpha+\theta}$, there exists a unique solution 
$Y\in V^{p, \infty}_{\alpha, \alpha+\theta}([0,T])$ satisfying 
\begin{align}\label{ppY}
Y_{t}=S_{t,0}Y_{0}+\int_{0}^{t}S_{t,r}f(Y_{r})\textup{d}X_{r}, \quad t\in [0,T].
\end{align}
\end{theorem}
\begin{proof}
Firstly, we prove the local well-posedness.  Define 
$$
Y_{0+}:= Y_0+f(Y_0)\Delta_0^+ X
$$
and 
$$
\tilde X_t:=\left\{
\begin{array}{ll}
X_{0+}, & t=0,  \\
X_t,&  t>0.
\end{array}
\right.
$$ 
One readily checks that $\tilde X$ is right-continuous at $t=0$. Then using  Lemma \ref{locyoung},  we claim that there exists a unique solution to  the following equation driven by $\tilde X$ on  some $[0,t_0)$, 
 \begin{align*}
\tilde Y_t =S_{t,0} Y_{0+}+ \int_0^t S_{t,r} f(\tilde Y_r)\textup{d}\tilde X_r, 
\end{align*}
which defines a solution to \eqref{ppY} by $Y_t=\tilde Y_t$,   $t\in [0,t_0]$, 
in view of the arguments in Remark \ref{localpp}. 

Now, we extend the local solution to a global one  using the approach in Lemma  \ref{cadlagGLOBAL}. By Lemma \ref{partition}, there exists a partition $\mathcal{P}:=\{0=t_{0}<t_{1}<\dots<t_{N}=T\}$, such that for any interval $(t_{i},t_{i+1})$, $i=0,\cdots, N-1$,  $\omega_{X,p}(t_{i+1}-,t_{i}+)$ is controlled by a small constant, which allows us to derive the local solution to 
\begin{align}\label{theY}
\tilde Y_t=  S_{t, t_i} y_i+\int_{t_i}^{t} S_{t, r}f(\tilde Y_r)\textup{d}\tilde X_r,  
\end{align}
with 
$$
\tilde X_t:=\left\{
\begin{array}{ll}
X_{t_{i}+}, & t=t_i, \\
X_t,&  t_i<t<t_{i+1}. 
\end{array}
\right.
$$
It remains to construct the initial values $y_i$  in \eqref{theY}.  
Denote  the solution to \eqref{theY} with initial value $y_i$ by $\tilde Y_t^{y_i}$. From (4) of Lemma \ref{llrr}, define 
\begin{align*} 
\tilde y_{i+1}:= \tilde Y_{t_{i+1}-}^{y_i}+f(\tilde Y_{t_{i+1}-}^{y_i}) \Delta_{t_{i+1}}^- X. 
\end{align*} 
Then define  
\begin{align*}
y_{i+1}:= \tilde y_{i+1}+f( \tilde y_{i+1})\Delta_{t_{i+1}}^+ X
\end{align*}
as the initial value on the next time interval $[t_{i+1}, t_{i+2})$.  Particularly, $y_0=Y_{0+}$. 

Finally, we obtain global well-posedness by defining $Y$ on each  interval $[t_{i}, t_{i+1})$  as
$$
Y_t:=\left\{
\begin{array}{ll}
\tilde y_i, & t=t_i, \\
\tilde Y_t,&  t_i
<t<t_{i+1}. 
\end{array}
\right.
$$
\end{proof}

\subsection{Comparison with H\"older continuous paths}

 The multiplicative sewing lemma   with jumps  in  Theorem \ref{thm:mul01} differs from the continuous case \cite[Theorem 3.4]{GHN21} in several aspects. 
First, the assumptions on the multiplicative defect are different.
 In the continuous setting, the multiplicative defect is usually controlled by a single control $\omega(t,s)^z$ with $z>1$.
In the discontinuous setting, small intervals may still contain jumps. Hence, the  multiplicative defect is instead estimated by 
$
\sum\limits_{j=1}^N \omega_{1,j}^{\alpha_{1,j}}(t,u)\omega_{2,j}^{\alpha_{2,j}}(u,s)$ with 
$\alpha_{1,j}+\alpha_{2,j}>1. 
$ In particular, $N=2$ corresponds to the Young case in Theorem \ref{affyoung01}, while $N=4$ corresponds to the rough case, see Theorem \ref{affrough01} below.

Second, the two results use different notions of convergence. In the continuous setting, continuity of the control implies that a small mesh size gives small local increments. Thus, convergence can be defined directly in the mesh sense. In the discontinuous setting, this obviously fails due to the jumps. Therefore, the limit is first constructed in the refinement sense. Under suitable one-sided continuity conditions \textcolor{black}{of the driven paths}, this can be strengthened to the standard mesh convergence.

Moreover, there are also differences in the affine sewing lemma Theorem \ref{affyoung01} and \cite[Theorem 4.1]{GHN21}. Regarding the sewing estimate, in the continuous setting, the smallness of the time interval can be used to gain additional spatial regularity, which yields estimates in spaces of higher regularity. In the discontinuous setting, such a gain is in general not available, since the control increments may remain non-small on arbitrarily short intervals due to jumps.
This is also one of the reasons why our assumptions on the nonlinearity $f$ differ from those in the continuous case. In our setting, $f$ is required to improve spatial regularity, whereas in the continuous case $f$ is typically allowed to lose regularity.  In addition,  the initial value is also required have higher space regularity.
\section{\texorpdfstring{$p$}-variation rough paths} \label{sec:sec:prough}
The aim of this section is to study $p$-rough paths, $p\in[2,3)$. We fisrt extend the affine sewing lemma in Theorem \ref{affyoung01}  from the Young case  to the rough case, and then construct a well-defined rough integral with jumps. At last, we establish local existence and uniqueness of mild solutions to rough PDEs with jumps.
\begin{theorem}[\textbf{Affine Sewing Lemma with Jumps and Rough}] \label{affrough01}
	Consider $(S_{t,s})_{(t, s)\in \Delta_2}$ as a propagator  in the sense of
Definition~\ref{fullrange}. Fix $\gamma \in \mathbb{R}$, $0\leq \theta_1<\theta_2\leq 1$.    Let  $\xi:=\xi^{1}+\xi^{2}:\Delta_{2}\rightarrow  \mathcal{B}_{\gamma}+\mathcal{B}_{\gamma-\theta_{1}}$,  $\delta\xi: \Delta_{3}\rightarrow  \mathcal{B}_{\gamma-\theta_{2}}$ and assume that there exist controls $\omega_1, \omega_2, \omega_{1,j}, \omega_{2,j}$ and positive constants $\alpha_1, \alpha_2, \alpha_{1,j}, \alpha_{2,j}$ with $\min\limits_{j=1,2} \{\alpha_{1}+\theta_{2},  \alpha_{2}+\theta_{2}-\theta_{1}, \alpha_{1,j}+\alpha_{2,j}\}>1$, such that for every $(t,u,s)\in \Delta_3$, 
    \begin{align}
&|\xi^{1}_{t,s}|_{\gamma}\leq \omega_{1}^{\alpha_{1}}(t,s), \quad |\xi^{2}_{t,s}|_{\gamma-\theta_{1}}\leq \omega_{2}^{\alpha_{2}}(t,s), \label{xixi}\\
&|\delta \xi_{t,u,s}|_{\gamma-\theta_{2}}\leq  \sum_{j=1}^{2} \omega_{1,j}^{\alpha_{1,j}}(t, u) \omega_{2,j}^{\alpha_{2,j}}(u,s), \label{subcon}
    \end{align}
    Then for every $t\in [0,T]$, the limit
	\begin{equation} \label{app02}
		\mathcal{I}_t(\xi):= RRS- \lim_{|\mathcal{P}|\to 0} \sum_{[u,v]\in \mathcal{P}}S_{t,u}\xi_{v,u}
	\end{equation}
exists uniquely in $\mathcal{B}_{\gamma-\theta_{2}}$, over arbitrary partitions $\mathcal{P}$ of $[0,t]$.

Morover, for every $0\leq s\leq t\leq T$, 
\begin{equation}\label{rough-bd}
\begin{split}
		&| \delta^{S}\mathcal{I}_{t,s}(\xi)-S_{t,s}\xi_{t,s}|_{\gamma-\theta_{2}}\\
		\leq&~~ C [\sum\limits_{j=1}^{2}\omega_{1,j}^{\alpha_{1,j}}(t,s+)\omega_{2,j}^{\alpha_{2,j}}(t-,s)+\omega_1^{\alpha_1}(t,s+)|t-s|^{\theta_2}+\omega_2^{\alpha_{2}}(t,s+)|t-s|^{\theta_2-\theta_1}],
        \end{split}
    \end{equation}
      where $C$ depends on $S$, $\min\limits_{j=1,2} \{\alpha_{1}+\theta_{2},  \alpha_{2}+\theta_{2}-\theta_{1}, \alpha_{1,j}+\alpha_{2,j}\}.$

    Furthermore, if $\omega_{1,j}$ are right-continuous in $s$ or $\omega_{2,j}$ are left-continuous in $t$, then the convergence holds in the MRS sense.
\end{theorem}
\begin{proof}
	Denote the monoid $\mathcal{M}=\mathcal{M}(\gamma-\theta_{2})$. We claim that 
	 $\mathcal{M}$ can be endowed with a submultiplicative monoid structure. (Please refer to  Step 1 in Theorem \ref{affyoung01}). Define $\mu :\Delta_2 \to \mathcal{M}$ as
$
\mu _{t,s}:=(S_{t,s}, S_{t,s}\xi _{t,s})\,.
$
For every $(t,u,s)\in\Delta_3,$
observe that 
\begin{equation*}
	\mu _{t,s}-\mu _{t,u }\circ\mu _{u ,s}
	\equiv
	(0,S_{t,s}\delta \xi _{t,u ,s}+ S_{t,u }(S_{u ,s}-\textup{id})\xi _{t,u}).
\end{equation*}
From \eqref{xixi} and  \eqref{subcon}, we have
\begin{align*}
d(\mu _{t,s},\mu _{t,u }\circ\mu _{u ,s})
	&\lesssim |\delta \xi_{t,u,s} |_{\gamma -\theta_{2}} + |(S_{u ,s}-\textup{id})\xi _{t,u }|_{\gamma -\theta_{2}}\\
    &\lesssim 
\sum_{j=1}^{2}\omega_{1,j}^{\alpha_{1,j}}(t,u) \omega_{2,j}^{\alpha_{2,j}}(u,s)+|\xi^{1}_{t,u}|_{\gamma}|u-s|^{\theta_{2}}+|\xi^{2}_{t,u}|_{\gamma-\theta_{1}}|u-s|^{\theta_{2}-\theta_{1}}\\
&\lesssim \sum_{j=1}^{2}\omega_{1,j}^{\alpha_{1,j}}(t,u) \omega_{2,j}^{\alpha_{2,j}}(u,s)+\omega_{1}^{\alpha_{1}}(t,u)|u-s|^{\theta_{2}}+\omega_{2}^{\alpha_{2}}(t,u)|u-s|^{\theta_{2}-\theta_{1}}\\
&:=\sum_{j=1}^{4}\omega_{1,j}^{\alpha_{1,j}}(t,u) \omega_{2,j}^{\alpha_{2,j}}(u,s),
\end{align*}
which implies  $\mu $ is almost-multiplicative \eqref{almostmul01}.   Similar to the proof of Step 3 in  Theorem \ref{affyoung01}, we can also prove the growth of $\mu$, i.e., 
\begin{align*}
|\mu_{t,s}^{\mathcal{P}}|
&\leq 1+\sum_{i=1}^{2}\omega_{i}^{\alpha_{i}}(t,s)+\sum_{j=1}^{4} \omega_{1,j}^{\alpha_{1,j}}(t,s+) \omega_{2,j}^{\alpha_{2,j}}(t-,s)\\
&:=\sigma(t,s), 
\end{align*}
where $\sigma(t,s)$ is bounded, increasing in $t$ and decreasing in $s$. 
Hence, Theorem \ref{thm:mul01} and Remark \ref{MRSRRS} imply that the Riemann sums converge to a  unique limit, denoted by 
\begin{equation*}
		\mathcal{I}_t(\xi):= RRS- \lim_{|\mathcal{P}|\to 0} \sum_{[u,v]\in \mathcal{P}}S_{t,u}\xi_{v,u}. 
	\end{equation*}
Moreover, the estimate \eqref{rough-bd}  is obtained by similar arguments as Step 4 in Theorem \ref{affyoung01}. 
\end{proof}
\begin{remark} Similar to the estimate \eqref{young-bound} in Lemma \ref{YOU01}, we  can write the estimate \eqref{rough-bd} as follows
\begin{equation}
		| \delta^{S}\mathcal{I}_{t,s}(\xi)-S_{t,s}\xi_{t,s}|_{\gamma-\theta_{2}}
		\leq C \sum_{j=1}^{4} \omega_{1,j}^{\alpha_{1,j}}(t,s+) \omega_{2,j}^{\alpha_{2,j}}(t-,s),
	\end{equation}
        where $\omega_{1,j}$, $\omega_{2,j}$ are controls, $\alpha_{1,j}$, $\alpha_{2,j}$ are positive parameters such that $\min\{\alpha_{1,j}+\alpha_{2,j}\}>1$,  and  $C$ depends on $S, \min\limits_{j}\{\alpha_{1,j}+\alpha_{2,j}\}$.
\end{remark}
In what follows, we present the definitions of rough paths with finite  $p$-variation and the controlled rough paths. To simplify the notation, we use the convention: $\delta X_{t,s}=X_{t,s}:=X_t-X_s$, for all $t\geq s$. 
\begin{definition}[\textbf{Rough path}]
Let $p\in [2,3)$. We define the space of rough paths $\mathbb{V}^{p}([0,T],\mathbb{R}^{d})$ to consist of elements $(X, \mathbb{X}):=\mathbf{X}$ such that $X\in V^{p}([0,T], \mathbb{R}^{d})$, $\mathbb{X}\in V^{\frac{p}{2}}([0,T], \mathbb{R}^{d\times d})$ and that Chen's relation is satisfied
\begin{align}\label{chen}
\delta X_{t,u,s}&:=X_{t,s}-X_{t,u}-X_{u,s}=0,
\\
\delta \mathbb{X}_{t,u,s}^{i,j}&:=\mathbb{X}_{t,s}^{i,j}-\mathbb{X}_{t,u}^{i,j}-\mathbb{X}_{u,s}^{i,j}=X_{t,u}^{i}  X_{u,s}^{j},
\end{align}
for every $(t,u,s)\in \Delta_{3}$, $1\leq i, j\leq d$. The rough paths  space is equipped with norm  $$\|\mathbf{X}\|_{\mathbb{V}^{p}([0,T];\mathbb{R}^d)}:=\|X\|_{V^{p}([0,T];\mathbb{R}^d)}+\|\mathbb{X}\|_{V^{\frac{p}{2}}([0,T];\mathbb{R}^{d\times d})}<\infty.$$ 
The rough path distance is denoted by 
$$
\|\mathbf{X}; \tilde {\mathbf{X}}\|_{\mathbb{V}^{p}([0,T];\mathbb{R}^d)}:=\|X-\tilde X\|_{V^{p}([0,T];\mathbb{R}^d)}+\|\mathbb{X}-\tilde{\mathbb{X}}\|_{V^{\frac{p}{2}}([0,T];\mathbb{R}^{d\times d})}.
$$
\end{definition}
We further introduce the left and right jumps of a rough path 
\begin{align*}
\Delta_t^- \mathbf X:=\mathbf X_{t,t-}:=\lim_{s\uparrow t}\mathbf X_{t,s},\quad 
\Delta_t^+ \mathbf X:=\mathbf X_{t+,t}:=\lim_{s\downarrow t}\mathbf X_{s,t}.
\end{align*}For any fixed $t>0$, we say that $\mathbf X_{t,s}$ is right-continuous in $s\in[0,t)$, if $X_{t,s}=X_{t,s+}$, $\mathbb X_{t,s}=\mathbb X_{t,s+}$, where $X_{t,s+}:=\lim\limits_{u\downarrow s}X_{t,u}$, $\mathbb X_{t,s+}:=\lim\limits_{u\downarrow s}\mathbb X_{t,u}$. By Chen's relation, this is equivalent to
$\Delta_s^+X=0$ and $\Delta_s^+\mathbb X=0$.

For simplicity, we introduce the following notation:
 $$
V^{p}_{\alpha}([0,T]):= V^{p}([0,T]; \mathcal{B}_{\alpha}), \quad  V^{\infty}_{\alpha}([0,T]):= V^{\infty}([0,T]; \mathcal{B}_{\alpha}),
 $$
$$
V^{\frac{p}{2},p}_{\alpha-\frac{2}{p}, \alpha-\frac{1}{p}}([0,T]):=V^{\frac{p}{2}}([0,T]; \mathcal{B}_{\alpha-\frac{2}{p}})\cap V^{p}([0,T];\mathcal{B}_{\alpha-\frac{1}{p}}). 
$$

\begin{definition}[\textbf{Controlled rough path}]\label{crp}
Let $\{\mathcal{B}_{\alpha}\}_{\alpha\in\mathbb{R}}$ be a monotone family of interpolation spaces. Assume $\mathbf{X}=(X,\mathbb{X})\in \mathbb{V}^{p}([0,T]; \mathbb{R}^{d})$
is a rough path with finite $p$-variation for $p\in [2,3)$. We say that a pair $(Y, Y^{\prime})$ is controlled by $\mathbf{X}$ according to $\mathcal{B}_{\alpha}$ if 
 $$
 Y^{i}\in V^{\infty}_{\alpha}([0,T]), \quad (Y^{\prime})^{i}\in V^{p,\infty}_{\alpha-\frac{2}{p}, \alpha-\frac{1}{p}}([0,T])
 $$
 for $i=1,\dots, d.$
And the remainder $R^{Y}$
is defined as 
$$
R^{Y}_{t,s}:=
Y_{t,s}-Y_{s}^{\prime}\cdot X_{t,s}= Y_{t,s}-\sum_{i=1}^{d}(Y_{s}^{\prime})^{i}X_{t,s}^{i},
$$ belongs to $ V^{\frac{p}{2},p}_{\alpha-\frac{2}{p}, \alpha-\frac{1}{p}}([0,T]).$

Denote the space of all such controlled rough paths by $\mathcal{V}_{X}^{\frac{p}{2}}([0,T]; \mathcal{B}_{\alpha})$ or simply denote by $\mathcal{V}^{\frac{p}{2}}_{X, \alpha}$ when $T>0$ is fixed. The space $\mathcal{V}_{X, \alpha}^{\frac{p}{2}}$ is endowed  with the norm 
\begin{align*}
\|Y, Y^{\prime}\|_{\mathcal{V}^{\frac{p}{2}}_{X, \alpha}}=|Y_{0}|_{\alpha}+\|Y\|_{V^{\infty}_{\alpha}}+|Y_{0}^{\prime}|_{\alpha-\frac{1}{p}}+\|Y^{\prime}\|_{V^{p,\infty}_{\alpha-\frac{2}{p}, \alpha-\frac{1}{p}}}+\|R^{Y}\|_{V^{\frac{p}{2}, p}_{\alpha-\frac{2}{p}, \alpha-\frac{1}{p}}}.
\end{align*}
\end{definition}

In view of the above definitions, we observe that  $Y\in V^{p,\infty}_{\alpha-\frac{1}{p}, \alpha}([0,T])$ and 
\begin{align}\label{est-rough-yp}
\|Y\|_{V^p_{\alpha-\frac1p}([0,T])}\leq 
\|R^Y\|_{V^p_{\alpha-\frac 1p}([0,T])}+(|Y_0'|_{\alpha-\frac{1}{p}}+\|Y'\|_{V^{\infty}_{\alpha-\frac 1p}([0,T])})\|X\|_{V^p([0,T];\mathbb{R}^d)}. 
\end{align}
Recall that in the preliminaries, we use $\|X\|_{V^{\infty}([0,T];\mathcal{E})}:=\sup\limits_{t\in [0,T]}|X_t-X_0|_{\mathcal{E}}$. Here, if
we use the convention that $\|Y \|_{\infty, \alpha}:=\sup\limits_{0\leq t\leq T} \|Y_t\|_\alpha$, $\|Y'\|_{\infty, \alpha-\frac{1}{p}}:=\sup\limits_{0\leq t\leq T}\|Y'_t\|_{\alpha-\frac{1}{p}}$ and denote $\|Y'\|_{V^{p,\infty}_{\alpha-\frac{2}{p}, \alpha-\frac{1}{p}}}:=\|Y'\|_{V^{p}_{\alpha-\frac{2}{p}}}+\|Y'\|_{\infty, \alpha-\frac{1}{p}}$, we have 
\begin{align*}
\|Y, Y^{\prime}\|_{\mathcal{V}^{\frac{p}{2}}_{X, \alpha}}=\|Y\|_{\infty, \alpha}+\|Y^{\prime}\|_{V^{p,\infty}_{\alpha-\frac{2}{p}, \alpha-\frac{1}{p}}}+\|R^{Y}\|_{V^{\frac{p}{2},p}_{\alpha-\frac{2}{p}, \alpha-\frac{1}{p}}}.
\end{align*}
\begin{remark} 
For a  Banach space $\mathcal{E}$, $p\geq 1$, we write $C^{p-var}([0,T];\mathcal{E}):=C([0,T];\mathcal{E}) \cap V^p([0,T];\mathcal{E})$ and  $C_2^{p-var}([0,T];\mathcal{E}):=C_2([0,T];\mathcal{E}) \cap V^p([0,T];\mathcal{E})$. 
Let $p\in[2,3)$. A rough path $\mathbf{X}=(X, \mathbb X)\in \mathbb V^p([0,T];\mathbb{R}^d)$ is called  continuous, if  $X\in C^{p-var}([0,T]; \mathbb{R}^d)$ and $\mathbb{X} \in C_2^{\frac{p}{2}-var}([0,T];\mathbb{R}^{d\times d})$. A controlled rough path 
$(Y, Y')\in \mathcal{V}^{\frac p2}_{X,\alpha}$ is called continuous if
$Y\in C([0,T];\cB_\alpha)$, $Y'\in C([0,T];\cB_{\alpha-\frac{1}{p}})\cap C^{p-var}([0,T];\cB_{\alpha-\frac{2}{p}})$ with remainder $R^Y\in C_2^{p-var}([0,T];\cB_{\alpha-\frac{1}{p}}) \cap C_2^{\frac{p}{2}-var}([0,T];\cB_{\alpha-\frac{2}{p}})$. Setting $\gamma=\frac{1}{p}$, the embeddings $C\hookrightarrow V^{\infty}$,$C^{\gamma}\hookrightarrow C^{p-var}$, $C_2^{2\gamma}\hookrightarrow C_2^{\frac p2-var}$ show that  Definition \ref{crp} contains the  continuous $\gamma$-Hölder controlled rough paths framework and also allows for jumps case.
 \end{remark}
\begin{lemma}[\textbf{Integration in the rough case}]\label{int-rough}
Let   $\mathbf{X}=(X, \mathbb{X})\in \mathbb{V}^{p}([0,T]; \mathbb{R}^{d})$ for some $p\in [2,3)$ and let $\alpha\in \mathbb{R}$. 
 Let $(Y, Y^{\prime})\in \mathcal{V}^{\frac{p}{2}}_{X, \alpha}([0,T]).$
Then the integral 
\begin{align}\label{roughint01}
\int_{0}^{t}S_{t,s} Y_{s}\cdot \textup{d}\mathbf{X}_{s}:=RRS-\lim\limits_{|\mathcal{P}|\rightarrow 0}\sum\limits_{(u,v)\in \mathcal{P}} S_{t,u}(Y_{u}\cdot X_{v,u}+Y_{u}^{\prime}:\mathbb{X}_{v,u}),
\end{align}
exists as an element of $\mathcal{B}_{\alpha-\frac{2}{p}}$, where $Y_{u}^{\prime}: \mathbb{X}_{v,u}:=\sum\limits_{1\leq i,j\leq d}(Y^{\prime})^{ij}\mathbb{X}_{v,u}^{ij}.$

Moreover, 
we have the estimate
\begin{equation}\label{rough-bound}
\begin{split}
&|\int_{s}^{t} S_{t,u} Y_{u}\cdot \textup{d}\mathbf{X}_{u}-S_{t,s}(Y_{s}\cdot X_{t,s}+Y_{s}^{\prime}:\mathbb{X}_{t,s})|_{\alpha-\frac{2}{p}} \\
&\leq C [ \|X\|_{V^{p}((s,t]; \mathbb{R}^{d})}\|R^{Y}\|_{V^{\frac{p}{2}}([s,t); \mathcal{B}_{\alpha-\frac{2}{p}})}+
\|\mathbb{X}\|_{V^{\frac{p}{2}}((s,t]; \mathbb{R}^{d\times d})}\|Y^{\prime}\|_{V^{p}([s,t); \mathcal{B}_{\alpha-\frac{2}{p}})}\\
&~~+ \|Y, Y^{\prime}\|_{\mathcal{V}^{\frac{p}{2}}_{X, \alpha}([0,T])}\|X\|_{V^{p}((s,t];  \mathbb{R}^{d})} |t-s|^{\frac{2}{p}}+ \|Y, Y^{\prime}\|_{\mathcal{V}^{\frac{p}{2}}_{X, \alpha}([0,T])}\|\mathbb{X}\|_{V^{\frac{p}{2}}((s,t];  \mathbb{R}^{d\times d})}|t-s|^{\frac{1}{p}}].
\end{split}
\end{equation}
where $C$ depends on $S, p$.  In particular, if $\mathbf{X}$ is c\`adl\`ag, then the convergence in \eqref{roughint01}
 holds in MRS-sense. 
 \begin{remark} 
The estimate in \eqref{rough-bound} can be written in the form of control functions.
\begin{align*}
|\int_{s}^{t} S_{t,u} Y_{u}\cdot \textup{d}\mathbf{X}_{u}-S_{t,s}(Y_{s}\cdot X_{t,s}+Y_{s}^{\prime}:\mathbb{X}_{t,s})|_{\alpha-\frac{2}{p}} \leq 
C\sum\limits_{j=1}^{4}\omega_{1,j}^{\alpha_{1,j}}(t,s+)\omega_{2,j}^{\alpha_{2,j}}(t-,s), 
 \end{align*}
 where 
 \begin{align*}
&\omega_{1,1}(t,s+)=  \|X\|_{V^{p}((s,t]; \mathbb{R}^{d})}^p ,\quad  \omega_{2,1}(t-,s)= \|R^{Y}\|_{V^{\frac{p}{2}}([s,t); \mathcal{B}_{\alpha-\frac{2}{p}})}^{\frac p2} , 
\\
&\omega_{1,2}(t,s+)=\|\mathbb{X}\|^{\frac p2}_{V^{\frac{p}{2}}((s,t]; \mathbb{R}^{d\times d})}  ,\quad  \omega_{2,2}(t-,s)=\|Y^{\prime}\|^p_{V^{p}([s,t); \mathcal{B}_{\alpha-\frac{2}{p}})},\\
&\omega_{1,3}(t,s
+)=\|Y, Y^{\prime}\|^p_{\mathcal{V}^{\frac{p}{2}}_{X, \alpha}([0,T])}\|X\|^p_{V^{p}((s,t];  \mathbb{R}^{d})}  ,\quad  \omega_{2,3}(t,s)= |t-s|,\\
&\omega_{1,4}(t,s+)=\|Y, Y^{\prime}\|^{\frac p2}_{\mathcal{V}^{\frac{p}{2}}_{X, \alpha}([0,T])}\|\mathbb{X}\|^{\frac p2}_{V^{\frac{p}{2}}((s,t];  \mathbb{R}^{d\times d})},\quad  \omega_{2,4}(t,s)=|t-s|,\\
&\alpha_{1,1}=\frac 1p,~ \alpha_{1,2}=\frac 2p, ~\alpha_{2,1}=\frac 2p, ~\alpha_{2,2}=\frac 1p,\\
&\alpha_{1,3}=\frac 1p, 
~\alpha_{2,3}=\frac 2p,
~\alpha_{1,4}=\frac 2p,~ \alpha_{2,4}=\frac 1p.
 \end{align*}
\end{remark}
\end{lemma}
\begin{proof}[Proof of Lemma \ref{int-rough}]
For $(t,s)\in \Delta_{2}$, $(t,u,s)\in \Delta_{3}$, we have 
	$$\xi_{t,s}=Y_{s}\cdot X_{t,s}+Y_{s}^{\prime}:\mathbb{X}_{t,s}:=\xi^{1}_{t,s}+\xi^{2}_{t,s}
    $$ and
$$
\delta \xi_{t,u,s}
=-X_{t,u} \cdot R_{u,s}^{Y}+\mathbb{X}
_{t,u}
:Y^{\prime}_{u,s}.
$$
Then, we get
\begin{align*}
|\xi^{1}_{t,s}|_{\alpha}&=|Y_{s}\cdot  X_{t,s}|_{\alpha}\\
&\leq |Y_{s}|_{\alpha}\|X\|_{V^{p}([s,t]; \mathbb{R}^{d})}\\
&\leq (|Y_0|_{\alpha}+\|Y\|_{V^\infty_\alpha([0,s])})\omega^{\frac{1}{p}}_X(t,s)\\
&\leq  \|Y, Y^{\prime}\|_{\mathcal{V}^{\frac{p}{2}}_{X, \alpha}([0,s])} \omega^{\frac{1}{p}}_X(t,s).
\end{align*}
\begin{align*}
|\xi^{2}_{t,s}|_{\alpha-\frac{1}{p}}&=|Y_{s}^{\prime}:\mathbb{X}_{t,s}|_{\alpha-\frac{1}{p}}\\
&\leq |Y_{s}^{\prime}|_{\alpha-\frac{1}{p}}\|\mathbb{X}\|_{V^{\frac{p}{2}}([s,t]; \mathbb{R}^{d\times d})}\\
&\leq  (|Y'_0|_{\alpha-\frac{1}{p}}+\|Y'\|_{V^\infty_{\alpha-\frac{1}{p}}([0,s])})\omega^{\frac{2}{p}}_{\mathbb X}(t,s)\\
&\leq \|Y, Y^{\prime}\|_{\mathcal{V}^{\frac{p}{2}}_{X, \alpha}([0,s])}\omega^{\frac{2}{p}}_{\mathbb{X}}(t,s).
\end{align*}
\begin{align*}
|\delta \xi_{t,u,s}|_{\alpha-\frac{2}{p}}&=|X_{t,u}\cdot  R_{u,s}^{Y}+\mathbb{X}
_{t,u}: Y_{u,s}^{\prime}|_{\alpha-\frac{2}{p}}\\
&\leq \|X\|_{V^{p}([u,t]; \mathbb{R}^{d})}\|R^{Y}\|_{V^{\frac{p}{2}}([s,u]; \mathcal{B}_{\alpha-\frac{2}{p}})}+\|\mathbb{X}\|_{V^{\frac{p}{2}}([u,t];  \mathbb{R}^{d\times d})}\|Y^{\prime}\|_{V^{p}([s,u]; \mathcal{B}_{\alpha-\frac{2}{p}})}\\
&:=\omega^{\frac{1}{p}}_X(t,u)\omega_{R^Y}^{
\frac{2}{p}
}(u,s)+\omega_{\mathbb X}^{\frac{2}{p}}(t,u)\omega_{Y'}^{\frac{1}{p}}(u,s).
\end{align*}
From \eqref{smooth}, we have
\begin{align*}
|(S_{u,s}-\textup{id})\xi_{t,u}|_{\alpha-\frac{2}{p}}&\leq C (|\xi_{t,u}^1|_\alpha |u-s|^{\frac{2}{p}}+|\xi_{t,u}^2|_{\alpha-\frac{1}
{p}}|u-s|^{\frac{1}{p}})\\
&=C(|Y_{u}\cdot  X_{t,u}|_{\alpha}|u-s|^{\frac{2}{p}}+
|Y_{u}^{\prime}: \mathbb{X}_{t,u}|_{\alpha-\frac{1}{p}}|u-s|^{\frac{1}{p}})\\
&\leq C(|Y_{u}|_{\alpha}\|X\|_{V^{p}([u,t]; \mathbb{R}^{d})}|u-s|^{\frac{2}{p}}+|Y_{u}^{\prime}|_{\alpha-\frac{1}{p}}\|\mathbb{X}\|_{V^{\frac{p}{2}}([u,t]; \mathbb{R}^{d\times d})}|u-s|^{\frac{1}{p}})\\
&\leq C \|Y, Y^{\prime}\|_{\mathcal{V}^{\frac{p}{2}}_{X, \alpha}[0,u]}[\omega^{\frac{1}{p}}_X(t,u)|u-s|^{\frac{2}{p}}+\omega_{\mathbb X}^{\frac{2}{p}}(t,u)|u-s|^{\frac{1}{p}}].
\end{align*}
Hence, 
\begin{align}
&|\delta \xi_{t,u,s}+(S_{u,s}-\textup{id})\xi_{t,u}|_{\alpha-\frac{2}{p}}\\\leq&~~  C  [\omega^{\frac{1}{p}}_X(t,u)\omega_{R^Y}^{
\frac{2}{p}
}(u,s)+\omega_{\mathbb X}^{\frac{2}{p}}(t,u)\omega_{Y'}^{\frac{1}{p}}(u,s) \nonumber \\
&~~+
 \|Y, Y^{\prime}\|_{\mathcal{V}^{\frac{p}{2}}_{X, \alpha}([0,T])} \omega^{\frac{1}{p}}_X(t,u)|u-s|^{\frac{2}{p}}+ \|Y, Y^{\prime}\|_{\mathcal{V}^{\frac{p}{2}}_{X, \alpha}([0,T])}\omega^{\frac{2}{p}}_{\mathbb{X}}(t,u)|u-s|^{\frac{1}{p}}] \nonumber \\
:=&~~C [\sum\limits_{j=1}^{2}\omega_{1,j}^{\alpha_{1,j}}(t,u)\omega_{2,j}^{\alpha_{2,j}}(u,s)+\omega_1^{\alpha_1}(t,u)|u-s|^{\theta_2}+\omega_2^{\alpha_{2}}(t,u)|u-s|^{\theta_2-\theta_1}], \label{rough-N4}
\end{align}
where $\omega_{1,1}(t,s)=\omega_{X}(t,s)$, $\omega_{1,2}(t,s)=\omega_{\mathbb X}(t,s)$, $\omega_{2,1}(t,s)=\omega_{R^Y}(t,s)$, $\omega_{2,2}(t,s)=\omega_{Y'}(t,s)$, $\omega_1(t,s)=\|Y, Y'\|^p_{\mathcal{V}^{\frac p2}_{X, \alpha}([0,T])}\omega_X(t,s)$, $\omega_2(t,s)=\|Y, Y'\|^{\frac{p}{2}}_{\mathcal{V}^{\frac p2}_{X, \alpha}([0,T])}\omega_{\mathbb X}(t,s)$.
Therefore, by using Theorem \ref{affrough01} with the case that  $\gamma=\alpha$, $\alpha_{1}=\theta_{1}=\frac{1}{p}$, $\alpha_{2}=\theta_{2}=\frac{2}{p}$,  we have the required results.  
 Note that \eqref{rough-N4} can be written as a sum of products of some control functions, i.e., 
\begin{align*}
|\delta \xi_{t,u,s}+(S_{u,s}-\textup{id})\xi_{t,u}|_{\alpha-\frac{2}{p}}\leq C  \sum\limits_{j=1}^{4}\omega_{1,j}^{\alpha_{1,j}}(t,u)\omega_{2,j}^{\alpha_{2,j}}(u,s),
\end{align*}
as stated in 
Theorem \ref{thm:mul01} for $N=4$. 
\end{proof}



\begin{remark}
The choice of spatial regularity is consistent with the Young case.
Indeed, setting $\gamma=\alpha+\frac{1}{p}$ in Lemma \ref{int-rough}, 
the estimate \eqref{rough-bound} is obtained in
$\mathcal B_{\alpha-\frac{1}{p}}$  for  $(Y, Y')\in \mathcal{V}^{\frac{p}{2}}_{X, \alpha+\frac{1}{p}}$ and the remainder $R^{Y}\in V^{\frac{p}{2},p}_{\alpha-\frac{1}{p}, \alpha}$. 
On the other hand, applying Lemma \ref{YOU01} with $\alpha$ replaced by
$\alpha-\frac{1}{p}$ and $\theta=\frac{1}{p}$ gives the corresponding
Young estimate in the same space $\mathcal B_{\alpha-\frac{1}{p}}$;   see Remark \ref{exp-young} for more details.

\end{remark}

In order to prove that the solution remains a regulated path even in the highest regularity space, we give the following lemma. Please see \cite[Lemma 2.9]{FZ18} for finite-dimensional result.
\begin{lemma}\label{rough-regular}
Let $p\in[2,3)$, $\alpha\in \mathbb{R}$. Suppose $\mathbf{X}=(X,\mathbb X)\in\mathbb V^p([0,T],\mathbb R^d)$, and $(Y, Y')\in \mathcal{V}^{\frac p2}_{X, \alpha+\frac{2}{p}}$ is a controlled rough path.  Then for $0\leq s< t\leq T$, 

(1) the right-hand limit $\lim\limits_{u\downarrow s}\int_u ^t S_{t,r}Y_r\cdot \textup{d}\mathbf{X}_r:=\int_{s+}^t S_{t,r} Y_r\cdot \textup{d}\mathbf{X}_r$ exists in $\cB_{\alpha}$.
\begin{align*}
\int_{s+}^t S_{t,r}Y_r\cdot \textup{d}\mathbf{X}_r=\int_s^t S_{t,r}Y_r\cdot \textup{d}\mathbf{X}_r-S_{t,s}(Y_s \Delta_s^+ X+Y_s':\Delta_s^+\mathbb{X}).
\end{align*}

(2) The left-hand limit $\lim \limits_{u\uparrow s}\int_u^t S_{t,r}  Y_r \cdot \textup{d}\mathbf{X}_r:=\int_{s-}^t S_{t,s}Y_r\cdot \textup{d}\mathbf{X}_r$ exists in $\cB_{\alpha}$.
\begin{align*}
\int_{s-}^t S_{t,s}Y_r\cdot \textup{d}\mathbf{X}_r=\int_{s}^t S_{t,r}Y_r\cdot \textup{d}\mathbf{X}_r+S_{t,s}(Y_{s-}\Delta_{s}^-X+Y_{s-}': \Delta_s^- \mathbb{X}).
\end{align*}

(3)  The right-hand limit $\lim\limits_{u\downarrow t}\int_s^u S_{u,r}Y_r\cdot \textup{d}\mathbf{X}_r:=\int_{s}^{t+}S_{t,r} Y_r \cdot \textup{d}\mathbf{X}_r$ exists in $\cB_{\alpha}$.
\begin{align*}
\int_s^{t+}S_{t,r} Y_r\cdot \textup{d}\mathbf{X}_r=\int_s^{t}S_{t,r} Y_r \cdot \textup{d}\mathbf X_r+Y_{t} \cdot \Delta_t^
+ X+Y_t':\Delta_t^+ \mathbb{X}.
\end{align*}

(4)
 Furthermore, suppose $(Y,Y')\in
\mathcal V^{\frac p2}_{X,\alpha+\frac{2}{p}
+\varepsilon}([0,T])$, $\varepsilon>0$, the left-hand limit $\int_{s}^{t-}S_{t,r}Y_r\cdot \textup{d}\mathbf{X}_r:=\lim\limits_{u\uparrow t} \int_s^u S_{u,r}Y_r\cdot \textup{d}\mathbf{X}_r$ exists in $\cB_{\alpha}$.
\begin{align*}
\int_{s}^{t-}S_{t,r} Y_r\cdot \textup{d}\mathbf{X}_r=  \int_s^tS_{t,r}Y_{r}\cdot \textup{d}\mathbf{X}_{r}-Y_{t-}\cdot \Delta_t^- X-Y_{t-}' :\Delta_t^- \mathbb{X}.
\end{align*}
\end{lemma}
\begin{proof}
The idea of this proof is similar to that in Lemma \ref{llrr}. Here, we give a brief proof of the key part (3). 
Construct auxiliary process $\tilde{\mathbf{X}}=(\tilde X, \tilde {\mathbb{X}})$:
\[
\tilde X_{m,u}:=\left\{
\begin{array}{ll}
	X_{m,t+}, & u=t \\
	X_{m,u},& u>t,
\end{array}
\right.
\ \ \
\tilde{\mathbb X}_{m,u}:=\left\{
\begin{array}{ll}
	\mathbb X_{m, t+}, &  u=t, \\
	\mathbb X_{m,u},& u>t,\\
\end{array}
\right.
\]
Then one can  check that for any fixed $m$, $\tilde{\mathbf{X}}_{m,u}$ is right-continuous at $u=t$ and $\tilde{\textbf{X}}$ satisfies Chen's relation. In view of \eqref{tilde-}, for any $u\in (t, T]$,  we have $\|\tilde X\|_{V^p([t,u];\mathbb{R}^d)}\leq \|X\|_{V^{p}((t,u];\mathbb{R}^d)}$, $\|\tilde{\mathbb{X}}\|_{V^{\frac p2}([t,u];\mathbb{R}^{d\times d})}\leq \|\tilde{\mathbb{X}}\|_{V^{\frac p2}((t,u];\mathbb{R}^{d\times d})}$. 
We verify that $(Y,Y')$ is controlled by $\tilde X$, i.e., $ (Y,Y')\in \mathcal V^{\frac p2}_{\tilde X,\alpha+\frac 2p}([t,T])$ with the remainder:
\begin{equation}
\tilde R^Y_{v,u} :=Y_{v,u}-Y'_u\cdot \tilde X_{v,u}=\begin{cases}
R^Y_{v,t}+Y'_t \cdot \Delta_t^+X,&u=t,\\
R^Y_{v,u},&u>t.
\end{cases}
\end{equation}
 Minkowski’s inequality yields
 \begin{align*}
   \|\tilde R^Y\|_{V^{p}([t,T];\mathcal B_{\alpha+\frac 1p})} &\lesssim   \|R^Y\|_{V^{p}([t,T];\mathcal B_{\alpha+\frac 1p})} + |Y'_t \cdot \Delta_t^+X|_{\alpha+\frac 1p} <\infty. \\
  \|\tilde R^Y\|_{V^{\frac p2}([t,T];\mathcal B_\alpha)} &\lesssim \|R^Y\|_{V^{\frac p2}([t,T];\mathcal B_\alpha)} + |Y'_t \cdot \Delta_t^+X|_\alpha <\infty. 
\end{align*}
Thus, $\tilde R^Y\in V^{\frac p2,p}_{\alpha, \alpha+\frac 1p}([t,T])$. Now, Lemma \(\ref{int-rough}\)  ensures that the rough integral against
$ \tilde{\mathbf X}$
is well defined in $\cB_\alpha$. 

For $h>0$, 
\begin{align*}
&|\int_s^{t+h}S_{t+h,r}Y_r\cdot \textup{d}\mathbf{X}_r-\int_s^t S_{t,r}Y_r\cdot \textup{d}\mathbf{X}_r-Y_t \Delta_t^+ X-Y_t': \Delta_t^+ \mathbb{X}|_{\alpha}\\
 &\leq  |(S_{t+h, t}-\textup{id})\int_s^t S_{t,r}Y_r\cdot \textup{d}\mathbf X_r|_{\alpha}
+|\int_t^{t+h}S_{t+h, r}Y_r \cdot \textup{d} \mathbf X_r-Y_t \cdot \Delta_t^+ X-Y_t' : \Delta_t^+ \mathbb{X}|_\alpha\\
& \leq  |(S_{t+h, t}-\textup{id})\int_s^t S_{t,r}Y_r\cdot \textup{d}\mathbf X_r|_{\alpha}\\
&~~+|\int_t^{t+h}S_{t+h, r}Y_r \cdot \textup{d} (\mathbf X-\tilde{\mathbf{X}})_r-Y_t \cdot \Delta_t^+ X-Y_t' : \Delta_t^+ \mathbb{X}|_\alpha\\
 &~~+|\int_t^{t+h}S_{t+h, r}Y_r\cdot \textup{d}\tilde {\mathbf{X}}_r|_\alpha\\
&:= I_1+I_2+I_3.
\end{align*}
Due to the continuity of the propagator $S$ in Remark \ref{smooth}, we have $I_1\rightarrow 0$,  as $h\rightarrow 0$. 

Denote $\mathcal P=\{t=r_0<r_1<\cdots<r_n=t+h\}$. As shown in \eqref{roughint01}, we have 
$$\int_t^{t+h} S_{t+h,r} Y_r\cdot \textup{d}\mathbf X_r:=RRS-\lim_{|\mathcal{P}|\to 0}\sum_{(u,v)\in \mathcal P}S_{t+h,u}(Y_u\cdot X_{v,u}+Y_u':\mathbb X_{v,u})
$$
and 
$$
\int_t^{t+h}S_{t+h, r}Y_r\cdot \textup{d}\tilde{\mathbf X}_r:=RRS-\lim_{|\mathcal{P}|\to 0} \sum_{(u,v)\in \mathcal P}S_{t+h,u}(Y_u\cdot \tilde X_{v,u}+Y_u':\tilde{\mathbb X}_{v,u}.)
$$ For $t<r<r_1$, Chen's relation gives
\begin{align*}
\mathbb X_{r_1,t}
=
\mathbb X_{r_1,r}
+\mathbb X_{r,t}
+X_{r_1,r}\otimes X_{r,t}.
\end{align*}
Keeping $r_1$ fixed and letting $r\downarrow t$, we obtain
\begin{align*}
\mathbb X_{r_1,t}
=
\mathbb X_{r_1,t+}
+\Delta_t^+\mathbb X
+(X_{r_1}-X_{t+})\otimes\Delta_t^+X.
\end{align*}
Since
$\tilde{\mathbb X}_{r_1,t}=\mathbb X_{r_1,t+}$,
it follows that
\begin{align*}
\mathbb X_{r_1,t}-\tilde{\mathbb X}_{r_1,t}
=
\Delta_t^+\mathbb X
+(X_{r_1}-X_{t+})\otimes\Delta_t^+X.
\end{align*}
By comparing the  Riemann sums above,
we have 
\begin{equation}\label{rsum-compare}
\begin{split}
&\sum_{(u,v)\in \mathcal P}S_{t+h, u}[Y_u\cdot (X_{v,u}-\tilde X_{v,u})+Y_u':(\mathbb X_{v,u}-\tilde{\mathbb X}_{v,u})] \\
=&~S_{t+h,t}
[
Y_t\cdot(X_{r_1,t}-\tilde X_{r_1,t})
+Y'_t:(\mathbb X_{r_1,t}-\tilde{\mathbb X}_{r_1,t})
]  \\
=&~S_{t+h,t}
[Y_t\cdot\Delta_t^+X
      +Y'_t:( \Delta_t^+\mathbb X+
(X_{r_1}-X_{t+})\otimes\Delta_t^+X)].
\end{split}
\end{equation}
Since $X_{r_1}\to X_{t+}$ as $r_1\downarrow t$,
for every $\varepsilon>0$, there exists a refinement
$\mathcal P_\varepsilon$ of $\mathcal P$ such that, 
\begin{align*}
\sum_{(u,v)\in \mathcal P_\varepsilon}S_{t+h, u}[Y_u\cdot (X_{v,u}-\tilde X_{v,u})+Y_u':(\mathbb X_{v,u}-\tilde{\mathbb X}_{v,u})] 
=S_{t+h,t}
[Y_t\cdot\Delta_t^+X
      +Y'_t:\Delta_t^+\mathbb X].
\end{align*}
Thus, by \eqref{rsum-compare} and the definition of RRS convergence,
$$
\int_t^{t+h} S_{t+h,r} Y_r \cdot \textup{d}(\mathbf X-\tilde{\mathbf X})_r=S_{t+h,t}( Y_t\cdot\Delta_t^+ X+Y_t' : \Delta_t^+ \mathbb{X}).
$$
Using the smoothing property of $S$,  we obtain 
\begin{align*}
I_2=|(S_{t+h, t}-\textup{id})(Y_t \cdot \Delta_t^+ X+Y_t' : \Delta_t^+ \mathbb{X})|_\alpha\rightarrow 0, \quad \text{as}\quad h\rightarrow 0. 
\end{align*}
According to \eqref{rough-bound} and the right-continuity of $\tilde{\mathbf X}$,  
\begin{align*}
I_3&\lesssim \|\tilde X\|_{V^p((t, t+h];\mathbb{R}^d)}\|\tilde R^Y\|_{V^{\frac p 2}([t, t+h);\cB_\alpha)}+\|\tilde{\mathbb{X}}\|_{V^{\frac p2}((t,t+h];\mathbb{R}^{d\times d})}\|Y'\|_{V^{p}([t,t+h);\cB_\alpha)}\\
&~~+\|Y, Y'\|_{\mathcal{V}^{\frac{p}{2}}_{\tilde X, \alpha+\frac 2p}([t,T])}(\|\tilde X\|_{V^p((t,t+h];\mathbb{R}^d)}h^{\frac 2p}+\|\tilde{\mathbb{X}} \|_{V^{\frac p2}((t,t+h];\mathbb{R}^{d\times d})}h^{\frac 1p})\\
&~~+|S_{t+h, t}(Y_t\cdot \tilde{X}_{t+h, t}+Y_t':\tilde{\mathbb{X}}_{t+h, t})|_\alpha\\
&\lesssim (1+h^{\frac 1p})(\|\tilde X\|_{V^p([t,t+h];\mathbb{R}^d)}+\|\tilde{\mathbb{X}}\|_{V^{\frac p2}([t,t+h];\mathbb{R}^{d\times d})})\rightarrow 0, \quad \text{as}\quad h\rightarrow 0. 
\end{align*}
Therefore, we obtain the required result in (3). 

To prove the existence of the left-hand limit in (4), we use the following inequality, analogous to that for $I_3$ in the proof of Lemma \ref{llrr} (4). We need additional spatial regularity  of $(Y,Y')$ in order to ensure that rough convolution is bounded in $\cB_{\alpha+\varepsilon}$.
\begin{align*}
|(\textup{id}-S_{t,t-h})\int_s^{t-h} S_{t-h, r}Y_r \cdot \textup{d}\mathbf X_r)|_\alpha\lesssim h^{\varepsilon}|\int_s^{t-h} S_{t-h, r}Y_r \cdot \textup{d}\mathbf X_r|_{\alpha+\varepsilon}.
\end{align*} The rest of the proof is similar to (3) and Lemma \ref{llrr}, so we omit it.
\end{proof}
\begin{lemma}\label{control-roughint}
Let $p\in[2,3)$, $\alpha\in \mathbb{R}$, $\varepsilon>0$  and let $\mathbf{X}=(X, \mathbb{X})\in \mathbb{V}^{p}([0,T];\mathbb{R}^d)$.
If $(Y, Y')\in \mathcal{V}^{\frac{p}{2}}_{X, \alpha+\varepsilon}([0,T])$, 
then $$(Z, Z'):=(\int_0^\cdot S_{\cdot, r}Y_r \cdot \textup{d}\mathbf{X}_r, Y)\in \mathcal{V}^{\frac{p}{2}}_{X, \alpha-\frac{2}{p}}([0,T]).$$   Moreover, 
\begin{align}\label{bound-Z}
\|Z, Z'\|_{\mathcal V^{\frac p2}_{X, \alpha-\frac 2p}([0,T])}\leq C (1+\|\mathbf X\|_{\mathbb V^p([0,T];\mathbb R^d)})\|Y,Y'\|_{\mathcal V^{\frac p2}_{X, \alpha}([0,T])},  
\end{align}
where $C$ depends on $T, S, p$. 
\end{lemma}
\begin{proof}
The continuous embeddings of the Banach spaces, $\cB_{\alpha+\varepsilon}\hookrightarrow \cB_\alpha$, imply that  $(Y, Y')\in \mathcal{V}^{\frac{p}{2}}_{X, \alpha}$. Thus Lemma $\ref{int-rough}$ ensures that $Z_t$ is well defined in $\cB_{\alpha-\frac{2}{p}}$.   Lemma \ref{rough-regular} with spatial index $\alpha-\frac{2}{p}$, shows that 
 $Z$ is a regulated path in $\cB_{\alpha-\frac{2}{p}}$. Combining \eqref{smooth} and \eqref{rough-bound} gives
 \begin{equation}\label{est-Z}
\begin{split}
\|Z\|_{V^\infty_{\alpha-\frac 2p}}&=
\sup\limits_{t\in[0,T]}|Z_t-Z_0|_{\alpha-\frac{2}{p}}\\
&\lesssim \|X\|_{V^p((0,t];\mathbb R^d)}\|R^Y\|_{V^{\frac p2}([0,t);\cB_{\alpha-\frac 2p})}+\|\mathbb X\|_{V^{\frac p2}(0,t];\mathbb R^{d\times d})}\|Y'\|_{V^p([0,t);\cB_{\alpha-\frac 2p})}\\
&~~+\|Y,Y'\|_{\mathcal V^{\frac p2}_{X, \alpha}}(\|X\|_{V^p((0,t];\mathbb R^d)}t^{\frac 2p}+\|\mathbb X\|_{V^{\frac p2}((0,t];\mathbb R^{d\times d})}t^{\frac 1p})\\
&~~+|S_{t,0}Y_0\cdot X_{t,0}|_{\alpha-\frac 2p}+|S_{t,0}Y_0':\mathbb X_{t,0}|_{\alpha-\frac 2p}\\
&\lesssim \|Y,Y'\|_{\mathcal V^{\frac p2}_{X, \alpha}}[(1+T^{\frac 2p})\|X\|_{V^p([0,t];\mathbb R^d)}+(1+T^{\frac 1p})\|\mathbb X\|_{V^{\frac p2}([0,t];\mathbb R^{d\times d})}].
\end{split}
\end{equation}
Hence, we have $Z\in V^\infty_{\alpha-\frac{2}{p}}([0,T])$.
Next, it is not hard to check that $Y$ has the required regularity as the Gubinelli derivative of
$Z$. In view of \eqref{est-rough-yp} and 
using the continuous embedding of the Banach spaces, $\cB_{\alpha-\frac{1}{p}}\hookrightarrow \cB_{\alpha-\frac{4}{p}}$,  we have  $Y\in V^{p}_{\alpha-\frac{4}{p}}([0,T])$. On the other hand, $V^\infty_{\alpha}\hookrightarrow V^\infty_{\alpha-\frac{3}{p}}$. Thus, 
$Y\in V^{p,\infty}_{\alpha-\frac{4}{p}, \alpha-\frac{3}{p}}([0,T])$ and 
\begin{align}\label{est-Z'}
|Z_0'|_{\alpha-\frac 3p}+\|Z'\|_{V^{p,\infty}_{\alpha-\frac 4p, \alpha-\frac 3p}}\lesssim (1+\|X\|_{V^p([0,T];\mathbb R^d)})\|Y,Y'\|_{\mathcal V^{\frac p2}_{X, \alpha}}. 
\end{align}

It remains to prove  $R^Z\in V^{\frac{p}{2}, p}_{\alpha-\frac{4}{p}, \alpha-\frac{3}{p}}$. We observe that 
 \begin{align*}
R_{t,s}^{Z}&=Z_{t,s}-Y_s\cdot X_{t,s}\\
&=\int_s^t S_{t,r}Y_r \cdot \textup{d}\mathbf{X}_r-Y_s\cdot X_{t,s}+(S_{t,s}-\textup{id})\int_0^s S_{s,r}Y_r\cdot \textup{d}\mathbf{X}_r\\
&=\int_s^t S_{t,r}Y_r \cdot \textup{d} \mathbf{X}_r-S_{t,s}(Y_s \cdot X_{t,s}+Y_s': \mathbb{X}_{t,s})\\
&~~+ (S_{t,s}-\textup{id})Y_s \cdot X_{t,s}\\
&~~
+(S_{t,s}-\textup{id})\int_0^s S_{s,r}Y_r\cdot \textup{d}\mathbf{X}_r\\
&~~+S_{t,s}Y_s': \mathbb{X}_{t,s}\\
&:=I_1(t,s)+I_2(t,s)+I_3(t,s)+I_4(t,s).
\end{align*}
We first prove that each term has finite $\frac{p}{2}$-variation in $\cB_{\alpha-\frac{4}{p}}$.
Using \eqref{rough-bound}, we have 
\begin{align*}
|I_1(t,s)|_{\alpha-\frac{4}{p}}&\lesssim |I_1(t,s)|_{\alpha-\frac{2}{p}}\\
&\lesssim  \|X\|_{V^{p}((s,t]; \mathbb{R}^{d})}\|R^{Y}\|_{V^{\frac{p}{2}}([s,t); \mathcal{B}_{\alpha-\frac{2}{p}})}+
\|\mathbb{X}\|_{V^{\frac{p}{2}}((s,t]; \mathbb{R}^{d\times d})}\|Y^{\prime}\|_{V^{p}([s,t); \mathcal{B}_{\alpha-\frac{2}{p}})}\\
&~~+\|Y, Y^{\prime}\|_{\mathcal{V}^{\frac{p}{2}}_{X, \alpha}} (\|X\|_{V^{p}((s,t];  \mathbb{R}^{d})} |t-s|^{\frac{2}{p}}+\|\mathbb{X}\|_{V^{\frac{p}{2}}((s,t];  \mathbb{R}^{d\times d})}|t-s|^{\frac{1}{p}}). 
\end{align*}
For any partition $\mathcal{P}$ of $[0,T]$, we have 
\begin{align*}
\sum\limits_{(s,t)\in \mathcal P} |I_1(t,s)|_{\alpha-\frac{4}{p}}^{\frac p2}&\lesssim  \sum\limits_{(s,t)\in \mathcal P}[\|X\|_{V^{p}((s,t]; \mathbb{R}^{d})}\|R^{Y}\|_{V^{\frac{p}{2}}([s,t); \mathcal{B}_{\alpha-\frac{2}{p}})}+
\|\mathbb{X}\|_{V^{\frac{p}{2}}((s,t]; \mathbb{R}^{d\times d})}\|Y^{\prime}\|_{V^{p}([s,t); \mathcal{B}_{\alpha-\frac{2}{p}})}\\
&~~+\|Y, Y^{\prime}\|_{\mathcal{V}^{\frac{p}{2}}_{X, \alpha}} (\|X\|_{V^{p}((s,t];  \mathbb{R}^{d})} |t-s|^{\frac{2}{p}}+\|\mathbb{X}\|_{V^{\frac{p}{2}}((s,t];  \mathbb{R}^{d\times d})}|t-s|^{\frac{1}{p}})]^{\frac p2}\\
&\lesssim \|X\|_{V^p([0,T];\mathbb{R}^d)}^{\frac p2}
\|R^Y\|_{V^{\frac p2}_{\alpha-\frac{2}{p}}}^{\frac p2}
+
\|\mathbb X\|_{V^{\frac p2}([0,T];\mathbb{R}^{d\times d})}^{\frac p2}
\|Y'\|_{V^p_{\alpha-\frac{2}{p}}}^{\frac p2}\\
&~~+
\|Y,Y'\|_{\mathcal V^{\frac p2}_{X,\alpha}}^{\frac p2}
\left(
T\|X\|_{V^p([0,T];\mathbb{R}^d)}^{\frac p2}
+
T^{\frac12}\|\mathbb X\|_{V^{\frac p2}([0,T];\mathbb{R}^{d\times d})}^{\frac p2}
\right).
\end{align*}
Taking the supremum over all partitions yields $I_1\in V^{\frac{p}{2}}_{\alpha-\frac{4}{p}}$. Applying \eqref{smooth}, we have  
\begin{align*}
|I_2(t,s)|_{\alpha-\frac{4}{p}}&=|(S_{t,s}-\textup{id})Y_s \cdot X_{t,s}|_{\alpha-\frac{4}{p}}\\
&\lesssim |t-s|^{\frac{2}{p}}|Y_s \cdot X_{t,s}|_{\alpha-\frac{2}{p}}\\
&\lesssim (\|Y\|_{V^{\infty}_\alpha}+|Y_0|_\alpha)|X_{t,s}||t-s|^{\frac{2}{p}}, 
\end{align*}
which allows us to derive
\begin{align*}
\sum\limits_{(s,t)\in \mathcal{P}}|I_2(t,s)|^{\frac{p}{2}}_{\alpha-\frac{4}{p}}&\lesssim T\|X\|^{\frac p2}_{V^p([0,T];\mathbb{R}^d)}(\|Y\|_{V^{\infty}_\alpha}+|Y_0|_\alpha)^{\frac{p}{2}}. 
\end{align*}
Hence, $I_2\in V^{\frac{p}{2}}_{\alpha-\frac{4}{p}}$. Similarly, we have 
\begin{align*}
|I_3(t,s)|_{\alpha-\frac{4}{p}}&=|(S_{t,s}-\textup{id})\int_0^s S_{s,r}Y_r\cdot \textup{d}\mathbf{X}_r|_{\alpha-\frac{4}{p}}\\
&\lesssim |t-s|^{\frac{2}{p}}|Z_s|_{\alpha-\frac{2}{p}}.
\end{align*}
Then we get 
\begin{align*}
\sum_{(s,t)\in \mathcal P}|I_3(t,s)|^{\frac p2}_{\alpha-\frac 4p}\lesssim T(\sup_{s\in [0,T]}|Z_s|_{\alpha-\frac 2p})^{\frac{p}{2}}.
\end{align*}
Since $Z\in V^\infty_{\alpha-\frac{2}{p}}$, we have $I_3\in V^{\frac{p}{2}}_{\alpha-\frac{4}{p}}$. Finally, the boundedness of $S$
 gives 
 \begin{align*}
|I_4(t,s)|_{\alpha-\frac{4}{p}}&=|S_{t,s}Y_s': \mathbb{X}_{t,s}|_{\alpha-\frac{4}{p}}\\
&\lesssim |Y_s'|_{\alpha-\frac{1}{p}}|\mathbb{X}_{t,s}|.
\end{align*}
Since $Y'$ is bounded in $\cB_{\alpha-\frac{1}{p}}$ and $\mathbb{X}$ has finite $\frac{p}{2}$ variation, taking the supremum over partitions implies $I_4\in V^{\frac{p}{2}}_{\alpha-\frac{4}{p}}$. Combining $I_1, I_2, I_3, I_4$, we have $R^Z\in V^{\frac{p}{2}}_{\alpha-\frac{4}{p}}$
Repeating the above argument in $\cB_{\alpha-\frac{3}{p}}$  yields
$R^Z\in V^p_{\alpha-\frac{3}{p}}$. We have the estimate for the remainder 
\begin{equation}\label{est-RZ}
\begin{split}
\|R^Z\|_{V^{\frac p2, p}_{\alpha-\frac 4p, \alpha-\frac 3p}}&\lesssim  \|Y,Y'\|_{\mathcal V^{\frac p2}_{X, \alpha}}[(1+T^{\frac 1p}+T^{\frac 2p})\|X\|_{V^p([0,T];\mathbb R^d)}+(1+T^{\frac 1p})\|\mathbb X\|_{V^{\frac p2}([0,T];\mathbb R^{d\times d})}]\\
&~~+(T^{\frac 1p}
+T^{\frac 2p})\|Z\|_{V^\infty_{\alpha-\frac 2p}}\\
&\leq C  \|Y,Y'\|_{\mathcal V^{\frac p2}_{X, \alpha}} \|\mathbf X\|_{\mathbb V^p([0,T]; R^d)},
\end{split}
\end{equation}
where $C$ depends on $T, S, p$. 
Therefore, we proved $(Z, Z')\in \mathcal V^{\frac p2}_{X, \alpha-\frac{2}{p}}([0,T])$. The required estimate \eqref{bound-Z}  follows from \eqref{est-Z}, \eqref{est-Z'} and \eqref{est-RZ}. 
\end{proof}
\begin{lemma}\label{composition-f}
Let $p\in[2,3)$, $\alpha\in \mathbb{R}$, $\varepsilon>0$ and let $\mathbf X\in \mathbb{V}^p([0,T];\mathbb{R}^d)$.
Suppose $(Y, Y')\in \mathcal{V}^{\frac p2}_{X, \alpha}([0,T])$, $f\in C^2_b(\cB_{\alpha-\frac{i}{p}}, \mathcal{L}(\mathbb{R}^d, \cB_{\alpha+\frac{2-i}{p}+\varepsilon}))$, $i=0,1,2$.  Then 

(1) $(f(Y), f(Y)'):=(f(Y), Df(Y)Y')\in \mathcal{V}^{\frac p2}_{X, \alpha+\frac{2}{p}+\varepsilon}([0,T])$. 
\begin{align}
\|f(Y), Df(Y)Y'\|_{\mathcal{V}^{\frac p2}_{X, \alpha+\frac{2}{p}+\varepsilon}}\lesssim  \|f\|_{C_b^2}[1+\|Y, Y'\|_{\mathcal{V}^{\frac{p}{2}}_{X, \alpha}}+(1+\|X\|_{V^p([0,T];\mathbb{R}^d)})^2\|Y, Y'\|_{\mathcal{V}^{\frac{p}{2}}_{X, \alpha}}^2]. \label{comp-f5}
\end{align}

(2) $
(\int_0^{\cdot} S_{\cdot, r}f(Y_r)\cdot \textup{d}\mathbf X_r, f(Y))\in  \mathcal{V}^{\frac p2}_{X, \alpha}([0,T]).$

(3) Furthermore, if we assume that $f\in C_b^3(\cB_{\alpha-\frac{i}{p}}, \mathcal{L}(\mathbb{R}^d, \cB_{\alpha+\frac{2-i}{p}+\varepsilon}))$, $i=0,1,2$, and suppose $(f(\tilde Y), Df(\tilde Y)\tilde Y')$ for  another $(\tilde Y, \tilde Y')\in \mathcal{V}^{\frac p2}_{X, \alpha}([0,T])$, then 
\begin{equation}\label{diff-comp}
\begin{split}
&\|f(Y)-f(\tilde Y), Df(Y)Y'-Df(\tilde Y)\tilde Y'\|_{\mathcal V^{\frac p2}_{X, \alpha+\frac 2p+\varepsilon}}\\
\lesssim &~~ \|f\|_{C_b^3}(1+\|X\|_{V^p([0,T];\mathbb R^d)})^2(1+\|Y,  Y'\|_{\mathcal V^{\frac p2}_{X, \alpha}}+\|\tilde Y, \tilde Y'\|_{\mathcal V^{\frac p2}_{X, \alpha}})^2\|Y-\tilde Y, Y'-\tilde Y'\|_{\mathcal V^{\frac p2}_{X, \alpha}}.
\end{split}
\end{equation}
\end{lemma}
\begin{proof}
Denote $\|f\|_{C_b^a}:=\max\limits_{i=0,1,2}\|f\|_{C_b^a(\cB_{\alpha-\frac{i}{p}},\mathcal{L}(\mathbb{R}^d, \cB_{\alpha+\frac{2-i}{p}+\varepsilon}))}$, $a=2,3$. Since $f$ is a smooth function. It is not hard to check that $f(Y)$, $Df(Y)Y'$ are regulated paths in $\cB_{\alpha+\frac{2}{p}+\varepsilon}$, $\cB_{\alpha+\frac{1}{p}+\varepsilon}$, respectively.
Similar to Lemma \ref{fff2}, using Taylor's formula, we have 
\begin{equation}\label{comp-f1}
\begin{split}
\|f(Y)\|_{V^\infty_{\alpha+\frac{2}{p}+\varepsilon}}&=\sup\limits_{t\in[0,T]}
\|f(Y_t)-f(Y_0)\|_{\alpha+\frac{2}{p}+\varepsilon}\\
&\leq \|f\|_{C_b^2}
\sup_{t\in[0,T]}|Y_t-Y_0|_\alpha\\
&=
\|f\|_{C_b^2}\|Y\|_{V^\infty_\alpha}\\
&\leq \|f\|_{C_b^2}\|Y, Y'\|_{\mathcal{V}^{\frac{p}{2}}_{X, \alpha}}
\end{split}
\end{equation}
and 
\begin{equation}\label{comp-f2}
\begin{split}
\|Df(Y)Y'\|_{V^\infty_{\alpha+\frac{1}{p}+\varepsilon}}&=\sup_{t\in[0,T]}
\|Df(Y_t)Y'_t-Df(Y_0)Y'_0\|_{\alpha+\frac{1}{p}+\varepsilon}\\
&\leq
\|f\|_{C_b^2}
(
\|Y'\|_{V^\infty_{\alpha-\frac{1}{p}}}
+
|Y'_0|_{\alpha-\frac{1}{p}}
\|Y\|_{V^\infty_{\alpha-\frac{1}{p}}}
)\\
&\lesssim 
\|f\|_{C_b^2}
(
\|Y'\|_{V^\infty_{\alpha-\frac{1}{p}}}
+
|Y'_0|_{\alpha-\frac{1}{p}}
\|Y\|_{V^\infty_\alpha}
)\\
&\lesssim \|f\|_{C_b^2}( \|Y, Y'\|_{\mathcal{V}^{\frac{p}{2}}_{X, \alpha}}+\|Y, Y'\|_{\mathcal{V}^{\frac{p}{2}}_{X, \alpha}}^2)
.
\end{split}
\end{equation}
Since $Df\in C_b^1(\cB_{\alpha-\frac{i}{p}}, \mathcal{L}(\cB_{\alpha-\frac{i}{p}},\mathcal{L}(\mathbb{R}^d,\cB_{\alpha+\frac{2-i}{p}+\varepsilon})))$, $Y\in V^p_{\alpha-\frac 1p}$, $Y'\in V^{p,\infty}_{\alpha-\frac{2}{p},\alpha-\frac{1}{p}}$, we have the following inequality with $i=2$:
\begin{align*}
\|Df(Y_t)Y_t'-Df(Y_s)Y_s'\|_{\alpha+\varepsilon}&\leq \|Df(Y_t)Y_t'-Df(Y_t)Y_s'\|_{\alpha+\varepsilon}+\|Df(Y_t)Y_s'-Df(Y_s)Y_s'\|_{\alpha+\varepsilon}\\
&\leq \|Df(Y_t)\|_{\mathcal{L}(\cB_{\alpha-\frac{2}{p}}, \mathcal{L}(\mathbb{R}^d, \cB_{\alpha+\varepsilon}))}|Y_t'-Y_s'|_{\alpha-\frac{2}{p}}\\
&~~+\|Df(Y_t)-Df(Y_s)\|_{\mathcal{L}(\cB_{\alpha-\frac{2}{p}}, \mathcal{L}(\mathbb{R}^d, \cB_{\alpha+\varepsilon}))}|Y_s'|_{\alpha-\frac{2}{p}}\\
&\leq \|f\|_{C_b^2}(|Y_t'-Y_s'|_{\alpha-\frac 2p}+|Y_t-Y_s|_{\alpha-\frac 2p}|Y_s'|_{\alpha-\frac 2p}).
\end{align*}
Taking the $p$-variation  and using \eqref{est-rough-yp}, we have  
\begin{equation}\label{comp-f3}
\begin{split}
\|Df(Y)Y'\|_{V^p_{\alpha+\varepsilon}}&\lesssim \|f\|_{C_b^2}(\|Y'\|_{V^p_{\alpha-\frac 2p}}+\sup_{0\leq s\leq T}|Y_s'|_{\alpha-\frac 2p}\|Y\|_{V^p_{\alpha-\frac 1p}})\\
&\lesssim \|f\|_{C_b^2}[\|Y, Y'\|_{\mathcal{V}^{\frac{p}{2}}_{X, \alpha}}+(1+\|X\|_{V^p([0,T];\mathbb{R}^d)})\|Y, Y'\|_{\mathcal{V}^{\frac{p}{2}}_{X, \alpha}}^2].
\end{split}
\end{equation}
Recall that $R^{f(Y)}_{t,s}:=f(Y_t)-f(Y_s)-Df(Y_s)Y_s'\cdot X_{t,s}$. Taylor's formula and $Y_{t,s}=Y_s'\cdot X_{t,s}+R^{Y}_{t,s}$ give
\begin{align*}
R_{t,s}^{f(Y)}=Df(Y_s)R^Y_{t,s}+\int_0^1(1-r)D^2f(Y_s+rY_{t,s})(Y_{t,s}, Y_{t,s})\textup{d}r.
\end{align*}
For $i=1,2$, we have 
\begin{align*}
|R^{f(Y)}_{t,s}|_{\alpha+\frac{2-i}{p}+\varepsilon}\leq \|f\|_{C_b^2}(|R^Y_{t,s}|_{\alpha-\frac{i}{p}}+\frac{1}{2}|Y_{t,s}|^2_{\alpha-\frac{i}{p}}),
\end{align*}
which implies $R^{f(Y)}\in V^{\frac{p}
{2},p}_{\alpha+\varepsilon, \alpha+\frac{1}{p}+\varepsilon}$. Similarly, we have the following estimate
\begin{equation}\label{comp-f4}
\begin{split}
\|R^{f(Y)}\|_{V^{\frac{p}
{2},p}_{\alpha+\varepsilon, \alpha+\frac{1}{p}+\varepsilon}}\lesssim \|f\|_{C_b^2}[\|Y, Y'\|_{\mathcal{V}^{\frac{p}{2}}_{X, \alpha}}+(1+\|X\|_{V^p([0,T];\mathbb{R}^d)})^2\|Y, Y'\|_{\mathcal{V}^{\frac{p}{2}}_{X, \alpha}}^2].
\end{split}
\end{equation}
Combining \eqref{comp-f1}-\eqref{comp-f4}, we have  $(f(Y), f(Y)')\in \mathcal{V}^{\frac p2}_{X, \alpha+\frac{2}{p}+\varepsilon}$, 
\begin{align*}
\|f(Y), Df(Y)Y'\|_{\mathcal{V}^{\frac p2}_{X, \alpha+\frac{2}{p}+\varepsilon}}\lesssim  \|f\|_{C_b^2}[1+\|Y, Y'\|_{\mathcal{V}^{\frac{p}{2}}_{X, \alpha}}+(1+\|X\|_{V^p([0,T];\mathbb{R}^d)})^2\|Y, Y'\|_{\mathcal{V}^{\frac{p}{2}}_{X, \alpha}}^2].
\end{align*}

The second result follows from the arguments in  Lemma \ref{control-roughint}.   

For (3), we write 
\begin{align*}
Y_{t,s}-\tilde Y_{t,s}=(Y_s'-\tilde Y_s')X_{t,s}+(R^Y-R^{\tilde Y})_{t,s}.
\end{align*}
Using Minkowski inequality, we have 
\begin{equation}\label{diff0}
\begin{split}
\|Y-\tilde Y\|_{V^p_{\alpha-\frac1p}}
&\leq (\sum_{(s,t)\in \mathcal P}|(Y_s'-\tilde Y_s' )X_{t,s}|^p_{\alpha-\frac 1p})^{\frac 1p}+ (\sum_{(s,t)\in \mathcal P}|(R^Y-R^{\tilde Y})_{t,s}|^p_{\alpha-\frac 1p})^{\frac 1p}\\
&\leq \sup_{r\in [0,T]}|Y_s'-\tilde Y_s'|_{\alpha-\frac 1p}(\sum_{(s,t)\in \mathcal P}|X_{t,s}|^p)^{\frac 1p}+\|(R^Y-R^{\tilde Y})\|_{V^{p}_{\alpha-\frac 1p}([0,T])}\\
&\leq (|Y_0'-\tilde Y_0'|_{\alpha-\frac 1p}+\|Y'-\tilde Y'\|_{V^{\infty}_{\alpha-\frac 1p}})\|X\|_{V^{p}([0,T];\mathbb R^d)}+\|(R^Y-R^{\tilde Y})\|_{V^{p}_{\alpha-\frac 1p}([0,T])}\\
&\leq (1+\|X\|_{V^p([0,T];\mathbb R^d)})
\|Y-\tilde Y,Y'-\tilde Y'\|_
{\mathcal V^{\frac p2}_{X,\alpha}}.
\end{split}
\end{equation}
According to the Lipschitz property of $f$, we have 
\begin{equation}\label{diff1}
\begin{split}
&|f(Y_0)-f(\tilde Y_0)|_{\alpha+\frac 2p+\varepsilon}+\|f(Y)-f(\tilde Y)\|_{V^\infty_{\alpha+\frac 2p+\varepsilon}}\\\lesssim &~~\|f\|_{C_b^3}\sup_{0\leq t\leq T}(|Y_0-\tilde Y_0|_\alpha+\|Y-\tilde Y\|_{V^\infty_\alpha})\\
\lesssim &~~ \|f\|_{C_b^3} \|Y-\tilde Y, Y'-\tilde Y'\|_{\mathcal V^{\frac p2}_{X, \alpha}}.
\end{split}
\end{equation}
Using the boundedness and Lipschitz property of $Df$, we have 
\begin{align*}
|Df(Y_t)Y'_t-Df(\tilde Y_t)\tilde Y'_t|
_{\alpha+\frac1p+\varepsilon}&\leq |Df(Y_t)(Y_t'-\tilde Y_t')+(Df(Y_t)-Df(\tilde Y_t))\tilde Y_t'|_{\alpha+\frac 1p+\varepsilon}\\
&\leq
\|f\|_{C_b^3}
(
|Y'_t-\tilde Y'_t|_{\alpha-\frac1p}
+
|Y_t-\tilde Y_t|_{\alpha-\frac1p}
|\tilde Y'_t|_{\alpha-\frac1p}
).
\end{align*}
and 
\begin{align*}
|Df(Y_0)Y'_0-Df(\tilde Y_0)\tilde Y'_0|
_{\alpha+\frac1p+\varepsilon}
\leq
\|f\|_{C_b^3}
(
|Y'_0-\tilde Y'_0|_{\alpha-\frac1p}
+
|Y_0-\tilde Y_0|_{\alpha-\frac1p}
|\tilde Y'_0|_{\alpha-\frac1p}
),
\end{align*}
which yields
\begin{equation}\label{diff2}
\begin{split}
&|Df(Y_0)Y'_0-Df(\tilde Y_0)\tilde Y'_0|
_{\alpha+\frac1p+\varepsilon}
+
\|Df(Y)Y'-Df(\tilde Y)\tilde Y'\|
_{V^\infty_{\alpha+\frac1p+\varepsilon}}\\
\lesssim&~~
\|f\|_{C_b^3}
[
|Y'_0-\tilde Y'_0|_{\alpha-\frac1p}
+
\|Y'-\tilde Y'\|_{V^\infty_{\alpha-\frac1p}}
+
(
|Y_0-\tilde Y_0|_\alpha
+\|Y-\tilde Y\|_{V^\infty_\alpha}
)(
|\tilde Y'_0|_{\alpha-\frac1p}
+\|\tilde Y'\|_{V^\infty_{\alpha-\frac1p}}
)
]
\\
\lesssim&~~
\|f\|_{C_b^3}
(
1+\|\tilde Y,\tilde Y'\|_
{\mathcal V^{\frac p2}_{X,\alpha}}
)
\|Y-\tilde Y,Y'-\tilde Y'\|_
{\mathcal V^{\frac p2}_{X,\alpha}}.
\end{split}
\end{equation}
For  $0\leq s<t\leq T$, 
\begin{align*}
&Df(Y_t)Y_t'-Df(\tilde Y_t)\tilde Y_t'-Df(Y_s)Y_s'+Df(\tilde Y_s)\tilde Y_s'\\
=&~~Df(Y_t)(Y_{t,s}'-\tilde Y_{t,s}')+[Df(Y_t)-Df(Y_s)](Y_s'-\tilde Y_s')
+[Df(Y_t)-Df(\tilde Y_t)]\tilde Y_{t,s}'\\
&~~+[Df(Y_t)-Df(\tilde Y_t)-Df(Y_s)+Df(\tilde Y_s)]\tilde Y_s', 
\end{align*}
where we have 
\begin{align*}
&Df(Y_t)-Df(\tilde Y_t)-Df(Y_s)+Df(\tilde Y_s)
=
\int_0^1
D^2f\bigl(\tilde Y_t+r(Y_t-\tilde Y_t)\bigr)([Y_{t,s}-\tilde Y_{t,s})\textup dr\\
&+\int_0^1
[
D^2f\bigl(\tilde Y_t+r(Y_t-\tilde Y_t)\bigr)
-D^2f\bigl(\tilde Y_s+r(Y_s-\tilde Y_s)\bigr)
]
(Y_s-\tilde Y_s)\textup dr.
\end{align*}
Following from the boundedness and the Lipschitz property of $D^2f$, we have 
\begin{align*}
&\|Df(Y_t)-Df(\tilde Y_t)-Df(Y_s)+Df(\tilde Y_s)\|_{\mathcal L(\cB_{\alpha-\frac 2p}, \mathcal L(\mathbb R^d, \cB_{\alpha+\varepsilon}))}\\
\lesssim &~~\|f\|_{C_b^3}[|Y_{t,s}-\tilde Y_{t,s}|_{\alpha-\frac 2p}+|Y_s-\tilde Y_s|_{\alpha-\frac 2p}(|Y_{t,s}|_{\alpha-\frac 2p}+|\tilde Y_{t,s}|_{\alpha-\frac 2p}).
\end{align*}
Then we have 
\begin{equation}\label{diff3}
\begin{split}
&\|Df(Y)Y'-Df(\tilde Y)\tilde Y'\|_{V^p_{\alpha+\varepsilon}}\\
\lesssim &~~\|f\|_{C_b^3}(1+\|X\|_{V^p([0,T];\mathbb R^d)})
(
1+\|Y,Y'\|_{\mathcal V^{\frac p2}_{X,\alpha}}
+\|\tilde Y,\tilde Y'\|_{\mathcal V^{\frac p2}_{X,\alpha}}
)^2
\|Y-\tilde Y,Y'-\tilde Y'\|_{\mathcal V^{\frac p2}_{X,\alpha}}.
\end{split}
\end{equation}
Taylor's formula gives 
\begin{align*}
R^{f(Y)}_{t,s}-R_{t,s}^{f(\tilde Y)}&=Df(Y_s)(R^{Y}-R^{\tilde Y})_{t,s}+[Df(Y_s)-Df(\tilde Y_s)]R^{\tilde Y}_{t,s}\\
&~~+\int_0^1 (1-r)(D^2 f(Y_s+rY_{t,s})(Y_{t,s}, Y_{t,s})-D^2f(\tilde Y_s+r\tilde Y_{t,s})(\tilde Y_{t,s}, \tilde Y_{t,s}))\textup{d}r.
\end{align*}
For $i=1,2$, 
\begin{align*}
|R^{f(Y)}_{t,s}-R^{f(\tilde Y)}_{t,s}|_{\alpha+\frac{2-i}{p}+\varepsilon}&\lesssim \|f\|_{C_b^3}[|(R^Y-R^{\tilde Y})_{t,s}|_{\alpha-\frac ip}
+|Y_s-\tilde Y_s|_{\alpha-\frac ip}|R^{\tilde Y}_{t,s}|_{\alpha-\frac ip}\\
&~~+|(Y-\tilde Y)_{t,s}|_{\alpha-\frac ip}(|Y_{t,s}|_{\alpha-\frac ip}+|\tilde Y_{t,s}|_{\alpha-\frac ip})+\sup_{s\leq r\leq t}|Y_r-\tilde Y_r|_{\alpha-\frac ip}|Y_{t,s}|_{\alpha-\frac ip}^2
]
\end{align*}
By Hölder inequality, we have the estimate 
\begin{equation}\label{diff4}
\begin{split}
\|R^{f(Y)}-R^{f(\tilde Y)}\|_{V^{\frac p2, p}_{\alpha+\varepsilon, \alpha+\frac 1p+\varepsilon}}&\lesssim \|f\|_{C_b^3}[\|R^Y-R^{\tilde Y}\|_
{V^{\frac p2,p}_{\alpha-\frac2p,\alpha-\frac1p}}+
\sup_{0\leq r\leq T}|Y_r-\tilde Y_r|_\alpha
\|R^{\tilde Y}\|_
{V^{\frac p2,p}_{\alpha-\frac2p,\alpha-\frac1p}}
\\
&~~+
\|Y-\tilde Y\|_{V^p_{\alpha-\frac1p}}
(
\|Y\|_{V^p_{\alpha-\frac1p}}
+\|\tilde Y\|_{V^p_{\alpha-\frac1p}}
)+
\sup_{0\leq r\leq T}|Y_r-\tilde Y_r|_\alpha
\|Y\|_{V^p_{\alpha-\frac1p}}^2
]
\end{split}
\end{equation}
Combining with \eqref{diff0}-\eqref{diff4}, we have
\begin{align*}
&\|f(Y)-f(\tilde Y), Df(Y)Y'-Df(\tilde Y)\tilde Y'\|_{\mathcal V^{\frac p2}_{X, \alpha+\frac 2p+\varepsilon}}\\
\lesssim &~~ \|f\|_{C_b^3}(1+\|X\|_{V^p([0,T];\mathbb R^d)})^2(1+\|Y,  Y'\|_{\mathcal V^{\frac p2}_{X, \alpha}}+\|\tilde Y, \tilde Y'\|_{\mathcal V^{\frac p2}_{X, \alpha}})^2\|Y-\tilde Y, Y'-\tilde Y'\|_{\mathcal V^{\frac p2}_{X, \alpha}}.
\end{align*}
\end{proof}
\begin{lemma}
\label{local-rough-pde}
Let $p\in[2,3)$, $\alpha\in \mathbb{R}$, $\varepsilon>0$ and let $\mathbf X\in \mathbb{V}^p([0,T];\mathbb{R}^d)$ be right-continuous at $t=0$. Suppose that  $f\in C_b^3(\cB_{\alpha-\frac{i}{p}}, \mathcal{L}(\mathbb{R}^d, \cB_{\alpha+\frac{2-i}{p}+\varepsilon}))$, $i=0,1,2$.  Then for every $y_0\in \cB_\alpha$,  there exists a time horizon $\rho\in (0,T)$ and a unique local solution $(Y, Y')\in \mathcal{V}^{\frac p2}_{X, \alpha}([0,\rho])$ such that $Y'=f(Y)$ and 
\begin{align*}
Y_t=S_{t,0}y_0+\int_0^{t}S_{t,r}f(Y_r)\cdot \textup{d}\mathbf{X}_r, \quad t\in[0,\rho].
\end{align*}
\end{lemma}
\begin{proof}
Denote $\|f\|_{C_b^3}:=\max\limits_{i=0,1,2}\|f\|_{C_b^3(\cB_{\alpha-\frac{i}{p}},\mathcal{L}(\mathbb{R}^d, \cB_{\alpha+\frac{2-i}{p}+\varepsilon}))}$. Let $0<\rho<1$. 
Define the map
\begin{align*}
\mathcal{J}
(Y, Y')_t:=(Z_t, Z_t'):=(S_{t,0}y_0+\int_0^tS_{t,r}f(Y_r)\cdot \textup{d}\mathbf X_r, f(Y_t)), \quad t\in[0,\rho]
\end{align*}
and 
\begin{align*}
\xi_t:=S_{t,0}y_0+\int_0^tS_{t,r}f(y_0)\cdot \textup{d}\mathbf{X}_r,\quad \xi_t':=f(y_0).
\end{align*}
Using the smoothing property of  $S$, we have $(S_{\cdot,0}y_0,0)\in \mathcal{V}^{\frac p2}_{X, \alpha}([0,\rho])$. It is clear that  $(f(y_0), 0)\in \mathcal{V}^{\frac p2}_{X, \alpha+\frac{2}{p}+\varepsilon}([0,\rho])$. Thus,  $(\xi, \xi')=(S_{\cdot, 0}y_0,0)+(\int_0^{\cdot}S_{\cdot, r}f(y_0)\textup{d}\mathbf X_r,f(y_0))\in \mathcal{V}^{\frac{p}{2}}_{X, \alpha}([0,\rho])$. 
Similar to \eqref{est-RZ}, we derive
\begin{align}\label{Rxi}
\|R^\xi\|_{V^{\frac p2, p}_{\alpha-\frac 2p, \alpha-\frac 1p}([0,\rho])}
\leq  C_{S,p} \rho^{\frac 1p}(|y_0|_\alpha+\|f\|_{C_b^3}\|\mathbf X\|_{\mathbb V^{p}([0,\rho];\mathbb R^d)}).
\end{align} 
And we have 
\begin{align}\label{xixi'}
\|\xi ,\xi'\|_{\mathcal V^{\frac p2}_{X, \alpha}([0,\rho])}\leq C_{S, p}[|y_0|_\alpha+\|f\|_{C_b^3}(1+\|\mathbf X\|_{\mathbb V^p([0,\rho];\mathbb R^d)})]. 
\end{align}
From \eqref{smooth} and Lemma  \ref{int-rough}, we have 
\begin{equation}\label{xi-y_0}
\begin{split}
\sup_{0\leq t\leq \rho}|\xi_t-y_0|_{\alpha-\frac1p}&=\sup_{0\leq t\leq \rho}|(S_{t,0}-\textup{id})y_0+\int_0^{t}S_{t,r}f(y_0)\cdot \textup{d}\mathbf X_r|_{\alpha-\frac 1p}\\
&\leq C_{S, p}\big[\rho^{\frac 1p}|y_0|_\alpha+\|f\|_{C_b^3}[(1+\rho ^{\frac 2p})\|X\|_{V^p([0,\rho];\mathbb R^d)}+\rho^{\frac 1p}\|\mathbb X\|_{V^{\frac p2}([0,\rho];\mathbb R^{d\times d})}]\big]\\
&\leq C_{S, p}(\rho^{\frac 1p }|y_0|_\alpha+\|f\|_{C_b^3}\|\mathbf X\|_{\mathbb V^p([0,\rho];\mathbb R^d)}). 
\end{split}
\end{equation}

 Since $\mathbf X_{t,0}$ is right continuous at $t=0$, we have $\|\mathbf X\|_{\mathbb{V}^p([0,\rho];\mathbb{R}^d)}<1$ by choosing sufficiently small $\rho$. 
 Consider a closed set
\begin{align*}
\mathscr{B}_\rho(y_0):=\{&(Y, Y')\in \mathcal{V}^{\frac{p}{2}}_{X, \alpha}([0,\rho]) \mid~  (Y_0, Y_0')=(y_0, f(y_0)),\quad  \|Y-\xi\|_{V^\infty_\alpha([0,\rho])}\leq 1, \\
&\|Y'-\xi'\|_{V^{p,\infty}_{\alpha-\frac 2p, \alpha-\frac 1p}([0,\rho])}\leq 1, \quad 
\|R^Y-R^\xi\|_{V^{\frac p2, p}_{\alpha-\frac{2}{p}, \alpha-\frac{1}{p}}([0,\rho])}\leq \frac{1}{4(1+C_{S,p}\|f\|_{C_b^3})}\}.
\end{align*}
Observe that 
\begin{align*}
Z_t-\xi_t=\int_0^t S_{t,r}[f(Y_r)-f(y_0)]\cdot \textup{d}\mathbf{X}_r, \quad Z_t'-\xi_t'=f(Y_t)-f(y_0). 
\end{align*}
Since $Y_0=\xi_0=y_0$, $Y_0'=\xi_0'=f(y_0)$, we have 
$
Y_t-\xi_t=(R^Y-R^\xi)_{t,0}.
$
Using triangle inequality and \eqref{xi-y_0}, we have 
\begin{equation}\label{Y-y}
\begin{split}
\sup_{0\leq t\leq \rho}|Y_t-y_0|_{\alpha-\frac{1}{p}}&\leq  \sup_{0\leq t\leq \rho}|Y_t-\xi_t|_{\alpha-\frac 1p}+\sup_{0\leq t\leq \rho}|\xi_t-y_0|_{\alpha-\frac 1p}\\
&\leq \|R^Y-R^\xi\|_{V^{\frac p2, p}_{\alpha-\frac 2p, \alpha-\frac 1p}([0,\rho])}+C_{S, p}(\rho^{\frac 1p }|y_0|_\alpha+\|f\|_{C_b^3}\|\mathbf X\|_{\mathbb V^p([0,\rho];\mathbb R^d)}).
\end{split}
\end{equation}
First, we prove $\mathcal J$ maps $\mathscr B_\rho$ into itself.

According to Lemma \ref{composition-f},  we have $(f(Y)-f(y_0), Df(Y)Y')\in \mathcal{V}^{\frac p2}_{X, \alpha+\frac{2}{p}+\varepsilon}([0,\rho])$ and 
\begin{equation}\label{fY-}
\begin{split}
\|f(Y)-f(y_0), Df(Y)Y'\|_{\mathcal{V}^{\frac{p}{2}}_{X,\alpha+\frac{2}{p}}([0,\rho])}&\leq C_{S,p} \|f\|_{C_b^3}[1+\|Y-y_0, Y'\|_{\mathcal V^{\frac p2}_{X, \alpha}([0,\rho])}\\
&~~+(1+\|X\|_{V^p([0,\rho];\mathbb R^d)})^2\|Y-y_0, Y'\|_{\mathcal V^\frac{p}{2}_{X,\alpha}([0,\rho])}^2],
\end{split}
\end{equation}
which allows us to apply Lemma \ref{control-roughint} to obtain $(Z-\xi, Z'-\xi')\in \mathcal V^{\frac{p}{2}}_{X, \alpha}([0,\rho])$. 

In view of the definition of $\mathscr{B}_\rho$ and  \eqref{xixi'}, we have 
\begin{align*}
\|Y-y_0, Y'\|_{\mathcal V^{\frac p2}_{X, \alpha}([0,\rho])} &\leq \|Y-\xi, Y'-\xi'\|_{\mathcal V^{\frac p2}_{X, \alpha}([0,\rho])} +\|\xi-y_0, \xi'\|_{\mathcal V^{\frac p2}_{X, \alpha}([0,\rho])}\\
&\leq 2+\frac{1}{4(1+C_{S, p}\|f\|_{C_b^3})}+C_{S, p}[|y_0|_\alpha+\|f\|_{C_b^3}(1+\|\mathbf X\|_{\mathbb V^p([0,\rho];\mathbb R^d)})]\\
&\leq M,
\end{align*}
where $M$ depends on $S, p, \|f\|_{C_b^3}$ and $y_0$. 
In view of \eqref{rough-bound} and \eqref{fY-},  we have 
\begin{equation}\label{int0tX}
\begin{split}
&\|Z-\xi\|_{V_\alpha^\infty([0,\rho])}=\sup_{0\leq t\leq \rho}|Z_t-\xi_t|_\alpha\\
&=\sup_{0\leq t\leq \rho}|\int_0^tS_{t,r}[f(Y_r)-f(y_0)]\cdot \textup{d}\mathbf X_r|_\alpha\\
&\leq C_{S, p}[\sup_{0\leq t\leq \rho} \{\|X\|_{V^p([0,t];\mathbb{R}^d)}\|R^{f(Y)}\|_{V^{\frac p2}([0,t);\cB_{\alpha})}+\|\mathbb{X}\|_{V^{\frac{p}{2}}([0,t];\mathbb{R}^{d\times d})}\|Df(Y)Y'\|_{V^{p}([0,t);\cB_{\alpha})}\\
&~~+\|f(Y)-f(y_0), Df(Y)Y'\|_{\mathcal{V}^{\frac{p}{2}}_{X,\alpha+\frac{2}{p}}([0,t])}(\|X\|_{V^{p}([0,t];\mathbb{R}^d)}t^{\frac{2}{p}}+\|\mathbb X\|_{V^{\frac{p}{2}}([0,t];\mathbb{R}^{d\times d})}t^{\frac{1}{p}})\\
&~~+|S_{t,0}Df(y_0)f(y_0):\mathbb{X}_{t,0}|_\alpha\}]\\
&\leq C_{S,p, f, y_0} \|\mathbf X\|_{\mathbb V^p([0,\rho];\mathbb{R}^d)}.
\end{split}
\end{equation}
In view of \eqref{Y-y} and $\cB_{\alpha+\frac 1p}\hookrightarrow \cB_{\alpha-\frac 1p}$, 
 \begin{align*}
\|Z'-\xi '\|_{V^\infty_{\alpha-\frac 1p}([0,\rho])}&:=\sup_{0\leq t\leq \rho}|f(Y_t)-f(y_0)|_{\alpha-\frac{1}{p}}
\\
&\lesssim  \sup_{0\leq t\leq \rho}|f(Y_t)-f(y_0)|_{\alpha+\frac{1}{p}}\\
&\lesssim  \|f\|_{C_b^3}\sup_{0\leq t\leq \rho}|Y_t-y_0|_{\alpha-\frac 1p}\\
&\leq C_{S,p}\|f\|_{C_b^3}(\|R^Y-R^\xi\|_{V^{\frac p2, p}_{\alpha-\frac 2p, \alpha-\frac 1p}([0,\rho])}+\rho^{\frac1p}|y_0|_\alpha+\|f\|_{C_b^3}\|\mathbf X\|_{\mathbb V^p([0,\rho];\mathbb{R}^d)}). 
\end{align*}
From  \eqref{est-rough-yp} and \eqref{Rxi}, 
\begin{equation}\label{Y_ts}
\begin{split}
\|Y\|_{V^p_{\alpha-\frac1p}([0,\rho])}&\leq 
\|R^Y\|_{V^p_{\alpha-\frac 1p}([0,\rho])}+(|Y_0'|_{\alpha-\frac{1}{p}}+\|Y'\|_{V^{\infty}_{\alpha-\frac 1p}([0,\rho])})\|X\|_{V^p([0,\rho];\mathbb{R}^d)}\\
&\leq \|R^Y-R^\xi\|_{V^p_{\alpha-\frac 1p}([0,\rho])}+\|R^\xi\|_{V^p_{\alpha-\frac 1p}}+(1+|f(y_0)|_{\alpha-\frac 1p})\|X\|_{V^p([0,\rho];\mathbb{R}^d)}\\
&\leq  \|R^Y-R^\xi\|_{V^p_{\alpha-\frac 1p}([0,\rho])}+C_{S,p}(\rho^{\frac 1p}|y_0|_\alpha+\|f\|_{C_b^3}\|\mathbf X\|_{\mathbb V^p([0,\rho];\mathbb R^d)})\\
&~~+(1+\|f\|_{C_b^3})\|X\|_{V^p([0,\rho];\mathbb R^d)}\\
&\leq   \|R^Y-R^\xi\|_{V^p_{\alpha-\frac 1p}([0,\rho])}+C_{S,p}\rho^{\frac 1p}|y_0|_\alpha+C_{S, p}(1+\|f\|_{C_b^3})\|\mathbf X\|_{\mathbb V^p([0,\rho];\mathbb R^d)}. 
\end{split}
\end{equation}
We therefore impose  $\|R^Y-R^\xi\|_{V^{\frac p2, p}_{\alpha-\frac{2}{p}, \alpha-\frac{1}{p}}}\leq \frac{1}{4(1+C_{S, p}\|f\|_{C_b^3})}$ in the definition of $\mathscr B_\rho$, in order to obtain the  smallness condition required for self-mapping.

Using \eqref{Rxi}, \eqref{Y_ts}, we have 
\begin{align*}
\|Z'-\xi'\|_{V^p_{\alpha-\frac 2p}([0,\rho])}&=
(\sup\limits_{\mathcal P}\sum_{(s,t)\in \mathcal P}|Z_t'-\xi_t'-(Z_s'-\xi_s')|^p_{\alpha-\frac{2}{p}})^{\frac{1}{p}}\\
&=(\sup\limits_{\mathcal P}\sum_{(s,t)\in \mathcal P}|f(Y_t)-f(Y_s)|^p_{\alpha-\frac{2}{p}})^{\frac 1p}\\
&\leq  C_{S, p} \|f\|_{C_b^3}\|Y\|_{V^p_{\alpha-\frac 1p}([0,\rho])}\\
&\leq   C_{S, p}\|f\|_{C_b^3} 
[  \|R^Y-R^\xi\|_{V^p_{\alpha-\frac 1p}([0,\rho])}+\rho^{\frac 1p}|y_0|_\alpha+(1+\|f\|_{C_b^3})\|\mathbf X\|_{\mathbb V^p([0,\rho];\mathbb R^d)}]. 
\end{align*}
Thus, we have 
\begin{align*}
\|Z'-\xi'\|_{V^{p,\infty}_{\alpha-\frac 2p, \alpha-\frac 1p}([0,\rho])}
&\leq \|Z'-\xi'\|_{V_{\alpha-\frac 1p}^\infty([0,\rho])}+\|Z'-\xi'\|_{V^p_{\alpha-\frac 2p}([0,\rho])}\\
&\leq C_{S, p}\|f\|_{C_b^3}(\|R^Y-R^\xi\|_{V^{\frac p2, p}_{\alpha-\frac 2p, \alpha-\frac 1p}([0,\rho])}+\rho^\frac 1p |y_0|_\alpha+(1+\|f\|_{C_b^3})\|\mathbf X\|_{\mathbb V^p([0,\rho];\mathbb R^d)})\\
&\leq \frac 14+  C_{S, p}\|f\|_{C_b^3}[\rho^\frac 1p |y_0|_\alpha+(1+\|f\|_{C_b^3})\|\mathbf X\|_{\mathbb V^p([0,\rho];\mathbb R^d)}].
\end{align*}
According to the definition of the remainder, we have 
\begin{align*}
(R^Z-R^\xi)_{t,s}&=(Z-\xi)_{t,s}-(f(Y_s)-f(y_0))X_{t,s}\\
&=\int_s^t S_{t,r}(f(Y_r)-f(y_0))\cdot \textup{d}\mathbf X_r-S_{t,s}[(f(Y_s)-f(y_0))\cdot X_{t,s}+Df(Y_s)Y_s':\mathbb X_{t,s}]
\\
&~~+(S_{t,s}-\textup{id})(f(Y_s)-f(y_0))X_{t,s}\\
&~~+(S_{t,s}-\textup{id})(Z_s-\xi_s)\\
&~~+S_{t,s}Df(Y_s)Y_s':\mathbb X_{t,s}\\
&:=I_1(t,s)+I_2(t,s)+I_3(t,s)+I_4(t,s).
\end{align*}
Similar to \eqref{int0tX},  we deduce that $\|I_1\|_{V^{\frac p2}([0,\rho]; \cB_\alpha)}\leq C_{S, p, f, y_0}\|\mathbf X\|_{\mathbb V^p([0,\rho];\mathbb{R}^d)}.$ Since $\cB_{\alpha}\hookrightarrow \cB_{\alpha-\frac{i}{p}}$, $i=1,2$, and $V^{\frac p2}\subset V^p$, we claim that $$\|I_1\|_{V^{\frac p2, p}_{\alpha-\frac 2p, \alpha-\frac 1p}([0,\rho])}\leq C_{S, p, f, y_0}\|\mathbf X\|_{\mathbb V^p([0,\rho];\mathbb{R}^d)}.$$
As for terms $I_2, I_3, I_4$, using the smoothing properties of $S$ \eqref{smooth}, we have 
\begin{align*}
|I_2(t,s)|_{\alpha-\frac ip}&\leq C_S \|f\|_{C_b^3}|t-s|^{\frac ip}|X_{t,s}|,\\
|I_3(t,s)|_{\alpha-\frac ip}&\leq C_S |t-s|^{\frac ip}\|Z-\xi\|_{V^\infty_\alpha([0,s])}\\
&\leq C_{S, p, f, y_0} \|\mathbf X\|_{\mathbb V^p([0,\rho];\mathbb R^d)} |t-s|^{\frac ip},\\
|I_4(t,s)|_{\alpha-\frac ip}&\leq C_S \|f\|_{C_b^3}(1+\|f\|_{C_b^3})|\mathbb X_{t,s}|, 
\end{align*}
which implies 
\begin{align*}
\|I_2\|_{V^{\frac p2, p}_{\alpha-\frac 2p, \alpha- \frac 1p}([0, \rho])}&\leq C_{S,p} \rho^\frac 1p \|f\|_{C_b^3}\|X\|_{V^p([0,\rho];\mathbb R^d)},\\
\|I_3\|_{V^{\frac p2, p}_{\alpha-\frac 2p, \alpha- \frac 1p}([0, \rho])}&\leq C_{S, p, f, y_0} \|\mathbf X\|_{\mathbb V^p([0,\rho];\mathbb R^d)} \rho^{\frac 1p}, \\
\|I_4\|_{V^{\frac p2, p}_{\alpha-\frac 2p, \alpha- \frac 1p}([0, \rho])}&\leq C_{S, p}\|f\|_{C_b^3}(1+\|f\|_{C_b^3})\|\mathbb X\|_{V^{\frac p2}([0,\rho];\mathbb R^{d\times d})}. 
\end{align*}
Therefore, we have 
\begin{align*}
\|R^Z-R^\xi\|_{V^{\frac p2, p}_{\alpha-\frac 2p, \alpha-\frac 1p}([0,\rho])}
\leq  C_{S, p, f, y_0}\|\mathbf X\|_{\mathbb V^p([0,\rho];\mathbb{R}^d)}.
\end{align*}
Choosing sufficiently small $\rho$ such that  $$C_{S,p, f, y_0} \|\mathbf X\|_{\mathbb V^p([0,\rho];\mathbb{R}^d)}\leq \frac{1}{4(1+C_{S, p}\|f\|_{C_b^3})}, \quad C_{S,p}\|f\|_{C_b^3}(\rho^\frac 1p |y_0|_\alpha+(1+\|f\|_{C_b^3})\|\mathbf X\|_{V^p([0,\rho];\mathbb R^d)})<\frac 34, $$ which yields  
$\|Z-\xi\|_{V^\infty_\alpha([0,\rho])}\leq 1$, $\|Z'-\xi'\|_{V^{p,\infty}_{\alpha-\frac 2p, \alpha-\frac 1p}([0,\rho])} \leq 1$,  $\|R^Z-R^\xi\|_{V^{\frac p2, p}_{\alpha-\frac 2p, \alpha-\frac 1p}([0,\rho])}\leq \frac{1}{4(1+C_{S, p}\|f\|_{C_b^3})}$. 
Therefore,  $\mathcal J$ is a self-mapping.

Next, we will prove $\mathcal J$ is a contraction mapping. For $\lambda> 1$, we introduce an equivalent norm for the controlled rough path space
\begin{align*}
\|Y, Y'\|^{(\lambda)}_{\mathcal{V}^{\frac p2}_{X, \alpha}}:=|Y_0|_\alpha+\|Y\|_{V^\infty_{\alpha}}+|Y_0'|_{\alpha-\frac{1}{p}}+\|Y'\|_{V^{p,\infty}_{\alpha-\frac{2}{p}, \alpha-\frac{1}{p}}}+\lambda \|R^Y\|_{V^{\frac{p}{2}, p}_{\alpha-\frac{2}{p}, \alpha-\frac{1}{p}}}.
\end{align*}
For any $(Y, Y'), (\tilde Y, \tilde Y')\in \mathscr B_\rho$,  note that 
\begin{align*}
Z_t-\tilde Z_t
=\int_0^t S_{t,r}
\bigl(f(Y_r)-f(\tilde Y_r)\bigr)
\cdot\textup d\mathbf X_r, \quad Z_t'-\tilde{Z}_t'=f(Y_t)-f(\tilde{Y}_t). 
\end{align*}
  The composition difference estimate \eqref{diff-comp}, together with the uniform bounds on $\mathscr B_\rho$, gives
$$
\|f(Y)-f(\tilde Y),Df(Y)Y'-Df(\tilde Y)\tilde Y'\|_
{\mathcal V^{\frac p2}_{X,\alpha+\frac2p+\varepsilon}([0,\rho])}
\leq C_{S,p,f,y_0}
\|Y-\tilde Y,Y'-\tilde Y'\|_{\mathcal V^{\frac p2}_{X,\alpha}([0,\rho])}.
$$
  Applying \eqref{est-Z} and \eqref{est-RZ} to $(f(Y)-f(\tilde Y), Df(Y)Y'-Df(\tilde Y)\tilde Y')$ with spatial index $\alpha+\frac 2p$,  we obtain 
  \begin{align}\label{ZR}
\|Z-\tilde Z\|_{V^\infty_\alpha([0,\rho])}
+\|R^Z-R^{\tilde Z}\|_{V^{\frac p2,p}_{\alpha-\frac2p,\alpha-\frac1p}([0,\rho])}\leq C_{S,p,f,y_0}
\|\mathbf X\|_{\mathbb V^p([0,\rho];\mathbb R^d)}
\|Y-\tilde Y,Y'-\tilde Y'\|_{\mathcal V^{\frac p2}_{X,\alpha}([0,\rho])}.
\end{align}
For any $0\leq s<t\leq \rho$, we have 
\begin{align*}
Y_{t,s}-\tilde{Y}_{t,s}=R^Y_{t,s}-R^{\tilde{Y}}_{t,s}+(Y_s'-\tilde{Y}_s')X_{t,s}, 
\end{align*}
which yields
\begin{align*}
\|Y- \tilde Y\|_{V^{p}_{\alpha-\frac 1p}([0,\rho])}\leq (\|R^Y-R^{\tilde Y}\|_{V^p_{\alpha-\frac 1p}([0,\rho])}+\|Y'-\tilde Y'\|_{V^\infty_{\alpha-\frac{1}{p}}([0,\rho])}\|X\|_{V^{p}([0,\rho];\mathbb{R}^d)}).
\end{align*}
Since $Y_0=\tilde Y_0$, $Y_0'=\tilde Y_0'$,  we have $Y_t-\tilde Y_t=(R^Y-R^{\tilde Y})_{t,0}$. Thus,  
\begin{align*}
\|Y-\tilde Y\|_{V^\infty_{\alpha-\frac 1p}}\leq \|R^Y-R^{\tilde Y}\|_{V^p_{\alpha-\frac 1p}}. 
\end{align*}
Using Taylor's formula, we have 
\begin{align*}
&|f(Y_t)-f(\tilde Y_t)-f(Y_s)+f(\tilde Y_s)|_{\alpha-\frac 2p}\\
\lesssim &~~\|f\|_{C_b^3}[|Y_{t,s}-\tilde Y_{t,s}|_{\alpha-\frac 1p}+|Y_s-\tilde Y_s|_{\alpha-\frac 1p}(|Y_{t,s}|_{\alpha-\frac 1p}+|\tilde Y_{t,s}|_{\alpha-\frac 1p})],
\end{align*}
which implies 
\begin{align*}
\|Z'-\tilde{Z}'\|_{V^p_{\alpha-\frac 2p}([0,\rho])}
&\lesssim \|f\|_{C_b^3}[\|Y-\tilde Y\|_{V^p_{\alpha-\frac 1p}}+\sup_{0\leq t\leq \rho}|Y_t-\tilde Y_t|_{\alpha-\frac 1p}(\|Y\|_{V^p_{\alpha-\frac 1p}}+\|\tilde Y\|_{V^p_{\alpha-\frac 1p}})]\\
&\lesssim \|f\|_{C_b^3}[\|X\|_{V^p([0,\rho];\mathbb R^d)}\sup_{0\leq t\leq \rho}|Y_t'-\tilde Y_t'|_{\alpha-\frac 1p}\\
&~~+(1+\|Y\|_{V^p_{\alpha-\frac 1p}}+\|\tilde Y\|_{V^p_{\alpha-\frac 1p}})\|R^Y-R^{\tilde Y}\|_{V^p_{\alpha-\frac 1p}}]\\
&\leq C_{S, p, y_0}\|f\|_{C_b^3} (\|X\|_{V^p([0,\rho];\mathbb R^d)}\|Y-\tilde Y, Y'-\tilde Y'\|^{(\lambda)}_{\mathcal V^{\frac p2}_{X, \alpha}}+\|R^Y-R^{\tilde Y}\|_{V^p_{\alpha-\frac 1p}}).
\end{align*}
It is not hard to check that 
\begin{align*}
\|Z'-\tilde Z'\|_{V^\infty_{\alpha-\frac 1p}([0,\rho])}&=\sup_{0\leq t\leq \rho}|f(Y_t)-f(\tilde Y_t)|_{\alpha-\frac 1p}\\
&\leq C_{S, p}\|f\|_{C_b^3} \|R^Y-R^{\tilde Y}\|_{V^p_{\alpha-\frac 1p}([0,\rho])}. 
\end{align*}
Hence, we have 
\begin{equation}\label{Z'-tildeZ'}
\begin{split}
\|Z'-\tilde Z'\|_
{V^{p,\infty}_{\alpha-\frac2p,\alpha-\frac1p}}  &\leq C_{S, p, f, y_0}(\|X\|_{V^p([0,\rho];\mathbb R^d)}\|Y-\tilde Y, Y'-\tilde Y'\|^{(\lambda)}_{\mathcal V^{\frac p2}_{X, \alpha}}+\|R^Y-R^{\tilde Y}\|_{V^p_{\alpha-\frac 1p}})\\
&\leq C_{S, p, f, y_0} (\lambda^{-1}+\|X\|_{V^p([0,\rho];\mathbb R^d)})\|Y-\tilde Y, Y'-\tilde Y'\|^{(\lambda)}_{\mathcal V^{\frac p2}_{X, \alpha}}. 
\end{split}
\end{equation}
Combining \eqref{ZR} and \eqref{Z'-tildeZ'} we have 
\begin{align*}
\|Z-\tilde Z, Z'-\tilde Z'\|^{(\lambda)}_{\mathcal V^{\frac p2}_{X, \alpha}}&=
\|Z-\tilde Z\|_{V_\alpha^\infty}+\|Z'-\tilde Z'\|_{V^{p,\infty}_{\alpha-\frac 2p, \alpha-\frac 1p}}+\lambda\|R^Z-R^{\tilde Z}\|_{V^{\frac p2, p}_{\alpha-\frac 2p, \alpha-\frac 1p}}\\
&\leq
\lambda (
\|Z-\tilde Z\|_{V^\infty_\alpha}
+\|R^Z-R^{\tilde Z}\|_{V^{\frac p2,p}_{\alpha-\frac2p,\alpha-\frac1p}}
)+\|Z'-\tilde Z'\|_{V^{p,\infty}_{\alpha-\frac2p,\alpha-\frac1p}}\\
&\leq C_{S,p,f,y_0}
[
\lambda\|\mathbf X\|_{\mathbb V^p([0,\rho]; \mathbb R^d)}
+\lambda^{-1}+\|X\|_{V^p([0,\rho];\mathbb R^d)}
]
\|Y-\tilde Y,Y'-\tilde Y'\|^{(\lambda)}_{\mathcal V^{\frac p2}_{X,\alpha}}\\
&\leq C_{S,p,f,y_0}
(
\lambda^{-1}
+\lambda\|\mathbf X\|_{\mathbb V^p([0,\rho];\mathbb R^d)}
)
\|Y-\tilde Y,Y'-\tilde Y'\|^{(\lambda)}_{\mathcal V^{\frac p2}_{X,\alpha}}.
\end{align*}
Choose $\lambda$ large enough, such that  $C_{S, p, f, y_0}\lambda^{-1}\leq 1/4$ and take  sufficiently small $\rho$, such that 
$$C_{S, p, f, y_0}\lambda
\|\mathbf X\|_{\mathbb V^p([0,\rho];\mathbb R^d)}
\leq\frac14.$$
Thus, $\mathcal J$ is contractive on $\mathscr B_\rho(y_0)$.
Finally, there exists a unique solution $(Y,Y')\in \mathcal{V}^{\frac{p}{2}}_{X,\alpha}([0,\rho])$ by applying Banach's fixed point theorem. 
\end{proof}
 \begin{theorem}[Local existence and uniqueness of the mild solution to \eqref{main}]\label{local-rough-jump}
Let $p\in[2,3)$, $\alpha\in \mathbb{R}$, $\varepsilon>0$ and let $\mathbf X\in \mathbb{V}^p([0,T];\mathbb{R}^d)$. Suppose that  $f\in C_b^3(\cB_{\alpha-\frac{i}{p}}, \mathcal{L}(\mathbb{R}^d, \cB_{\alpha+\frac{2-i}{p}+\varepsilon}))$, $i=0,1,2$.  Then for every $y_0\in \cB_\alpha$,  there exists a time horizon $t_1$ and a unique local solution $(Y, Y')\in \mathcal{V}^{\frac p2}_{X, \alpha}([0,t_1])$ such that $Y'=f(Y)$ and 
\begin{align}\label{rpde1}
Y_t=S_{t,0}y_0+\int_0^{t}S_{t,r}f(Y_r)\cdot \textup{d}\mathbf{X}_r, \quad t\in[0,t_1].
\end{align}
\end{theorem}
\begin{proof}
Set $Y_{0+}:=Y_0+f(Y_0)\cdot \Delta_0^+ X+Df(Y_0)f(Y_0):\Delta_0^+ \mathbb X \in \cB_\alpha$ and   define 
\[
\tilde X_t:=\left\{
\begin{array}{ll}
	X_{0+}, & t=0 \\
	X_t,& t>0,
\end{array}
\right.
\ \ \
\tilde{\mathbb X}_{t,s}:=\left\{
\begin{array}{ll}
	\mathbb X_{t, 0+}, &  s=0, \\
	\mathbb X_{t,s},& s>0, 
\end{array}
\right.
\]
where $\mathbf X_{t,0}$ is right-continuous at $t=0$ and satisfies Chen's relation. According to  Lemma \ref{local-rough-pde}, there exists a unique solution to the following equation
\begin{align}\label{rpde2}
\tilde Y_t=S_{t,0}Y_{0+}+\int_0^tS_{t,r}f(\tilde Y_r)\cdot \textup{d}\tilde{\mathbf X}_r.
\end{align}
Following the similar arguments  in Remark \ref{localpp}, we claim that 
\[
Y_t:=\left\{
\begin{array}{ll}
	Y_0, & t=0 \\
	\tilde Y_t,& t>0
\end{array}
\right.
\]
solves  \eqref{rpde1}.
\end{proof}
\section{Examples}\label{sec:example}
\begin{example}\label{ex1}
Let $k\in \mathbb N$, $\alpha\in \mathbb R$, $\varepsilon >0$, $q\in (1,\infty)$ and let  $\mathbb T^d$  be $d$-dimensional torus. 
Consider the Banach scale $\cB_\gamma=H^{k+2\gamma, q}(\mathbb T^d)$, $\gamma \in \mathbb R$, we have 
$(\textup{id}-\Delta)^{-\beta}\in \mathcal L(\cB_\gamma, \cB_{\gamma +\beta})$, for any $\beta>0$. 

Let $p\in [2,3)$,  $g\in C^\infty(\mathbb T^d)$, $\mathbf X\in \mathbb V^p([0,T]; \mathbb R)$. Consider the rough PDE: 
\begin{equation}\label{example1}
\begin{cases}
&\textup{d}u_t(x)=\Delta u_t(x) \textup{d}t+ g(x)(\textup{id}-\Delta)^{-(\frac 2p+\varepsilon)}u_t(x) \textup{d}\mathbf X_t,\\
&u_0(x)\in \cB_\alpha.
\end{cases}
\end{equation}
Notice that  $(\textup{id}-\Delta)^{-(\frac 2p+\varepsilon)}:\cB_{\alpha-\frac ip} \to \cB_{\alpha+\frac{(2-i)}{p}+\varepsilon}$, $i=0,1,2$ and the multiplication with a
smooth  function  $g$ is a smooth function from $\cB_{\alpha+\frac 2p+\varepsilon}$ to itself.  
Theorem  \ref{local-rough-jump} therefore gives the existence and uniqueness of the local mild solution  to \eqref{example1}. 

As for Young case,  we can take $p\in [1,2)$, $\theta\in [\frac 1p, 1]$,  $X\in V^{p}([0,T];\mathbb R)$. 
Hence, by Theorem \ref{cadlagGLOBAL}, we obtain the well-posedness of
\begin{equation}\label{example11}
\begin{cases}
&\textup{d}u_t(x)=\Delta u_t(x)\textup{d}t+g(x)(\textup{id}-\Delta)^{-(2\theta+\varepsilon)}u_t(x) \textup{d} X_t,\\
&u_0(x)\in \cB_{\alpha+\theta}.
\end{cases}
\end{equation}
\end{example}
\begin{example}
Let $\mathcal O\subset\mathbb R^n$ be a bounded domain with smooth boundary, and let $A=\Delta_D$ denote the Dirichlet Laplacian on $L^2(\mathcal O)$. For $\gamma\geq0$, consider 
$
\mathcal B_\gamma=D((-A)^{\gamma})$, 
$|u|_\gamma=\|(-A)^{\gamma}u\|_{L^2(\mathcal O)}. 
$
Denote the boundary trace operator by $\gamma_0$. For integer $N\geq 1$, define $$
\cB_N=\{u\in H^{2N}(\mathcal O), \gamma_0(\Delta^k u)=0, k=0, \cdots, N-1\},  
$$
 and $|\cdot|_N$ is equivalent to the standard $H^{2N}(\mathcal O)$ norm on $\cB_N$,   see \cite[Definition 2.1]{Gru16} for more details.  In particular, 
$|u|_{\cB_N}=\|(-A)^Nu\|_{L^2(\mathcal O)}\leq C_{n, N} \sum\limits_{|k|\leq 2N}\|\partial_x^k u\|_{L^2(\mathcal O)}$.

Let $p\in[2,3)$, $\alpha\in \mathbb R$, $\varepsilon>0$, such that   $2\alpha-\frac4p>\frac n2$. Set $W:= \cB_{\alpha-\frac 2p}$. 
 By the Sobolev embedding theorem  \cite[Proposition 3.3]{Ta23}, we have $W \hookrightarrow L^{\infty}(\mathcal O) \cap  L^2(\mathcal O).$  Choose an integer
$
N\geq\alpha+\frac2p+\varepsilon.
$
For \(j=1,\ldots,d\), let  $g_j\in C^\infty (\bar{\mathcal O}\times \mathbb R)$ satisfy  $$\sup\limits_{(x,y)\in \bar{\mathcal O}\times \mathbb R}|\partial^k_x \partial^m_y g_j(x,y)|<\infty, \quad  |k|\leq 2N, \quad m=0,1,2,3$$ and $$\partial_x^k g_j(x, y)=0, \quad x\in \partial \mathcal O, \quad y\in \mathbb R, \quad |k|\leq 2N-2.$$  Similar to \cite[Example 7.1]{HBN19}, 
 we consider the following integral operator: 
$$
G(\phi)(\psi)[x]:= \sum_{j=1}^{d}\psi_j\int_{\mathcal O} g_j(x, \phi(\tilde x))\textup{d}\tilde x, \quad \phi\in W, \quad \psi\in \mathbb R^d. 
$$
We show that  
$G\in C_b^3( W,  \mathcal L(\mathbb R^d, \cB_{\alpha+\frac 2p+\varepsilon}))$.  First, the boundary assumptions on $g_j$ allows us to verify that   $G(\phi)(\psi)\in \cB_N$. Indeed, 
\begin{align*}
|G(\phi)(\psi)|_{\cB_N}^2&\leq C_{n,N}
\sum_{|k|\leq 2N}\int_{\mathcal O}|\sum_j^d \psi_j\int_{\mathcal O}\partial^k_xg_j(x, \phi(\tilde x))\textup{d}\tilde x|^2\textup{d}x\\
&\leq C_{n,N}
|\mathcal O|^3|\psi|_{\mathbb R^d}^2\sum_{
|k|\leq 2N}\sum_{j=1}^d \|\partial_x^kg_j\|^2_{\infty}.
\end{align*}
Hence, for every $\phi\in \cB_{\alpha-\frac 2p}$,  $G(\phi)\in \mathcal L(\mathbb R^d, \cB_{\alpha+\frac 2p+\varepsilon})$.

 According to
 \cite[Chapter XVII.3]{KA82} and the embedding $W\hookrightarrow L^\infty(\mathcal O)$,  $G$ is three times  continuously Fréchet-differentiable and we have 
\begin{align*}
[DG(\phi)[h_1](\psi)](x)
&=
\sum_{j=1}^d\psi_j
\int_{\mathcal O}
\partial_y g_j(x,\phi(\tilde x))
h_1(\tilde x)\,\textup d\tilde x,
\\
[D^2G(\phi)[h_1,h_2](\psi)](x)
&=
\sum_{j=1}^d\psi_j
\int_{\mathcal O}
\partial_y^2g_j(x,\phi(\tilde x))
h_1(\tilde x)h_2(\tilde x)\,\textup  d\tilde x,
\\
[D^3G(\phi)[h_1,h_2,h_3](\psi)](x)
&=
\sum_{j=1}^d\psi_j
\int_{\mathcal O}
\partial_y^3g_j(x,\phi(\tilde x))
h_1(\tilde x)h_2(\tilde x)h_3(\tilde x)\,\textup d\tilde x,
\end{align*}
for $h_1, h_2, h_3\in W.$
Let $\phi_1, \phi_2\in W$, we have 
\begin{align*}
\|G(\phi_1)-G(\phi_2)\|_
{\mathcal L(\mathbb R^d,\mathcal B_{\alpha+\frac2p+\varepsilon})}
&=\sup_{|\psi|_{\mathbb R^d}=1}
|G(\phi_1)(\psi)-G(\phi_2)(\psi)|_{\alpha+\frac2p+\varepsilon}\\
&\leq  C \sup_{|\psi|_{\mathbb R^d}=1}|G(\phi_1)(\psi)-G(\phi_2)(\psi)|_{\cB_N}\\
&\leq C \sup_{|\psi|_{\mathbb R^d}=1}(\sum_{|k|\leq 2N}\int_{\mathcal O}|\sum_{j=1}^d \psi_j\int_{\mathcal O} [\partial_x^k g_j(x, \phi_1(\tilde x))-\partial_x^kg_j(x, \phi_2(\tilde x))]\textup{d}\tilde x |^2\textup{d}x)^{\frac 12}\\
&\leq C( \sum_{j=1}^d \sum_{|k|\leq 2N} \int_{\mathcal O}|\int_{\mathcal O}[\partial_x^k g_j(x, \phi_1(\tilde x))-\partial_x^kg_j(x, \phi_2(\tilde x))]\textup{d}\tilde x|^2\textup{d}x)^{\frac 12}\\
&\leq C |\mathcal O|^{\frac 12} (\sum_{j=1}^d \sum_{|k|\leq 2N} \|\partial y\partial_x^k  g_j\|^2_{\infty})^{\frac 12}\int_{\mathcal O}|\phi_1(\tilde x)-\phi_2(\tilde x)|\textup{d}\tilde x\\
&\leq C |\mathcal O|^{\frac 12} (\sum_{j=1}^d \sum_{|k|\leq 2N} \|\partial y\partial_x^k   g_j\|^2_{\infty})^{\frac 12}\|\phi_1-\phi_2\|_{L^2}\\
&\leq 
C|\mathcal O|
(
\sum_{|k|\le2N}\sum_{j=1}^d
\|\partial y\partial_x^k g_j\|_\infty^2
)^{1/2}
|\phi_1-\phi_2|_{\alpha-\frac2p}, 
\end{align*}
where the last line we use the embedding $W\hookrightarrow L^2(\mathcal O)$.  Hence $G$ is Lipschitz. It remains to show the boundedness of $D^m G$, $m=1,2,3.$ 
\begin{align*}
&\left|D^mG(\phi)[h_1,\ldots,h_m](\psi)\right|_
{\alpha+\frac2p+\varepsilon}^{2}
\\
&\leq C|\psi|_{\mathbb R^d}^{2}
\sum_{|k|\leq2N}\sum_{j=1}^{d}
\int_{\mathcal O}
(
\int_{\mathcal O}
\left|\partial_x^k\partial_y^m g_j(x,\phi(\tilde x))\right|
\prod_{i=1}^{m}|h_i(\tilde x)|\,
\textup d\tilde x
)^{2}\textup dx
\\
&\leq C|\psi|_{\mathbb R^d}^{2}
\sum_{|k|\leq2N}\sum_{j=1}^{d}
\|\partial_x^k\partial_y^m g_j\|_\infty^{2}
\int_{\mathcal O}
(
|\mathcal O|
\prod_{i=1}^{m}\|h_i\|_{L^\infty(\mathcal O)}
)^{2}\textup dx
\\
&=
C|\mathcal O|^{3}|\psi|_{\mathbb R^d}^{2}
(
\sum_{|k|\leq2N}\sum_{j=1}^{d}
\|\partial_x^k\partial_y^m g_j\|_\infty^{2}
)
\prod_{i=1}^{m}\|h_i\|_{L^\infty(\mathcal O)}^{2}
\\
&\leq
C|\mathcal O|^{3}|\psi|_{\mathbb R^d}^{2}
(
\sum_{|k|\leq2N}\sum_{j=1}^{d}
\|\partial_x^k\partial_y^m g_j\|_\infty^{2}
)
\prod_{i=1}^{m}|h_i|_W^{2}.
\end{align*}
Denote  the Banach space of bounded $m$-linear maps from $W_1^m$ to $W_2$ by $\mathcal L^m(W_1;W_2)$, equipped with the norm
$
\|T\|_{\mathcal L^m(W_1;W_2)}
:=
\sup\limits_{\substack{\|h_i\|_{W_1}=1\\ i=1,\cdots,m}}
\|T[h_1,\ldots,h_m]\|_{W_2}.
$
Now, we have 
\begin{align*}
\|D^mG(\phi)\|_
{\mathcal L^m(W;\mathcal L(\mathbb R^d,\mathcal B_{\alpha+\frac2p+\varepsilon}))}&= \sup\limits_{_{\substack{\|h_i\|_{W}=1\\ i=1,\cdots,m}}}\sup\limits_{|\psi|_{\mathbb R^d}=1} |D^m G(\phi)[h_1, \cdots, h_m](\psi)|_{\alpha+\frac 2p+\varepsilon} \\
&\leq 
C|\mathcal O|^{3/2}
(
\sum_{|k|\leq2N}\sum_{j=1}^{d}
\|\partial_x^k\partial_y^m g_j\|_\infty^{2}
)^{1/2}.
\end{align*}
Therefore, $G\in C_b^3(\cB_{\alpha-\frac 2p}, \mathcal L(\mathbb R^d, \cB_{\alpha+\frac 2p+\varepsilon}))$, which implies $G\in C_b^3(\cB_{\alpha-\frac ip}, \mathcal L(\mathbb R^d, \cB_{\alpha+\frac {2-i}{p}+\varepsilon}))$, $i=0,1,2$.  
Let $\mathbf X\in \mathbb V^p([0,T];\mathbb R^d)$. According to  Theorem  \ref{local-rough-jump}, we have the local existence and uniqueness of the  mild solution to the following rough PDE: 
\begin{equation}
\begin{cases}
&\textup du_t(x)
=
\Delta u_t(x)\,\textup dt
+\sum\limits_{j=1}^d \int_{\mathcal{O}} g_j (x,u_t(z)) ~\textup d z~ \textup d\mathbf X_t^j,\\
&\gamma_0 u_t=0,\\
&u_0\in\mathcal B_\alpha.
\end{cases}
\end{equation}
\end{example}

{\bf AI Disclosure Statement.} The authors used AI solely to assist with text formulation, phrasing, and grammatical refinement to improve the readability of the manuscript. All  mathematical proofs, scientific content, and conceptual ideas were entirely self-developed by the authors. The authors have reviewed, edited, and take full responsibility for the final version of this work.
\\

{\bf Acknowledgements.}~A. Blessing acknowledges funding by the DFG CRC/TRR 388 "Rough Analysis, Stochastic Dynamics and Related Fields" - Project ID 516748464.~H. Gao acknowledges NSFC Grant No. 12571164  and the Jiangsu Provincial Scientific Research Center of Applied Mathematics under Grant No. BK20233002. H. Li acknowledges the China Scholarship Council.
\end{document}